\documentclass[12pt]{article}
\usepackage{amsmath,amsthm,verbatim,amssymb,amsfonts,amscd,diagrams, graphics, mathrsfs,hyperref,soul}
\usepackage[usenames,dvipsnames]{color}
\newcommand{\aug}{\mathbf{a}}

\usepackage{color}

\newcommand\red[1]{\textcolor{black}{#1}}
\newcommand\green[1]{\textcolor{black}{#1}}

\theoremstyle{plain}
\newtheorem{theorem}{Theorem}[section]
\newtheorem*{intro-theorem}{Theorem}
\newtheorem{corollary}[theorem]{Corollary}
\newtheorem*{intro-corollary}{Corollary}
\newtheorem{lemma}[theorem]{Lemma}
\newtheorem{proposition}[theorem]{Proposition}

\theoremstyle{definition}
\newtheorem{definition}[theorem]{Definition}

\theoremstyle{remark}
\newtheorem{remark}[theorem]{Remark}

\newarrow{ul}---->
\newarrow{Backwards}<----   
\newarrow{Corresponds}<--->

\newcommand{\xr}{\xrightarrow}
\newcommand{\onto}{\twoheadrightarrow}

\newcommand{\td}[1]{\tilde{#1}}
\newcommand{\into}{\hookrightarrow}
\newcommand{\Z}{\mathbb{Z}}
\newcommand{\Q}{\mathbb{Q}}
\newcommand{\R}{\mathbb{R}}

\newcommand{\D}{\mathbb{D}}
\renewcommand{\O}{\mathbb{O}}

\newcommand{\bd}{\partial}
\newcommand{\pf}{\pitchfork}

\renewcommand{\H}{\mathbb H}

\newcommand{\mc}[1]{\mathcal{#1}}
\newcommand{\ms}[1]{\mathscr{#1}}

\newcommand{\dlim}{\varinjlim}

\newcommand{\Hom}{\operatorname{Hom}}
\newcommand{\Mor}{\operatorname{Mor}}
\newcommand{\SHom}{\text{\emph{Hom}}}
\newcommand{\codim}{\text{codim}}
\newcommand{\Ext}{\text{Ext}}
\newcommand{\mf}{\mathfrak}

\newcommand{\im}{\text{im}}

\newarrow{Onto}----{>>}
\newarrow{Equals}=====
\newarrow{Into}C--->

\begin{document}

\title{Intersection homology duality and pairings: singular, PL, and sheaf-theoretic}
\author{Greg Friedman and James E. McClure}
\date{\red{July 22, 2020}}

\maketitle
\tableofcontents



\begin{abstract}
We compare the sheaf-theoretic and singular chain versions of Poincar\'e duality for intersection homology, showing that they are isomorphic via naturally defined maps. Similarly, we demonstrate the existence of canonical isomorphisms between the singular intersection cohomology cup product, the hypercohomology product induced by the Goresky-MacPherson sheaf pairing, and, for PL pseudomanifolds, the Goresky-MacPherson PL intersection product. We also show that the de Rham isomorphism of Brasselet, Hector, and Saralegi preserves product structures. 
\end{abstract}

\section{Introduction}
\red{In its modern incarnation, Poincar\'e duality for a compact oriented $n$-manifold $M$ most often takes the form of an isomorphism $$H^{i}(M;R)\xr{\cap \Gamma}H_{n-i}(M;R),$$
where $R$ is a ring of coefficients and $\cap \Gamma$ is the cap product with a fundamental class $\Gamma\in H_n(M;R)$; see for example \cite[Theorem 3.30]{Ha}. If we use coefficients in a field $F$ for simplicity, then Poincar\'e duality together with the Universal Coefficient Theorem $H_{n-i}(M;F)\cong \Hom(H^{n-i}(M;F),F)$ provides an isomorphism $H^i(M;F)\cong \Hom(H^{n-i}(M;F),F)$, and it is well known that the resulting nonsingular pairing $$H^i(M;F)\otimes H^{n-i}(M;F)\to F$$ corresponds to evaluating the output of the cup product  $H^i(M;F)\otimes H^{n-i}(M;F)\xr{\cup}H^n(M;F)$ on the fundamental class $\Gamma$; again see \cite[Chapter 3]{Ha}. Dually one obtains a nonsingular pairing $H_{n-i}(M;F)\otimes H_i(M;F)\to F$ referred to as the \emph{intersection pairing}. Dold essentially defines the intersection pairing to be dual to the cup product pairing \cite[Formula VIII.13.5]{Dold}, though, at least in the piecewise linear (PL) setting, this pairing is induced at the chain level by a partially-defined operation of geometric intersection \cite{McC} (in fact, this geometric intuition came historically first \cite{Dieudonne}).}  

\red{These pairings and isomorphisms can also be approached via sheaf theory. If $\mc F$ is the constant sheaf on $M$ with stalk $F$, then the sheaf cohomology $H^i(M;\mc F)$ is isomorphic to the singular cohomology $H^i(M;F)$ \cite[Theorem III.1.1]{BR}, and there is a cohomology pairing induced by the sheaf product $\mc F\otimes \mc F\to \mc F$ that is given by stalkwise multiplication. Sheaf-theoretic Poincar\'e duality most often arises in modern textbook sources as a corollary of Verdier duality, which in this context says that the sheaf $\mc F$, treated as a sheaf complex concentrated in degree $0$, is quasi-isomorphic to the shifted Verdier dual $\mc D\mc F[-n]$. As the Verdier dualizing functor $\mc D$ admits a form of the Universal Coefficient Theorem---namely for any complex $\mc S^*$ we have that the hypercohomology $\H^i(M;\mc D \mc S^*)$ is isomorphic to $\Hom(\H^{-i}(M;\mc S^*),F)$ for $M$ compact---we have together that 
$$\H^i(M;\mc F)\cong \H^i(M;\mc D\mc F[-n]) \cong \H^{i-n}(M;\mc D\mc F)\cong \Hom(\H^{n-i}(M;\mc F),F).$$
In sources such as \cite[Theorem 3.3.1]{DI04} this is declared to be Poincar\'e duality. While it seems to be folklore that these sheaf-theoretic tools are equivalent to those from singular homology and cohomology, this is in fact a delicate issue, and one of our aims in this paper is to verify this expectation; we will see that it takes some work. Furthermore, we will show that deriving Poincar\'e duality sheaf-theoretically does not require Verdier duality; in fact Poincar\'e duality turns out to be due to the surprisingly simple observation that both the singular chain and cochain complexes sheafify in an appropriate manner to resolutions of the constant sheaf. Relating this observation to the singular cap product is a more complex matter.}

\red{These folklore results about manifolds will in fact be outcomes of our more general study of products and duality for intersection homology and intersection cohomology on pseudomanifolds. To explain, let us review some of that context.}

\red{While singular homology cup and cap products can be defined on any space, Poincar\'e duality will not hold in general for spaces that have singularities, i.e.\ points with non-Euclidean neighborhoods. However,
in \cite{GM1}, Goresky and MacPherson introduced intersection homology for compact oriented piecewise linear (PL) stratified pseudomanifolds, extending Poincar\'e duality to such spaces, which we recall in detail in Section \ref{S: ihreview}.} Their duality result involved constructing a geometric PL intersection product for PL chains in a suitable version of general position, which induces a pairing on intersection homology\footnote{The symbol $\pf$ is used in \cite{GM1} for a different purpose, but we will use it to denote the intersection pairing.}
$$I^{\bar p}H_i(X)\otimes I^{\bar q}H_j(X)\xr{\pf} I^{\bar r}H_{i+j-n}(X).$$
Here $X$ is an $n$-dimensional compact oriented PL stratified pseudomanifold and $\bar p,\bar q,\bar r$ are perversity parameters such that $\bar p+\bar q\leq \bar r$. Furthermore, they showed that if $i+j=n$ and $\bar q=D\bar p$, the complementary perversity to $\bar p$, then after tensoring with $\Q$ we obtain  a \emph{nonsingular} pairing
$$I^{\bar p}H_i(X;\Q)\otimes I^{D\bar p}H_{n-i}(X;\Q)\to I^{\bar t}H_{0}(X;\Q)\to\Q,$$
with $\bar t$ being the distinguished \emph{top perversity.}
\red{If $X$ is in fact a PL manifold, this is the intersection product mentioned above. Beyond the enormous importance of this version of duality, Goresky and MacPherson's manifestation of the intersection product was itself an innovation, utilizing the PL structure rather than fixed triangulations and employing the cup product to define intersections. This was the approach to intersection later used by the second-named author in \cite{McC} to define intersection products on PL manifolds in terms of chain maps (as opposed to only for chains in general position), which was then extended to intersections chains on PL pseudomanifolds by the first-named author in \cite{GBF18, GBF39}. }

In their follow-up \cite{GM2}, Goresky and MacPherson recast intersection homology in the derived category of sheaf complexes. This allowed an extension to \emph{topological} stratified pseudomanifolds, eliminating the requirement of piecewise linear structures, as well as carrying other advantages due to the powerful abstract machinery. For example, the sheaf complex whose hypercohomology gives intersection homology is determined up to isomorphism in the derived category by a simple set of axioms. Letting $\mc P_{\bar p}^*$ denote\footnote{Note the shift to cohomological indexing. There are a variety of further indexing conventions; we choose the one from \cite{Bo} such that $\H^i(X;\mc P_{\bar p}^*)\cong I^{\bar p}H_{n-i}(X;F)$.} such a sheaf complex for the perversity $\bar p$ and assuming coefficients in a field $F$, the duality itself takes the form $ \mc P_{\bar p}^*\sim (\mc D\mc P_{D\bar p}^*)[-n]$, where $\mc D$ is the Verdier dualizing functor,  $[-n]$ denotes a degree shift, and $\sim$ denotes quasi-isomorphism of sheaf complexes or, equivalently, isomorphism in the derived category. Applying hypercohomology and properties of $\mc D$, this quasi-isomorphism implies that $\H^i(X;\mc P_{\bar p}^*)\cong \Hom(\H^{n-i}(X;\mc P_{D\bar p}^*),F)$ on a compact pseudomanifold, which translates to $I^{\bar p}H_{n-i}(X;F)\cong \Hom(I^{D\bar p}H_{i}(X;F),F)$.
 We also have sheaf complex  pairings,  induced now  by morphisms of the form $\mc P_{\bar p}^*\otimes \mc P_{\bar q}^*\to \mc P_{\bar r}^*$. 
By \cite[Section 5.2]{GM2} (see also \cite[Section 9.C]{Bo}), such pairings turn out to be completely characterized by how they behave as maps of the cohomology stalks $\mc H^0(\mc P_{\bar p}^*)_x\otimes \mc H^0(\mc P_{\bar q}^*)_x\to \mc H^0(\mc P_{\bar r}^*)_x$ at the nonsingular points $x$, i.e.\ those with Euclidean neighborhoods. Up to isomorphism, each such pointwise pairing reduces to a map $F\otimes F\to F$, and the pairing corresponding to simple multiplication at each  $x$ we refer to as the \emph{Goresky-MacPherson sheaf pairing}\footnote{This pairing may differ from the original Goresky-MacPherson sheaf-theoretic pairing of \cite[Section 5.2]{GM2} by some signs, as we use different indexing conventions.}.

\red{Following the work of Goresky-MacPherson, a singular chain approach to intersection homology was developed by King \cite{Ki}, and even more recently intersection (co)homology versions of the cup and cap products have been developed by the authors using a variant of King's singular chains \cite{GBF25, GBF35}. 
In particular, there is a Poincar\'e duality isomorphism of the form}\footnote{We include a sign so that this isomorphism will be induced by a degree $-n$ chain map; see \cite[Remark 8.2.2]{GBF35}.}:\red{
$$I_{\bar p}H^i(X;F)\xr{(-1)^{in}\cdot \cap \Gamma} I^{D\bar p}H_{n-i}(X;F)$$
and a cup product pairing  $$I_{\bar p}H^*(X;F)\otimes I_{\bar q}H^*(X;F)\to I_{\bar r}H^*(X;F)$$ for appropriate $\bar p,\bar q,\bar r$. Note that here $I_{\bar p}H^*(X;F)$ is the cohomology of the dual complex of a complex of  singular intersection chains: $I_{\bar p}H^*(X;F)=H^*(\Hom(I^{\bar p}S_*(X;F),F))$.}

\red{As in the manifold case, the question naturally arises of how all of these pairings and duality isomorphisms are related to each other. Again they are generally assumed to be equivalent, though this is far from obvious. To emphasize again the subtleties even for manifolds, we note again that many expository resources for sheaf theory,} such as \cite{KS, IV, DI04},
first prove various forms of Verdier duality and then derive a version of Poincar\'e duality as a corollary; see e.g.\ \cite[Theorem 3.3.1]{DI04}. 
But  the relationship between these isomorphisms arising from Verdier duality and the duality isomorphism given by the cap product with the fundamental class is not so evident. On the other hand, there are some discussions of sheaf-theoretic cap products in the literature, e.g.\ in \cite{BR} and \cite[Section 10]{Sk94}, but it is not  clear how they relate to Verdier duality or that they extend to the intersection homology setting\footnote{See Remark \ref{R: intro cap} below for more about these.}.

Our goal then is to prove that the various pairings and duality maps \emph{are} isomorphic, at least in the case of intersection (co)homology of compact oriented stratified pseudomanifolds and with field coefficients. This includes ordinary (co)homology on compact oriented manifolds as a special case.  Furthermore, we show that the isomorphisms are all induced by appropriately canonical maps. 

We next outline the paper and our main results:

\paragraph{Preliminary material.} Section \ref{S: background} contains some conventions and a very brief review of  background material with references to more detailed sources.  Section \ref{S: hyper} contains some preliminary material about hypercohomology when working with sheaf complexes in the derived category $D(X)$. In particular, we consider conditions under which hypercohomology groups and maps between hypercohomology groups can be realized concretely by the cohomology of global sections of specific sheaf complexes and the maps between them. 

\paragraph{Sheaf complexes.} Section \ref{S: particular} introduces our main sheaf complexes of interest. We begin with brief reviews of the Verdier dualizing complex $\mathbb D^*$ and the Deligne sheaf complexes  $\mc P^*_{\bar p}$, the latter of which are determined up to quasi-isomorphism by the aforementioned Goresky-MacPherson axioms from \cite{GM2}. In Section \ref{S: SIC}, we recall the singular intersection chain sheaf complexes $\mc I^{\bar p}\mc S^*$ of \cite{GBF10, GBF23}, which are quasi-isomorphic to the Deligne sheaf complexes $\mc P_{\bar p}^*$ but are defined concretely by sheafifying the presheaf complexes of singular intersection chains $U\to I^{\bar p}S_{n-*}(X,X-\bar U;F)$. The quasi-isomorphism tells us that 
$\mc I^{\bar p}\mc S^*$ can be thought of as a specific object representing the abstract isomorphism class of $\mc P_{\bar p}^*$ in the derived category. In particular, its hypercohomology is isomorphic to the intersection homology $\H^i(X;\mc I^{\bar p}\mc S^*)\cong I^{\bar p}H_{n-i}(X;F)$. But since $\mc I^{\bar p}\mc S^*$ is homotopically fine, we can realize its hypercohomology in terms of the cohomology of the complex of global sections; this provides a way to make the isomorphism  $\H^i(X;\mc I^{\bar p}\mc S^*)\cong I^{\bar p}H_{n-i}(X;F)$ canonical, which will be important in what follows.

 In Section \ref{S: sheaf}, we introduce the sheaf complexes of intersection cochains $\mc I_{\bar p}\mc C^*$ as the sheafifications of the presheaf complexes of singular intersection cochains $U\to I_{\bar p}S^*(U;F)$, and we study its properties. Its hypercohomology is isomorphic to the intersection cohomology of \cite{GBF25}, i.e.\ $\H^i(X;\mc I_{\bar p}\mc C^*)\cong I_{\bar p}H^i(X;F)$.
Furthermore, the complex $\mc I_{\bar p}\mc C^*$ is flabby, and so it too canonically represents its hypercohomology by global sections.  We also show that $\mc I_{\bar p}\mc C^*$ is quasi-isomorphic to $\mc P_{D\bar p}^*$ and so to $\mc I^{D\bar p}\mc S^*$, providing another representative of the same isomorphism class in the derived category. This turns out to be the heart of intersection homology Poincar\'e duality, and we see that it does not explicitly require Verdier duality.

\paragraph{The cup product.} Before getting to duality results in detail, we first sheafify the cup product in Section \ref{S: sheaf cup}; this provides our first compatibility result, between the cup product and the Goresky-MacPherson sheaf product.
We state this theorem here with field coefficients for simplicity though a more general case is considered as Theorem \ref{S: sheaf cup}:

\begin{theorem}
Suppose $D\bar r\geq D\bar p+D\bar q$. Then there is a commutative diagram 
\begin{diagram}[LaTeXeqno]
I_{\bar p}H^*(X;F)\otimes I_{\bar q}H^*(X;F)&&\rTo^\cup && I_{\bar r}H^*(X;F)\\
\dTo^\cong&&&&\dTo_\cong\\
\H^*(X;\mc I_{\bar p}\mc C^*)\otimes \H^*(X;\mc I_{\bar q}\mc C^*)&\rTo  &
\H^*(X;\mc I_{\bar p}\mc C^*\otimes \mc I_{\bar q}\mc C^*)&\rTo^{\td \cup}&   \H^*(X;\mc  I_{\bar r}\mc C^*)
\end{diagram}
in which the vertical maps are isomorphisms induced by sheafification and $\td \cup$ is the Goresky-MacPherson sheaf product. 
\end{theorem}

As an application, we show that the intersection de Rham theorems of 
 Brasselet-Hector-Saralegi \cite{BHS} and Saralegi \cite{Sa05} are multiplicative (see Section \ref{S: mult} for further details):
 
\begin{intro-theorem}[Theorem \ref{T: de Rham}]
Let $X$ be an $\R$-oriented unfoldable stratified pseudomanifold, and suppose $D\bar r\geq D\bar p+D\bar q$.
 Let $\Omega_{\bar s}^*(X)$ be the complex of $\bar s$-perverse  differential forms on $X$. Then the following diagram commutes
\begin{diagram}[LaTeXeqno]
H^*(\Omega_{D\bar p}^*(X))\otimes H^*(\Omega_{D\bar q}^*(X))&\rTo^\wedge&H^*(\Omega_{D\bar r}^*(X))\\
\dTo^{\int\otimes \int}_\cong&&\dTo^\cong_\int \\
I_{\bar p}H^*(X;\R)\otimes I_{\bar q}H^*(X;\R)&\rTo^\cup&I_{\bar r}H^*(X;\R) .
\end{diagram}
\end{intro-theorem}

\green{Similarly, we will also show in  Section \ref{S: mult} that the singular cochain intersection cohomology cup product is isomorphic to the blown-up intersection cohomology cup product of Chataur, Saralegi-Aranguren, and Tanr\'e \cite{CST7} when they are both defined.}

\paragraph{Poincar\'e duality.}
We next turn to Poincar\'e duality. In Section \ref{S: compatibility} we prove the following:

\begin{intro-theorem}[Theorem \ref{T: MAIN}]
Let $X$ be a compact $F$-oriented $n$-dimensional stratified pseudomanifold, and let $\bar p,D\bar p$ be complementary perversities. The following diagram of isomorphisms commutes, where the vertical maps are induced by the sheafification of presheaf sections into sheaf sections,  the bottom map is induced by the quasi-isomorphism consistent with the orientation, and  the top map is the Poincar\'e duality map given by the signed cap product with the fundamental class $\Gamma$ determined by the orientation:

\begin{diagram}[LaTeXeqno]
I_{\bar p}H^i(X; F)&\rTo^{(-1)^{in}\cdot\cap \Gamma}& I^{D\bar p}H_{n-i}(X; F)\\
\dTo^\sigma&&\dTo_{\sigma'}\\
\H^i(X;\mc I_{\bar p}\mc C^*) &\rTo^{\mathbb{O}} &  \H^i(X;\mc I^{D\bar p}\mc S^*).
\end{diagram}
\end{intro-theorem}

Here the quasi-isomorphism $\mathbb O$ ``consistent with the orientation'' is the unique morphism $\mathbb{O}\in \Mor_{D(x)}(\mc I_{\bar p}\mc C^*,\mc I^{D\bar p}\mc S^*)$ that on cohomology stalks at nonsingular points $x$ takes the generator of $H^0(\mc I_{\bar p}\mc C^*_x)$ represented locally by the cocycle $1$ to the generator of $H^0(\mc I^{D\bar p}\mc S^*_x)$ represented locally by the local orientation class.  See Section \ref{S: O} for more details.

When $X$ is a manifold, $\mc I_{\bar p}\mc C^*$ and $\mc I^{D\bar p}\mc S^*$ reduce to the sheaf complexes $\mc C^*$ and $\mc S^*$ of ordinary singular cochains and chains, and we obtain the following corollary:

\begin{intro-corollary}[Corollary \ref{C: man duality}]
Let $M$ be a compact $F$-oriented $n$-dimensional manifold. The following diagram of isomorphisms commutes:

\begin{diagram}
H^i(M; F)&\rTo^{(-1)^{in}\cdot\cap \Gamma}& H_{n-i}(M; F)\\
\dTo^\sigma&&\dTo_{\sigma'}\\
\H^i(M;\mc C^*) &\rTo^{\mathbb{O}} &  \H^i(M;\mc S^*).
\end{diagram}
\end{intro-corollary}

\begin{remark}\label{R: intro cap}
Once again we emphasize that Poincar\'e duality from the sheaf-theoretic point of view does not require Verdier duality; in fact it is a manifestation of the existence on oriented manifolds of a canonical quasi-isomorphism between sheaves of cochains and chains, each of which is a resolution of the constant sheaf with stalk $F$. 
For pseudomanifolds, Poincar\'e duality results from the existence of such a quasi-isomorphism over the regular strata (the complement of the singular set). 

The part that requires effort is proving that the hypercohomology maps induced by these quasi-isomorphisms are compatible with a map that comes from an essentially global chain level construction---the cap product with the fundamental class. In fact, it is far from clear that the chain theoretic cap product can be realized as a map of global sections induced by a sheaf map. The authors have thus far failed in their attempts to construct it in this fashion.

There is a thorough treatment of a sheaf-theoretic cap product for locally compact spaces in Bredon \cite{BR}, though as observed by Sklyarenko \cite[Section 10]{Sk94}, ``we have not succeeded in establishing some bridge between the $\cap$-products corresponding to these  theories [Bredon's and the classical singular theory] using even one of the constructions available [in \cite{BR}].'' Sklyarenko then goes on to show that another quite general sheaf cap product of his construction on locally compact spaces agrees with the singular chain cap product in a number of important cases. The proof is very nontrivial. 

Unfortunately, Sklyarenko's cap product does not quite yield the intersection homology results we are after since although his cap product takes quite general hyperhomology theories as one of its inputs, the other input is always cohomology with coefficients in a single sheaf. This is not quite the right set-up for studying intersection homology and cohomology, which sheaf theoretically are both the (co)homology of sheaf complexes that are not resolutions of a single sheaf. Nonetheless, our goals are also much more modest than those of Sklyarenko.  We will limit ourselves to compact topological pseudomanifolds (which includes manifolds), and we will only look at intersection homology theories (which includes singular homology on manifolds). We will  not construct a sheaf cap product, per se, but rather we will look at the  canonical quasi-isomorphism $\mc I_{\bar p}\mc C^*\to \mc I^{D\bar p}\mc S^*$ determined by the orientation of $X$, and we will deduce, more or less by hand, that the hypercohomology isomorphisms induced by these quasi-isomorphisms are compatible with the singular intersection chain cap product with the fundamental class. When the underlying space is a compact manifold oriented over a field, this gives a possibly new proof that the canonical map in the derived category from the sheaf complex of singular cochains to the sheaf complex of singular chains  induces a map in hypercohomology consistent with the singular chain cap product with the fundamental class.

\end{remark}

\paragraph{Comparing products: cup, intersection, and sheaf-theoretic.}
With Theorem \ref{T: MAIN} in hand, we can turn to the compatibility of various pairings in Section \ref{S: cap/cup}. These results can best be summarized in terms of the commutativity of the double cube diagram below. We explain all of the groups and maps in the diagram following the statement of the theorem.

\begin{theorem}\label{T: cubes}
Let $F$ be a field, $X$ an $n$-dimensional compact $F$-oriented PL stratified pseudomanifold, and $\bar p,\bar q,\bar r$  perversities such that $D\bar r\geq D\bar p+D\bar q$. Then all of the squares in the following diagram commute except for the middle and bottom horizontal squares and the top left square, all three of which commute up to $(-1)^n$. Furthermore, all vertical and back-to-front maps are isomorphisms. If $X$ is only a topological stratified pseudomanifold, then the bottom square is not defined, but the signed commutativity of the top cube continues to hold. 
\begin{equation}
\resizebox{.95\hsize}{!}{
\begin{diagram}
  I_{\bar p}H^i(X)\otimes I_{\bar q}H^j(X)  &    &\rTo^{\cup} &      &    I_{\bar r}H^{i+j}(X)  \\
      & \rdTo_{\sigma\otimes \sigma} &      &      & \vLine& \rdTo^{\sigma}  \\
\dTo^{(-1)^{in}\cdot \cap \Gamma\otimes (-1)^{jn}\cdot \cap \Gamma} &    &   \H^{i} (X;\mc I_{\bar p}\mc C^*)\otimes\H^j(X; \mc I_{\bar q}\mc C^*)   & \rTo^{\td \cup}  & \HonV   &    &  \H^{i+j}( X; \mc I_{\bar r}\mc C^*) \\
      &    & \dTo^{\mathbb{O}\otimes \mathbb{O}}  &      & \dTo^{(-1)^{(i+j)n}\cdot \cap \Gamma }  \\
I^{D\bar p}H_{n-i}(X)\otimes I^{D\bar q}H_{n-j}(X) &\hLine & \VonH   & \rTo^{\psi} &   I^{D\bar r}H_{n-i-j}(X)&    & \dTo_{\mathbb{O}} \\
      & \rdTo_{\sigma'\otimes \sigma'} &      &      & \uTo& \rdTo^{\sigma'}  \\  
\dCorresponds &    &   H^{i} (X;\mc I^{D\bar p}\mc S^*)\otimes \H^j(X;\mc I^{D\bar q}\mc S^*)   & \rTo^{\td \psi}  & \HonV   &    &  \H^{i+j}( X; \mc I^{D\bar r}\mc S^*) \\
      &    & \dCorresponds &      & \dTo  \\
   I^{D\bar p}\mf H_{n-i}(X)\otimes I^{D\bar q}\mf H_{n-j}(X) & \hLine & \VonH   & \rTo^{\pf} &  I^{D\bar r}\mf H_{n-i-j}(X)  &    & \dCorresponds \\
      & \rdTo_{\sigma''\otimes\sigma''} &      &      &      & \rdTo^{\sigma''}  \\
      &    &   \H^{i} (X;\mc I^{D\bar p}\mc S_{PL}^*)\otimes \H^j(X;\mc I^{D\bar q}\mc S_{PL}^*)   &      & \rTo^{\td\pf} &    &   \H^{i+j}( X; \mc I^{D\bar r}\mc S_{PL}^*) .      
\end{diagram}}
\end{equation}
\end{theorem}

Here, using $\bar s$ as a placeholder perversity,   
$I_{\bar s}H^i(X)$, $I^{\bar s}H_{n-i}(X)$, and $I^{\bar s}\mf H_{n-i}(X)$ are respectively singular intersection cohomology, singular intersection homology, and PL intersection homology (with coefficients in $F$), and  $\mc I_{\bar s}\mc C^*$,  $\mc I^{\bar s}\mc S^*$, and $\mc I^{\bar s}\mc S^*_{PL}$ are respectively the sheaf complexes of singular intersection cochains, singular intersection chains, and PL intersection chains. Furthermore, the maps $\sigma$, $\sigma'$, and $\sigma''$ are all isomorphisms induced by sheafification. The maps $\td \cup$, $\td \psi$, and $\td \pf$ are all representatives of the Goresky-MacPherson sheaf product (with respect to different complexes isomorphic to the Deligne sheaf complexes), while 
 $\cup$, $\pf$, and $\psi$  are the cup product, the Goresky-MacPherson PL intersection product (up to sign conventions), and the ``singular chain intersection product,'' which we take to be defined by the sheaf product $\td \psi$ and the commutativity of the middle horizontal square up to $(-1)^n$. 
For the vertical maps, $\cap \Gamma$ is the  cap product with the fundamental class and $\mathbb{O}$ is the aforementioned quasi-isomorphism consistent with the local orientation. The vertical double-headed arrows in the back of the bottom cube represent chains of isomorphisms between the singular and PL intersection homology groups (see \cite[Section 5.4]{GBF35}), and those in the front come from their sheafifications. These last sheafifications are again compatible with the orientations in the sense that over each $x\in X-X^{n-1}$ they map the local singular homology orientation class to the local PL homology orientation class. 

\paragraph{Verdier duality.} Finally, in Section \ref{S: Verdier} we demonstrate that there is indeed a relationship between the Poincar\'e duality of singular intersection (co)homology and the Goresky-MacPherson sheaf-theoretic duality in terms of Verdier duals. In fact we provide two statements, one using the sheaves of intersection chains and one using the sheaves of intersection cochains. The maps involved are defined in detail in Section \ref{S: Verdier}. 

\begin{theorem}\label{T: Verdier}
Let $X$ be a compact $F$-oriented $n$-dimensional stratified pseudomanifold and $\bar p$ a perversity.
Then there are  commutative diagrams of isomorphisms
\begin{diagram}
I_{\bar p}H^i(X; F)&&\rTo^{(-1)^{in}\cdot\cap \Gamma} &&I^{D\bar p}H_{n-i}(X; F)\\
\dTo^{\sigma}&&&&\dTo_{\kappa}\\
\H^i(X;\mc I_{\bar p}\mc C^*)&\rTo^{\hat \Phi} &\H^{i}(X;\mc D\mc I_{D\bar p}\mc C^*[-n])&\rTo^{\sigma^*e}&\Hom(I_{D\bar p}H^{n-i}(X;F),F)\\
 I^{D\bar p}H_{n-i}(X) & &\lTo^{(-1)^{in}\cdot\cap\Gamma} & &I_{\bar p}H^i(X)    \\
 \dTo^{\sigma'}  & &&  &  \dTo_{\kappa'}  \\
\H^{i} (X;\mc I^{D\bar p}\mc S^*)  &  \rTo^{\hat{\Psi}}  &       \H^{i} (X;\mc D\mc I^{\bar p}\mc S^*[-n])    &      \rTo^{{\sigma'}^*e} &     \Hom(I^{\bar p}H_{i}(X),F) ,\\
\end{diagram}
in which the top maps are signed cap products with the fundamental class,  $\sigma$ and $\sigma'$ are sheafifications, $\sigma^*$ and  
${\sigma'}^*$ are the $\Hom$ duals of the sheafifications, $e$ is the  universal coefficient isomorphism for Verdier duals,
$\kappa$ and $\kappa'$ are Kronecker evaluation maps\footnotemark, and $\hat \Phi$ and $\hat \Psi$ are induced by  quasi-isomorphisms of sheaf complexes (in particular they are respective adjoints of the sheaf theoretic cup and intersection products composed with canonical maps to the Verdier dualizing complex).  
\end{theorem}

\footnotetext{\label{pageref kappa}If  $x\in I^{\bar p}H_{i}(X; F)$ and $\alpha\in I_{\bar p}H^{i}(X;F)$, then $\kappa(x)(\alpha)=(-1)^{i}\alpha(x)$ and $\kappa'(\alpha)(x)=\alpha(x)$. See Appendix \ref{Appendix}.}

As a corollary, we have the following statement in the manifold case:

\begin{corollary}\label{C: Verdier}
Let $M$ be a compact $F$-oriented $n$-manifold, let $\mc F$ be the constant sheaf with stalks $F$, and let $\D^*$ be the Verdier dualizing complex. Then there is a commutative diagram of isomorphisms
\begin{diagram}
H^i(M; F)&&\rTo^{(-1)^{in}\cdot\cap \Gamma} &&H_{n-i}(M; F)\\
\dTo^{\sigma}&&&&\dTo_{\kappa}\\
\H^i(M;\mc F)&\rTo^{\hat \Phi} &\H^{i}(M;\D^*[-n])&\rTo^e&\Hom(H^{n-i}(M;F),F).
\end{diagram}
\end{corollary}

\paragraph{Appendices.} We conclude with two appendices. The first contains some basic material about ``double duals.'' The second demonstrates that certain sign difficulties are unavoidable when working with pairings that are isomorphic via chain maps of non-zero degree.

\red{
\paragraph{A note about coefficients.} Through Section \ref{S: sheaf cup}, which includes the material on cup products, we allow our ground ring $R$ to be an arbitrary Dedekind domain. To define the singular intersection cohomology cup product, this requires the additional condition that the space $X$ be locally torsion free in an appropriate sense (meaning that the torsion of certain local intersection homology groups must vanish---see Section \ref{S: ihreview} below and the definition of the intersection cup product in \cite{GBF35}). Such conditions are automatic when working with field coefficients. }

\red{From Section \ref{S: compatibility} on we work with coefficients in a field. As the reader will see, this is required of our proofs, especially when where we apply Lemma \ref{L: coco UCT}, proven in Appendix
\ref{Appendix}. We expect the results of these sections to remain true for coefficients in a Dedekind domain and locally torsion free pseudomanifolds, but a proof will have to await further technology and perhaps a true sheafification of the intersection cap product. The authors plan to pursue this in future work.}

\red{Unfortunately, it does not make sense to ask these questions for spaces that are not locally torsion free, for in this case the authors' cup and cap products are not defined. }

\red{We mention, however, that Chataur, Saralegi-Aranguren, and Tanr\'e have devised an alternative definition of intersection cohomology, \emph{blown-up intersection cohomology} \cite{CST}, that permits cup and cap products \cite{CST7} without restrictions, though in general this intersection cohomology does not agree with that defined here  if the space is not locally torsion free.}

\red{\paragraph{Other recent work.}
Another approach to the relationship between cap products and Verdier duality has been given recently by Chataur, Saralegi-Aranguren, and Tanr\'e in \cite{CST20}. This cap product is defined in the context of the \emph{blown-up intersection cohomology} defined by the authors in \cite{CST}. This model for intersection cohomology is somewhat different from that considered here  as it is not defined directly in terms of duals of intersection chains. Nonetheless, \cite[Section 4]{CST20} provides a commutative diagram relating intersection chains, the Verdier dual of intersection cochains, and the blown-up intersection cochains, with the map from blown-up intersection cochains to intersection chains being a cap product defined in \cite{CST7}. Notably, the diagram in \cite{CST20} allows for noncompact spaces and arbitrary coefficients in a PID. The results are proven for PL pseudomanifolds but can likely be extended to the topological setting using the results of \cite{ST20}.
 }

\paragraph{Acknowledgment.} We thank the anonymous referee for several helpful suggestions.

\section{Conventions and some background}\label{S: background}

We assume the reader to be generally conversant with pseudomanifolds, intersection homology theory, sheaf theory, and derived categories.  We recommend \cite{GBF26} for an expository introduction to the version of sheaf-theoretic intersection homology with general perversities considered here and \cite{GBF23} for a more technical account\footnote{In \cite{GBF23} we utilized what we called ``stratified coefficients'' when working with singular intersection homology with arbitrary perversities. In \cite{GBF25} and \cite{GBF35} this notation was abandoned, and ``intersection homology with stratified coefficients'' is now simply called ``(non-GM) intersection homology.''}. 
 We also direct the reader to \cite{GBF25, GBF35} for intersection cochains and for the chain theoretic versions of intersection (co)homology cup and cap products. 

Other basic textbook introductions to intersection homology (with the original Goresky-MacPherson perversities)  include  \cite{Bo, KirWoo, BaIH}, and the original papers \cite{GM1, GM2, Ki} are well worth reading. The first three references also provide some background on sheaf theory in the derived category; other references  include \cite{SW, BR} for elementary sheaf theory, \cite{GelMan2} for derived categories, and \cite{DI04} for specifics in the derived category of sheaves.

\paragraph{Signs.}

We principally follow the signs in Dold \cite{Dold}, which agree with the Koszul
convention everywhere except in the definition of the coboundary on cochains. In particular, if $f\in \Hom^i(A^*,B^*)$, then $df=d_{B^*}\circ f-(-1)^if\circ d_{A^*}$ \cite[Definition VI.10.1]{Dold}. This also means that the coboundary acts on cochains by $(d\alpha)(x)=(-1)^{|\alpha|+1}(\bd x)$. The exception to following Dold is that we include a sign in our Poincar\'e duality isomorphisms so that they will also obey the Koszul sign conventions for degree $\dim(X)$ chain maps  (see \cite[Section 8.2.1]{GBF35}, \cite[Section 4.1]{GBF18}, or \cite{GBF25}  for further discussion). 

\subsection{Pseudomanifolds and intersection homology}\label{S: ihreview}

 We note here some of our conventions, which sometimes differ from those of other authors. 

Throughout the paper, $X$ will be a paracompact \emph{$n$-dimensional stratified topological pseudomanifold} \cite[Section 2.4]{GBF35}, \cite{GM2} with filtration $$X=X^n\supset X^{n-1}\supset X^{n-2}\supset\cdots\supset X^0.$$ Note that $X$ is allowed to have strata of codimension one. \emph{Skeleta} of $X$ will be denoted $X^i$. By a \emph{stratum} we will mean a connected component of one of the spaces $X^i-X^{i-1}$; each such stratum is an $i$-manifold or empty. A stratum $Z$ is a \emph{singular stratum} if $\dim(Z)<n$. Unless stated otherwise, we will assume $X$ is $F$-oriented with respect to some field $F$. By definition, this means $X-X^{n-1}$ is $F$-oriented.

We let $I^{\bar p}S_*(X;G)$ denote the complex of perversity $\bar p$ singular intersection chains with coefficients in the abelian group $G$.  
Here $\bar p$ is a general perversity $\bar p:\{\text{singular strata}\}\to \Z$, and $I^{\bar p}S_*(X;G)$ is the ``non-GM'' intersection chain complex (see \cite[Chapter 6]{GBF35}). \red{Recall that the Goresky-MacPherson-King singular intersection chain complex, which we denote $I^{\bar p}S^{GM}_*(X;G)$, is the subcomplex of the standard singular chain complex $S_*(X)$ consisting of chains $\xi$ such that} 
\begin{enumerate}
\item \red{if $\sigma:\Delta^i\to X$ is a singular $i$-simplex with non-zero coefficient in $\xi$ then for all singular strata $S$ the inverse image $\sigma^{-1}(S)$ is contained in the $i-\codim(S)+\bar p(S)$ skeleton of $\Delta^i$, and}

\item \red{ if $\tau:\Delta^{i-1}\to X$ is a singular $i-1$ simplex with non-zero coefficient in $\bd \xi$ then for all singular strata $S$ the inverse image $\tau^{-1}(S)$ is contained in the $i-1-\codim(S)+\bar p(S)$ skeleton of $\Delta^{i-1}$.}
\end{enumerate}
\red{The complex $I^{\bar p}S_*(X;G)$ is defined similarly except that no simplex of $\xi$ is allowed to have image contain in $X^{n-1}$ and the definition of the boundary map is modified so that any simplex of $\bd \xi$ contained in $X^{n-1}$ has its coefficient set to $0$ (and hence the second condition does not need to be checked for such simplices). More details can be found in \cite[Section 6.2]{GBF35}, where it is explained that this version of intersection homology theory possesses properties more closely aligned with the sheaf-theoretic intersection cohomology; cf. also \cite{Sa05}.}
If $X$ has no codimension one strata and $\bar p$ is a Goresky-MacPherson perversity as defined in \cite{GM1}, then $I^{\bar p}S_*(X;G)$ agrees with the singular intersection chain complex of King \cite{Ki} and the corresponding $I^{\bar p}H_*(X;G)$ is precisely the Goresky-MacPherson intersection homology \cite[Proposition 6.2.9]{GBF35}. We also consider the complex $I^{\bar p}S^\infty_*(X;G)$ of locally-finite intersection chains, i.e.\ intersection chains with possibly an infinite number of singular simplices with nonzero coefficient but such that every point has a neighborhood intersecting only finite many such simplices (see \cite[Section 2.2]{GBF10}). 
We denote by  $D\bar p$ the complementary perversity to $\bar p$, i.e. $D\bar p(Z)=\codim(Z)-2-\bar p(Z)$.

The intersection homology K\"unneth Theorem of \cite{GBF20} says that for any perversities $\bar p,\bar q$, there is a perversity $Q_{\{\bar p,\bar q\}}$ on $X\times X$ such that the exterior chain product induces a quasi-isomorphism $\times: I^{\bar p}S_*(X; F)\otimes I^{\bar q}S_*(X; F)\to I^{Q_{\{\bar p,\bar q\}}}S_*(X\times X; F)$. A more general K\"unneth Theorem allowing coefficients in a Dedekind domain is provided in \cite[Section 6.4]{GBF35}. See \cite{GBF35} for more details about these K\"unneth Theorems in general and \cite{GBF25, GBF35} for their use in defining cup and cap products. 

For a Dedekind domain $R$, a stratified pseudomanifold is 
\emph{locally $(\bar p,R)$-torsion free} if for all singular strata $Z$ and each $x\in Z$,  the module $I^{\bar p}H_{\codim(Z)-2-\bar p(Z)}(L_x;R)$ is $R$-torsion free, where $L_x$ is the  link of $x$ in $X$; see \cite[Definition 6.3.21]{GBF35}. This definition is originally due to Goresky-Siegel \cite{GS83}. This condition is automatic if $R$ is a field. To define the singular intersection cup and cap products it is required that $X$ satisfy such a condition. More details are provided below as needed.

\section{Hypercohomology in the derived category}\label{S: hyper}

We will work in both the category $C(X)$ of \red{unbounded complexes of sheaves} of $F$-vector spaces over a topological pseudomanifold $X$ and in its derived category $D(X)$. The sheaf theoretic approach to intersection (co)homology, and sheaf cohomology in general, is most often considered from the point of view of the derived category. Of course the objects of $C(X)$ and $D(X)$ are the same, but  $D(X)$ is the localization of $C(X)$ with respect to quasi-isomorphisms, so it is usual in the derived category language to identify quasi-isomorphic complexes. When doing this, we must exercise care with the identifying quasi-isomorphisms, especially if we are concerned, as we shall be, with keeping precise track of induced morphisms of hypercohomology groups.

One place  where it is often necessary to replace a sheaf complex $A^*$ with a quasi-isomorphic one is for the computation  of hypercohomology  $\H^*(X;A^*)$, which is typically computed by using any of a number of possible acceptable (e.g. flabby, injective, soft, fine) resolutions $A^*\to I^*$ and then taking $H^*(\Gamma(X;I^*))$. Some sources use the canonical Godement injective resolution $A^*\to I^*$ to define hypercohomology, though just as many seem content to leave the resolution unspecified. Morphisms in $D(X)$ induce maps of resolutions and hence maps of hypercohomology groups.

To illustrate the potential danger of imprecision in specifying resolutions, 
consider the following example in the derived category of abelian groups $D(Ab)$. Let $G$ be an injective group (as a $\Z$-module and thought of as a complex concentrated in degree $0$); then any automorphism of $G$  provides a quasi-isomorphism $G\to G$ and hence is an injective resolution. In particular, $\H^0(G)\cong G$ and $\H^i(G)=0$, $i\neq 0$. Given two such groups $G, H$, a morphism $f:G\to H$, and any automorphisms $\phi:G\to G$ and $\vartheta:H\to H$, we obtain a commutative diagram 
\begin{diagram}
G &\rTo^{g=\vartheta^{-1} f\phi}& H\\
\dTo^\phi&&\dTo^\vartheta\\
G&\rTo^f &H.
\end{diagram}
Since the identity maps, $\phi$, and $\vartheta$ are all injective resolutions, either
$f$ or $g$ could represent the hypercohomology morphism $\H^*(G)\to \H^*(H)$ induced by $f$ (up to specified isomorphisms, of course). 
 But clearly $g$ depends on our choices of $\phi$ and $\vartheta$.

 This  argument demonstrates the importance of paying careful attention to the resolutions if we hope to be precise about induced hypercohomology maps. 
To avoid this sort of  ambiguity as much as possible, we will mostly endeavor to work with particular natural choices of objects and to be as precise with maps between these objects as possible. In those cases where it is not possible to prescribe a specific map in the sheaf category, we will see that fixing sufficiently nice resolutions and giving a morphism between them in the derived category is at least sufficient to determine a unique map of hypercohomology groups. 

Toward these ends, we make a definition:

\begin{definition}
Let $\phi$ be a paracompactifying family of supports on $X$ \cite[Definition I.6.1]{BR}. We will call an object $A^*\in D(X)$ \emph{$\phi$-cohomology ready} (or \emph{$C_\phi$-ready}) if  some  (and hence any) injective resolution  $A^*\to I^*$ induces an isomorphism from
$H^*(\Gamma_\phi(X;A^*))$ to $H^*(\Gamma_\phi(X;I^*))$. 

If $B^*$ is any object of $D(X)$ and $A^*$ is a $C_\phi$-ready object that is quasi-isomorphic to $B^*$, we may call $A^*$ a \emph{$C_\phi$-ready representative of $B^*$}.

If $\phi$ is the family of closed sets, we write simply \emph{$C$-ready}. 
\end{definition}   

The point of the definition is that if $A^*$ is a $C_\phi$-ready representative of $B^*$, then the cohomology groups $H^*_{\phi}(\Gamma(X;A^*))$ are isomorphic to  $\H^*_{\phi}(X;B^*)$ (and also $\H^*_{\phi}(X;A^*)$). In particular, if $A^*$ is $C_\phi$-ready, we will take $\H^*_{\phi}(X;A^*)$ to mean precisely $H^*_{\phi}(\Gamma(X;A^*))$, eliminating ambiguity in the definition of hypercohomology for such sheaves.

\red{Before stating the next lemma, we recall the definition of a homotopically fine complex of sheaves, as this concept is less well known than some of the other more common flavors of sheaves, such as injective, fine, flabby, soft, etc. This definition can be found in Exercise 32 of \cite[Section II]{BR}. For a paracompactifying family of supports $\phi$, the sheaf complex $A^*$ is called \emph{homotopically $\phi$-fine} if for every locally finite covering $\{U_\alpha\}$ with one member of the form $X-K$ for some $K\in \phi$ there are degree zero maps $h_\alpha\in \Hom(A^*,A^*)$ such that $|h_\alpha|\subset U_\alpha$ and such that $\sum h_\alpha$ is chain homotopic to the identity. By comparison, fine sheaves can be defined in terms of such a partition of unity condition but satisfied exactly, not just up to homotopy; cf. \cite[Exercise II.13]{BR} or \cite[Definition II.3.3]{WELLS}. The utility of homotopically $\phi$-fine complexes is that they satisfy $H^*(H^p_\phi(X;A^*))=0$ for all $p>0$, even if the $A^i$ themselves are not $\phi$-acyclic; consequently their hypercohomology spectral sequences degenerate so that $\H_\phi^*(X;A^*)\cong H^*(\Gamma_\phi(X;A^*))$. Hence such an $A^*$ is $C_\phi$-ready, as we now argue in detail.}

\begin{lemma} \label{L: C-ready}
Let $\phi$ be a paracompactifying family of supports. Any complex $A^*$ such that the derived cohomology sheaves $\mc H^i(A^*)$ vanish for sufficiently small $i$ and such that each $A^j$ is  $\phi$-acyclic (e.g.\ flabby, injective, $\phi$-soft, or $\phi$-fine) is $C_\phi$-ready. This also holds if $\mc H^i(A^*)$ vanishes for sufficiently small $i$ and the complex $A^*$ is 
$\phi$-homotopically fine.
\end{lemma} 
\begin{proof}
The argument is essentially that given in the discussions following Proposition 3.5 and Lemma 3.10 in  \cite{GBF10}. 

Recall that we assume $X$ to be a paracompact and finite dimensional stratified pseudomanifold. 
If $A^*$ is $\phi$-homotopically fine or each $A^*$ is $\phi$-acyclic then   $H^*(H^p_\phi(X; A^*))=0$ for all $p>0$ by \cite[pages 172 and 202]{BR}.  So there exists a spectral sequence with $E_2^{p,q}=H^p_\phi(X;\mc{H}^q(A^*))$ abutting to $H^{p+q}(\Gamma_\phi(X;A^*))$ by \cite[Theorem IV.2.1]{BR}.

On the other hand, since $\mc H^i(A^*)$ vanishes for sufficiently small $i$, there is an $N$ such that 
the truncation map $A^*\to \tau^{\geq N} A^*$ is a quasi-isomorphism (see \cite[Section V.1.10]{Bo}). As a bounded from below complex, $\tau^{\geq N} A^*$ has an injective resolution \cite[Corollary V.1.18]{Bo}, say $\tau^{\geq N}A^*\to I^*$. The composition  $A^*\to \tau^{\geq N}A^*\to I^*$ of quasi-isomorphisms is a quasi-isomorphism, and so $H^*(\Gamma_\phi(X;I^*))$ is the hypercohomology $\H_\phi^*(X;A^*)$ by definition \cite[Section V.1.4]{Bo}. 

Finally, since injective sheaves are also $\phi$-acyclic (for any $\phi$), the map of spectral sequences induced by our quasi-isomorphism $A^*\to I^*$ gives an isomorphism  $H^i(\Gamma_\phi(X;A^*))\to H^i (\Gamma_\phi(X;I^*))$ for all $i$ by \cite[Theorem IV.2.2]{BR}. 
\end{proof}

The next lemma shows that once we have fixed specific $C_\phi$-ready objects,  any  morphism in the derived category between these objects induces a uniquely determined map on hypercohomology. 
This is not completely obvious because of the identifications that are made in the derived category - in particular there may be many possible ways to represent a morphism in $D(X)$ in terms of morphisms in $C(X)$.

\begin{lemma}\label{L: rep}
Suppose $A^*,B^*\in D(X)$ are $C_\phi$-ready with $\mc H^i(A^*)=\mc H^i(B^*)=0$ for sufficiently small $i$ and that $f\in \Mor_{D(X)}(A^*,B^*)$. Then $f$ determines a unique morphism $\H_\phi^*(X;A^*)=H^*(\Gamma_\phi(X;A^*)) \to H^*(\Gamma_\phi(X;B^*))=\H_\phi^*(X;B^*)$.
\end{lemma}

Before indicating the proof of the lemma, we make one more definition that follows naturally from it. 
Suppose $A^*$ and $B^*$ are $C_\phi$-ready representatives of $S^*$ and $T^*$. Then for any $f\in \Mor_{D(X)}(S^*,T^*)$, 
choices of fixed quasi-isomorphisms  $s: A^*\to S^*$ and $t:B^*\to T^*$
 determine a morphism $g\in \Mor_{D(X)}(A^*,B^*)$, and, by Lemma \ref{L: rep}, this determines a unique map $g: H^*(\Gamma_\phi(X;A^*))\to H^*(\Gamma_\phi(Y;B^*))$. We will refer to $g$ as a \emph{cohomological representative of $f$}. 

\begin{proof}[Proof of Lemma \ref{L: rep}]
An element  $f\in \Mor_{D(X)}(A^*,B^*)$ is represented by a ``roof'' diagram $A^*\xleftarrow{s} C^*\xrightarrow{\hat f} B^*$, where $s,\hat f$ are chain maps of sheaf complexes and $s$ is a quasi-isomorphism. Recall that the induced maps of hypercohomology groups can be defined by means of injective resolutions. These exist here due to the assumptions on the boundedness of $\mc H^i(A^*)$ and $\mc H^i(B^*)$ as in the proof of Lemma \ref{L: C-ready}.
 Taking such resolutions yields a diagram

\begin{diagram}
A^*&\lTo^s& C^*&\rTo^{\hat f}&B^*\\
\dTo^{i_1}&&\dTo&&\dTo^{i_3}\\
\mc I_1^*&\lTo^{\td s}& \mc I_2^*&\rTo^{\td f}&\mc I_3^*.
\end{diagram}
If we assume the vertical maps are fixed injective resolutions, then $\td s$ and $\td f$ are determined up to chain homotopy \cite[Section V.5.16]{Bo}. Now taking the cohomology of global sections provides a map $\H^*(X;A^*)\to \H^*(X;B^*)$, namely  $(i_3)^{-1}\td f\td s^{-1}i_1$. Note that $i_3$ is invertible as a map on cohomology of global sections by our assumption that $B^*$ is $C_\phi$-ready, while $\td s$ is invertible on cohomology since it is a quasi-isomorphism of injective complexes.

This map is independent of the choices of resolutions, since if $j:A^*\to \mc J^*$ is another resolution of $A^*$, then the properties of injective resolutions again give a chain homotopy class of chain homotopy equivalences $e:\mc I_1^*\to \mc J^*$ such that $ei_1=j$ up to chain homotopy, and similar for the other terms. Further application of the properties of injective objects allows one to obtain  a triangular prism diagram demonstrating that the map $\H^*(X;A^*)\to \H^*(X;B^*)$ does not depend on the choices of resolutions.

Now, suppose $A^*\xleftarrow{t} D^*\xrightarrow{\hat g} B^*$ is another roof diagram representing $f$, then there exists a chain homotopy commutative diagram (see \cite[Lemma III.2.8]{GelMan2})
\begin{diagram}
&&C^*\\
&\ldTo^s&\uTo&\rdTo^{\hat f}\\
A^*&\lTo^{r}&E^*&\rTo&B^*\\
&\luTo^t&\dTo&\ruTo^{\hat g}\\
&&D^*&&&
\end{diagram}
with $r,s,t$ quasi-isomorphisms (and thus so are the vertical maps, as well). 

Again by the properties of injective resolutions, we can form injective resolutions for each object in the diagram and a map from this diagram to the injective version that commutes up to chain homotopy. Thus, taking into account the independence of choice of injective resolutions, the homotopy commutativity of the injective version of the diagram shows that any representative of the morphism $f$ in $D(X)$ yields the same map $\H^*(X;A^*)\to \H^*(X;B^*)$.
\end{proof}

\section{Some particular sheaf complexes}\label{S: particular}

\subsection{The Verdier dualizing complex and Verdier duality}\label{S: Verdier complex}

We will let $\D^*_X$ (or simply $\D^*$ when the space is unambiguous) denote the Verdier dualizing sheaf complex on $X$. Specifically, we will assume $\D^*$ as constructed in \cite[Section V.7.A]{Bo}. This construction requires a fixed injective resolution $I^*$ of the ground ring and a $c$-soft flat resolution $\mc K^*$ of the constant sheaf with stalks isomorphic to the ground ring.  \red{When we consider Verdier duality in the context of coefficients in a field $F$, we can assume $F$ is its own injective resolution as a complex, i.e\ that the only nontrivial term in $I^*$ is $I^0=F$.} We also assume $\mc K^*$ chosen and fixed for the given space $X$. Then $\D^*$ is the sheaf $U\to \Hom^*(\Gamma_c(\mc K^*_U),I^*)$, and $\D^*$ is injective by \cite[Corollary V.7.6]{Bo}.

If $\ms S^*$ is a complex of sheaves, we let $\mc D \ms S^*$ stand for the Verdier dual of $\ms S^*$. Explicitly, we set $\mc D\ms S^*=\SHom^*(\ms S^*,\D^*)$, where $\SHom$ denote the $\Hom$ sheaf.  It is also common to suppress the grading decoration for the Verdier dual, but when necessary we will use the notation $\mc D^i\ms S^*\cong \SHom^i(\ms S^*,\D^*)$. This is not the  definition of $\mc D\ms S^*$ given by Borel in \cite[Section V.7]{Bo}, however it is equivalent up to isomorphism by \cite[Theorem V.7.8.ii]{Bo} (and the definition we use is quite prominent in the literature). Borel defines the Verdier dual of $\ms S^*$ as the sheaf $U\to \Hom^*(\Gamma_c( \mc J^*_U),I^*)$, where $\mc J^*$ is a $c$-soft resolution of $\ms S^*$ (or $\ms S^*$ itself if it were $c$-soft) and $I^*$ is an injective resolution of the ground ring. In what follows, we shall notate the sheaf complex $U\to \Hom^*(\Gamma_c( \mc J^*_U),I^*)$ as $\ms L\ms S^*$; in the particular examples we consider below, the resolution of $\ms S^*$  will be specified precisely (and in fact it will be $\ms S^*$ itself). 
Then \cite[Theorem V.7.8.ii]{Bo} says $\mc D\ms S^*\cong \ms L\ms S^*$ in the category of sheaf complexes.

\subsection{The Deligne sheaf}\label{R: uniqueness}

In the sheaf-theoretic formulation of intersection homology, $I^{\bar p}H_*(X;R)$ corresponds to the hypercohomology of the Deligne sheaf complex, denoted $\mc P_{\bar p}^*$. We sometimes write $\mc P^*_{\bar p}(\mc E)$ if we wish to specify a coefficient system $\mc E$ on $X-X^{n-1}$, \red{i.e.\ a locally constant sheaf of finitely-generated $R$ modules. If omitted from the notation, we typically assume the constant coefficient sheaf $\mc R$ whose stalks are the ground ring $R$, which we assume to be a Dedekind domain.} These sheaf complexes were first defined in \cite{GM2} for Goresky-MacPherson perversities. A version for more general perversities was constructed in \cite{GBF23} (see \cite{GBF26} for an exposition). \red{We briefly recall the construction and the important axiomatic characterization.}

\red{Given our n-dimensional stratified pseudomanifold $X$, let $U_k=X-X^{n-k}$ and write $i_k:U_k\into U_{k+1}$ for the inclusion. For simplicity, suppose that $\bar p$ is a Goresky-MacPherson perversity so that it depends only on the codimension of the strata and hence its value on any codimension $k$ stratum can be written $\bar p(k)$. Such perversities also satisfy $\bar p(k+1)\geq\bar p(k)$ (as well as additional restrictions). Let $\tau_{\leq j}$ denote the standard sheaf complex truncation functor and $Ri_{k*}$ the right derived functor of the sheaf pushforward induced by $i_k$. Finally, let $\mc E$ be a system of of coefficients over $U_1=X-X^{n-1}$. Then the Deligne sheaf is defined by an iterative pushforward-and-truncate procedure as:
$$\mc P^*_{\bar p}(\mc F)= \tau_{\leq \bar p(n)}Ri_{n*}    \cdots\tau_{\leq \bar p(1)}Ri_{1*}\mc E.$$
For arbitrary perversities, the truncation functor needs to be modified to a truncation that is localized at the strata; details can be found in \cite{GBF23}.}

\red{We also recall that there are several axiomatic characterizations of $\mc P^*_{\bar p}$. We review the simplest, typically referenced as the axioms Ax1. A complex $\ms S^*$ satisfies these axioms for perversity $\bar p$ and coefficient system $\mc E$ if 
\begin{enumerate}
\item Up to quasi-isomorphism, $\ms S^*$ is bounded, $\ms S^i=0$ for $i<0$, and $\ms S^*|_{U_1}=\mc E$, 
\item if $Z$ is a singular stratum and $x\in Z$ then $H^i(\ms S^*_x)=0$ for $i>\bar p(Z)$, and
\item if $Z$ is a singular stratum of codimension $k$ and $x\in Z$ then the attachment map $\alpha_k:\ms S^*|_{U_{k+1}}\to Ri_{k*}(\ms S^*|_{U_k})$ is a quasi-isomorphism at $x$ in degrees $\leq \bar p(Z)$. 
\end{enumerate}
A sheaf complex is quasi-isomorphic to $\mc P_{\bar p}^*(\mc E)$ if and only if it satisfies these axioms. See \cite[Section V.2]{Bo} for an exposition for Goresky-MacPherson perversities and \cite[Section 3]{GBF23} for general perversities.  }

Our most important examples of $C$-ready sheaves will be the sheaves of intersection chains $\mc I^{\bar p}\mc S^*$  and intersection cochains $\mc I_{D\bar p}\mc C^*$ with coefficients in \red{$\mc R$, the constant system with stalk $R$}, both of which are $C$-ready and $C_c$-ready representatives of the perversity $\bar p$ Deligne sheaf $\mc P^*_{\bar p}$ with the same coefficients (see Proposition \ref{P: IC is IS} below). In fact, the Deligne sheaf itself is generally only defined up to quasi-isomorphism \cite{GM2, Bo}, relying as it does on resolutions, which are rarely specified, at various stages of its construction. For our purposes, we assume that a specific sheaf representative of $\mc P^*_{\bar p}$ has been fixed, as well as a specific quasi-isomorphism from the constant sheaf $\mc R$ on $U_1=X-X^{n-1}$ with stalk $R$  to $\mc P^*_{\bar p}|_{U_1}$. Later on, we will be working with specific $C$-ready representatives of $\mc P^*_{\bar p}$, so we will not need these assumptions.  

\begin{remark}
Deligne sheaves, and hence intersection (co)homology, can be defined with coefficients in any  local system (i.e. locally constant sheaf) on $U_1$ of finitely generated modules over a  commutative noetherian ring of finite cohomological dimension \cite[Section V.2]{Bo}. However, aside from Section \ref{S: sheaf}, we will not pursue this level of generality here. We invite the reader to consider the natural generalizations in what follows.
\end{remark}

\smallskip

Now, suppose $\ms S^*$ and $\ms T^*$ are any sheaf complexes quasi-isomorphic respectively to $\mc P^*_{\bar p}$ and $\mc P^*_{\bar q}$ with $\bar p\leq \bar q$. Then a consequence of \cite[Theorem 3.5]{GM2} (see also \cite[Lemma V.9.1]{Bo}, \cite[Lemma 4.9]{GBF23}) is that $\Mor_{D(X)}(\ms S^*,\ms T^*)\cong \Mor_{D(U_1)}(\ms S^*|_{U_1},\ms T^*|_{U_1})$, where the isomorphism is induced by restriction. Furthermore, since the Deligne sheaves restrict over $U_1$ to locally constant systems (up to quasi-isomorphism), by \cite[Lemma V.9.13]{Bo}, $\Mor_{D(U_1)}(\ms S^*|_{U_1},\ms T^*|_{U_1})\cong \Mor_{C(U_1)}(\mc H^0(\ms S^*)|_{U_1},\mc H^0(\ms T^*)|_{U_1})$, where $\mc H^*$ is the derived cohomology sheaf complex. In particular, since we assume our Deligne sheaves to have constant coefficients, $\mc H^0(\ms S^*)|_{U_1}$ and $\mc H^0(\ms T^*)|_{U_1}$ are each isomorphic to the constant sheaf $\mc R$ with stalk $R$ on $U_1$, and maps $\mc R\to \mc R$ are determined by their restrictions to a set of points, one in each connected component of $U_1$. Thus if $U_1$ has $m$ connected components, $\Mor_{D(X)}(\ms S^*,\ms T^*)\cong R^m$. It also follows that the set of automorphisms in $\Mor_{D(X)}(\ms S^*,\ms S^*)$ is isomorphic to $(R^*)^m$, where $R^*$ is the set of units of $R$.

By \cite[Proposition V.9.4]{Bo} and \cite[Lemma 4.9]{GBF23}, the same conclusions hold if $\bar p\leq \bar t$ and $\ms T^*$ is quasi-isomorphic to $\D^*[-n]$. In this case $\mc H^0(\D^*[-n]|_{U_1})$ is isomorphic to the $R$-orientation sheaf $\mc O$ on $U_1$. However, if $X$ is $R$-oriented, we again have $\mc H^0(\D^*[-n]|_{U_1})\cong \mc O\cong \mc R$ and again $\Mor_{D(X)}(\ms S^*,\ms T^*)\cong R^m$ with the morphisms determined by restricting  $\mc H^0$ to an appropriate set of points.

Finally, suppose $\ms S^*, \ms T^*, \ms U^*$ are quasi-isomorphic to $\mc P^*_{\bar p}, \mc P^*_{\bar q}, \mc P^*_{\bar r}$ with $\bar p+\bar q\leq \bar r$. Then by \cite[Proposition V.9.14]{Bo} and \cite[Theorem 4.6]{GBF23}, $\Mor_{D(X)}(\ms S^*\otimes\ms T^*, \ms U^*)$ is bijective with the set of induced maps $\mc H^0(\ms S^*)|_{U_1}\otimes \mc H^0(\ms T^*)|_{U_1}\to \mc H^0(\ms U^*)|_{U_1}$ and hence with the collection of pairings $H^0(\ms S^*_x)\otimes H^0(\ms T^*_x)\to H^0(\ms U^*_x)$ as $x$ runs over a set of representative points, one in each connected component of $U_1$.

\subsection{The sheaf complex of intersection chains}\label{S: SIC}

Let $X$ be an $n$-dimensional stratified topological pseudomanifold, possibly with codimension one strata, and let $\bar p$ be a general perversity. \red{Recall the intersection chain complexes $I^{\bar p}S_*(X;G)$ and $I^{\bar p}S_*^\infty(X;G)$ as reviewed above in Section \ref{S: ihreview}.}
Further, recall from\footnote{In \cite{GBF10} the perversities were required to satisfy certain restrictive conditions, but the results there hold for completely general perversities, as noted in \cite[Section 2.2]{GBF23}, using what's now called ``non-GM'' singular intersection homology in \cite{GBF35}.} \cite[Section 3.1]{GBF10} that for any coefficient system of abelian groups $\mc G$ defined on $X-X^{n-1}$ (or, more generally, $R$-modules, where $R$ is a noetherian commutative ring of finite cohomological dimension), one can define a sheaf complex\footnote{Again, we tend to leave the coefficients out of the sheaf notation for simplicity; we use the notation $\mc I^{\bar p}\mc S^*(\mc G)$ when the coefficients $\mc G$ need to be emphasized. } $\mc I^{\bar p}\mc S^*$ on $X$ as the sheafification of the presheaf $U\to I^{\bar p}S^{\infty}_{n-*}(X,X-\bar U;\mc G)$ or, equivalently, of the presheaf $U\to I^{\bar p}S_{n-*}(X,X-\bar U;\mc G)$. To account for the grading shift, the boundary $d$ in the presheaf, and hence the sheaf, corresponds to\footnote{This sign to account for the grading shift was inadvertently neglected in prior work of the first-named author. Of course the sign does not affect (co)homology computations, but it will be important here as we pay careful attentions to signs and gradings.} $(-1)^n\bd$. This is consistent with shifting by $-n$ the presheaf complex $U\to I^{\bar p}S_{-*}(X,X-\bar U;\mc G)$ that is also commonly used in sheaf theoretic treatments of intersection homology. \red{To see that $\mc I^{\bar p}\mc S^*$ can be used to compute intersection homology, we review some further background.}

\red{ Recall that a presheaf $A$ is called a \emph{monopresheaf} if for any $U\subset X$ with $U=\cup_{\alpha}U_\alpha$, the $U_{\alpha}$ each open, and if $s,t\in A(U)$ satisfy $s|_{U_\alpha}=t|_{U_\alpha}$ for all $\alpha$ then $s=t$. A presheaf $A$ is called \emph{conjunctive} if for any collection of open subsets $U_\alpha\subset X$ and any $s_\alpha\in A(U_\alpha)$ such that $s_\alpha|_{U_\alpha\cap U_\beta}=s_\beta|_{U_\alpha\cap U_\beta}$ for all $\alpha,\beta$ there exists an $s\in A(U)$ with $s|_{U_\alpha}=s_\alpha$ for all $\alpha$. A presheaf is a sheaf precisely when it is a conjunctive monopresheaf  \cite[page 6]{BR}. A slightly weaker set of conditions is that a presheaf satisfy the conjunctivity condition for all coverings $\{U_{\alpha}\}$ and that the only element of $A(X)$ with empty support (i.e. that goes to $0$ under sheafification) is the zero element. If we let $\mc A$ denote the sheafification of $A$, then for such a presheaf it remains the case that $A_\phi(X)\cong \Gamma_\phi(X;\mc A)$ by \cite[Theorem I.6.2]{BR}.}

\red{It is shown in \cite[Lemmas 3.2, 3.3]{GBF10} that the presheaf $U\to I^{\bar p}S^{\infty}_{n-*}(X,X-\bar U;\mc G)$ is conjunctive for coverings and has no non-trivial global sections with empty support. It follows that $H^*(\Gamma(X;\mc I^{\bar p}\mc S^*))\cong I^{\bar p}H_{n-*}^\infty(X;\mc G)$; see  \cite[Corollary 3.4]{GBF10} for further details. }

Furthermore, the sheaf complex $\mc I^{\bar p}\mc S^*$  is homotopically fine (and so $c$-homotopically fine) by \cite[Proposition 3.5]{GBF10}, and in \cite{GBF23} it is shown\footnote{The proof is axiomatic and does not hinge on any choice of sign conventions for boundary maps. In fact, it is true in general that the complexes $(C^*,d)$ and $(C^*,-d)$ are quasi-isomorphic via the chain map $(-1)^{|\cdot|}\text{id}$.}  that $\mc I^{\bar p}\mc S^*$ is quasi-isomorphic to the Deligne sheaf complex $\mc P^*_{\bar p}$. It follows that the local intersection homology groups $H^i(\mc I^{\bar p}\mc S^*_x)$ vanish for $|i|$ sufficiently large, as this is true for $\mc P^*_{\bar p}$ by the Goresky-MacPherson axioms. Hence $\mc I^{\bar p}\mc S^*$ is  $C$-ready and $C_c$-ready by Lemma \ref{L: C-ready}, and in particular it is a  $C$-ready and $C_c$-ready representative of $\mc P^*_{\bar p}$. 
 Thus $\H^*(X; \mc P^*_{\bar p})\cong\H^*(X;\mc I^{\bar p}\mc S^*)\cong I^{\bar p}H_{n-*}^\infty(X; \mc G)$ and similarly  $\H_c^*(X; \mc P^*_{\bar p})\cong\H_c^*(X;\mc I^{\bar p}\mc S^*)\cong I^{\bar p}H_{n-*}(X; \mc G)$.

\subsection{The sheaf of intersection cochains}\label{S: sheaf}

We now introduce the sheaf of intersection cochains. \red{In this section, we assume coefficients in a Dedekind domain $R$. Recall that the Dedekind domains include the Principal Ideal Domains.}

Intersection cochains $I_{\bar p}S^*(X;R)\cong \Hom_R(I^{\bar p}S_*(X;R);R)$ 
appear sporadically in the literature with a thorough study appearing in \cite{GBF25} with field coefficients and then in \cite{GBF35} with $R$ a Dedekind domain. 

Using the natural restriction of cochains, we can define a sheaf\footnote{We hope that using a $\mc C$  for the intersection cochain sheaf $\mc I_{\bar p} C^*$  and an $\mc S$ for the 
intersection chain sheaf $\mc I^{\bar p}\mc S^*$  will help to eliminate any confusion that might be caused by leaving the difference in notation only up to the placement of the perversity decoration $\bar p$. As is usual for sheaf theory, we use cohomological indexing for both sheaf complexes.}
$\mc I_{\bar p}\mc C^*$ as the sheafification of  the presheaf $U\to I_{\bar p}S^*(U; R)$. We will see below in Proposition \ref{P: conj} that the presheaf is conjunctive. Also as for ordinary singular cochains, $H^*(I_{\bar p}S_0^*(X; R))=0$, where $I_{\bar p}S_0^*(X; R)$ is the group of presheaf sections with zero support; this follows from the fact that subdivision induces isomorphisms on singular intersection homology (by \cite[Corollary 6.3.10]{GBF35}). Thus  $I_{\bar p}H^*(X; R)\cong H^*(\Gamma(X;\mc I_{\bar p}\mc C^*))$, and similarly for compact supports. See the discussion of singular cohomology in \cite[Section I.7]{BR} for more details. 

The sheaf   $\mc I_{\bar p}\mc C^*$ (using traditional perversities and real coefficients) is considered also in \cite{BHS}, where 
it is claimed that the presheaf $U\to I_{\bar p}S^*(U;\R)$ is in fact a flabby sheaf. This does not seem to be correct, as the presheaf of cochains is not  a sheaf  even for ordinary cochains (see \cite[page 26]{BR}). Nonetheless,  the sheaf $\mc I_{\bar p}\mc C^*$ is indeed flabby, as we will show in Proposition \ref{P: flabby}.  

Before proving Propositions \ref{P: conj} and \ref{P: flabby}, we have the following important corollary:

\begin{corollary}\label{C: cohomology}
$\H^*(X;\mc I_{\bar p}\mc C^*)\cong I_{\bar p}H^*(X; R)$ and $\H_c^*(X;\mc I_{\bar p}\mc C^*)\cong I_{\bar p}H_c^*(X; R)$.
\end{corollary}
\begin{proof}
Since $\mc I_{\bar p}\mc C^*$ is flabby and bounded below, it is $C$-ready and $C_c$-ready by Lemma \ref{L: C-ready}, thus  $\H^*(X;\mc I_{\bar p}\mc C^*)\cong H^*(\Gamma(X;\mc I_{\bar p}\mc C^*))$. Furthermore, the presheaf of cochains is conjunctive and the cohomology of cochains with zero support is trivial, so $H^*(\Gamma(X;\mc I_{\bar p}\mc C^*)) \cong I_{\bar p}H^*(X;R)$ \cite[Theorem I.6.2]{BR}, cf.\ the proof of \cite[Equation I.7.(16)]{BR}. The argument with compact supports is analogous.
\end{proof}

We now turn to proving Proposition \ref{P: conj} and \ref{P: flabby}.

\begin{proposition}\label{P: conj}
For each $i$, the presheaf $U\to I_{\bar p}S^i(U;\R)$ is conjunctive.
\end{proposition}
\begin{proof}
Let $U_j$ be a set of open subsets of $X$ and let $U=\cup_j U_j$. We may assume that the indexing set $j$ is well ordered. Consider the sequence of maps 
\begin{equation*}
\oplus_{j<k}I^{\bar p}S_i(U_j\cap U_k;R)\xr{f} \oplus_j I^{\bar p}S_i(U_j;R)\xr{g} I^{\bar p}S_i(U;R),
\end{equation*}
where $g$ acts as the inclusion on each summand and $f$ takes  $\xi\in I^{\bar p}S_i(U_j\cap U_k;R)$ to $(\xi, -\xi)\in I^{\bar p}S_i(U_j;R)\oplus I^{\bar p}S_i(U_k;R)$ when $j<k$. Then the composition $gf$ is trivial, and we can think of this sequence as part of a chain complex $A_*$ with the shown modules being $A_2,A_1,A_0$ and all other $A_*$ trivial. 

The dual chain complex restricts on these three terms to 
\begin{equation*}
\prod_{j<k}I_{\bar p}S^i(U_j\cap U_k;R)\xleftarrow{f^*} \prod_j I_{\bar p}S^i(U_j;R)\xleftarrow{g^*} I_{\bar p}S^i(U;R).
\end{equation*}
Now if $\alpha\in I_{\bar p}S^i(U;R)$ then the image of $g^*(\alpha)$ in the factor $I_{\bar p}S^i(U_j;R)$ is just the restriction of $\alpha$. Similarly, if $\prod \alpha_j\in\prod_j I_{\bar p}S^i(U_j;R)$ and $\xi\in I_{\bar p}S^i(U_a\cap U_b;R)$, then $f^*(\prod \alpha_j)(\xi)$ is obtained by applying $\prod \alpha_j$ to $f(\xi)=(\xi,-\xi)\in I^{\bar p}S_i(U_a;R)\oplus I^{\bar p}S_i(U_b;R)$, which is $\alpha_a(\xi)-\alpha_b(\xi)$. Thus $f^*$ is a difference of restrictions on each $U_a\cap U_b$.  So the condition of conjunctivity of the presheaf is that $\ker f^*\subset \im g^*$, which in this case is equivalent to $H^1(\Hom(A_*,R))=0$. The modules $I^{\bar p}S_i(\cdot ;R)$ are projective by \cite[Lemma 6.3.1]{GBF35}, and hence so are the direct sums of such modules  and their submodules \cite[Proposition I.4.5 and Corollary I.5.3]{HS}. Therefore, we can employ the Universal Coefficient Theorem, and it suffices to show that $H_1(A_*)=0$ and that $\Ext(H_0(A_*),R)=0$.

Let $\xi=\oplus \xi_j\in \oplus_j I^{\bar p}S_i(U_j;R)$ with $g(\oplus \xi_j)=0$. As an element of a direct sum, all but a finite number of $\xi_j$ must be trivial. We will perform an induction over the number of indices $j$ for which $\xi_j$ is nontrivial. 

First suppose there is just one index, say $k$, for which $\xi_k$ is nontrivial. As the restriction of $g$ to each summand $g: I^{\bar p}S_i(U_j;R)\to I^{\bar p}S_i(U;R)$ is injective, $g(\xi)=g(\xi_k)=0$ implies that $\xi_k=\xi=0$, in which case $\xi$ is certainly in the image of $f$. 

Next suppose we have shown that if $g(\xi)=0$ and $\xi=\oplus \xi_j$ with fewer than $m$ of the $\xi_j$ nontrivial then $\xi\in \im f$. Now let $\xi=\oplus_{j=1}^m\xi_j$ for some indices that we relabel as $1,\ldots, m$. The hypothesis is that $\sum_{j=1}^m \xi_j=0\in I^{\bar p}S_i(U;R)$. Now each chain $\xi_j$ has the form $\xi_j=\sum a_k\sigma_k$, where $a_k\in R$ and the $\sigma_k$ are singular simplices in $U$ (see \cite[Definition 6.2.1]{GBF35}). In particular, each simplex of $\xi_1$ with a non-zero coefficient must have image in $U_1$ as well as some other $U_j$, $2\leq j\leq m$, as otherwise this simplex could not end up with a zero coefficient in $\sum_{j=1}^m \xi_j$. It therefore follows from \cite[Proposition 6.5.2]{GBF35} that $\xi_1=\sum_{2\leq j\leq m}\xi_{1j}$ where $\xi_{1j}\in I^{\bar p}S_i(U_1\cap U_j;R)$. Now, consider $\eta=\oplus_{j=2}^m (-\xi_{1j})\in \oplus I^{\bar p}S_i(U_1\cap U_j;R)$. We have $g(\xi+f(\eta))=g(\xi)+gf(\eta)=0$, but since the image under $f$ of $\eta$ is $\left(-\sum_{2\leq j\leq m}\xi_{1j}\right)\oplus \xi_{12}\oplus\cdots \oplus \xi_{1m}$, we have  $$\xi+f(\eta)=\left(\xi_1- \sum_{2\leq j\leq m}\xi_{1j}\right)\oplus (\xi_2+\xi_{12})\oplus \cdots\oplus(\xi_m+\xi_{1m})=0\oplus (\xi_2+\xi_{12})\oplus \cdots\oplus(\xi_m+\xi_{1m}).$$ So $\xi+f(\eta)$ is nontrivial only for the indices $2\leq j\leq m$. By induction, $\xi+f(\eta)\in \im f$. Thus $\xi\in \im(f)$. Therefore we have $H_1(A_*)=0$. 

Turning to $H_0(A_*)=I^{\bar p}S_i(U;R)/ \im(g)$, we will show in the following lemma that this module is projective. Therefore, $\Ext(H_0(A_*),R)=0$, and the proposition now follows.
\end{proof}

\begin{lemma}
Let $\mc U$ be a covering of $X$, and let $I^{\bar p}S^{\mc U}_i(X;R)=\im\left(\oplus_{U\in\mc U} I^{\bar p}S_i(U;R)\xr{g} I^{\bar p}S_i(X;R)\right)$, where $g$ restricts to the inclusion on each summand. Then for each degree $i$ the module $I^{\bar p}S^{\mc U}_i(X;R)$ is a direct summand of $I^{\bar p}S_i(X;R)$ and so $I^{\bar p}S_i(X;R)/I^{\bar p}S^{\mc U}_i(X;R)$ is projective.
\end{lemma}
\begin{proof}
The last claim follows from the first by identifying the quotient with the complementary summand to $I^{\bar p}S^{\mc U}_*(X;R)$. As $I^{\bar p}S_i(X;R)$ is projective \cite[Lemma 6.3.1]{GBF35} 
and $R$ is a Dedekind domain, every submodule is projective \cite[Section VII.5 and Theorem I.5.4]{CE}. 

To prove the first claim we consider the full singular chain complex $S_*(X;R)$ and let $S^{\mc U}_*(X;R)$ denote the subcomplex generated by the singular simplices whose supports are contained in some $U$ (i.e.\ $S^{\mc U}_*(X;R)=\im(\oplus S_*(U;R)\to S_*(X;R)$)). We will construct a chain map $h: S_*(X;R)\to S^{\mc U}_*(X;R)$ that splits the inclusion $S^{\mc U}_*(X;R)\to S_*(X;R)$. The construction is by induction over degree. We let $h$ be the identity on $S_0(X;R)=S^{\mc U}_0(X;R)$. Now suppose we have defined $h$ up through degree $i-1$ as a subdivision map (see \cite[Section 4.4.2]{GBF35}). For each $i$ simplex $\sigma$, the map $h$ has already been defined on $\bd \sigma$. If $\sigma$ is supported in some $U\in \mc U$, we let $h(\sigma)=\sigma$. Otherwise we can perform iterated generalized barycentric subdivision\footnote{See \cite[Section 16]{MK} for the definition of generalized barycentric subdivisions. Theorem 16.4 and its preceding lemmas of \cite{MK} have stronger hypotheses than the situation here, but also stronger conclusions. Nonetheless, the computations of \cite{MK} can be applied here. Here is a rough sketch: Pulling the open cover $\mc U$ back to a cover $\sigma^{-1}(\mc U)$ of the model simplex $\Delta^i$,  Step 1 of the proof of Lemma 16.3 shows that under a sufficiently iterated relative subdivision of $\Delta^i$ the simplices that intersect  $\bd \Delta^i$ can be made arbitrarily close to the simplices in the subdivision of $\bd \Delta^i$ coming from the induction; consequently each can be made to lie in some $\sigma^{-1}(U)$. Then further subdivision ensures that the ``interior'' simplices can be made $\sigma^{-1}(\mc U)$-small as in Step 4 of that proof.}
 of $\sigma$ relative to $h(\bd \sigma)$ to obtain a chain $\sigma'$ with $\bd \sigma'=h(\bd \sigma)$ and with each simplex of $\sigma'$ supported in some element of $\mc U$. Let $h(\sigma)=\sigma'$. Inductively, this gives a subdivision chain map $h: S_*(X;R)\to S^{\mc U}_*(X;R)$ that splits the inclusion. The restriction of $h$ to $I^{\bar p}S_i(X;R)$ has image in $I^{\bar p}S_i(X;R)$ by \cite[Corollary 6.3.10]{GBF35}, and since $I^{\bar p}S_i(X;R)\cap S^{\mc U}_i(X;R)=I^{\bar p}S^{\mc U}_i(X;R)$ by \cite[Proposition 6.5.2]{GBF35} we see that $h$ induces a map $I^{\bar p}S_i(X;R)\to I^{\bar p}S^{\mc U}_i(X;R)$. But $h$ remains the identity on $I^{\bar p}S^{\mc U}_i(X;R)\subset S^{\mc U}_i(X;R)$, so $h$ splits the inclusion 
$I^{\bar p}S^{\mc U}_i(X;R)\into I^{\bar p}S_i(X;R)$, as desired. 
\end{proof}

\begin{proposition}\label{P: flabby}
The sheaf $\mc I_{\bar p}\mc C^*$ of intersection cochains with coefficients in the \red{Dedekind domain} $R$ is flabby and hence soft and c-soft.
\end{proposition}
\begin{proof}
Let $W$ be an open subspace of $X$ and let $s\in \Gamma(W;\mc I_{\bar p}\mc C^*)$. It is easy to see from the definition that the sheaf of intersection cochains on $W$, which we shall denote $\mc I_{\bar p}\mc C^*_W$, is isomorphic to the restriction $\mc I_{\bar p}\mc C^*|_W$, and so
$\Gamma(W;\mc I_{\bar p}\mc C^*) =\Gamma(W;\mc I_{\bar p}\mc C^*_W)$. Since the presheaf of intersection cochains is conjunctive, there is a surjection  $\phi: I_{\bar p}S^*(W; R)\onto \Gamma(W;\mc I_{\bar p}\mc C^*_W)$. Choose  $\bar s\in I_{\bar p}C^*(W; R)$ such that $\phi(\bar s)=s$. Now consider the restriction $r:I_{\bar p}S^*(X; R)\to I_{\bar p}C^*(W; R)$, which is the dual of the inclusion $I^{\bar p}S_*(W; R)\into I^{\bar p}S_*(X; R)$. This inclusion splits \cite[Corollary 6.5.3]{GBF35}, and so $r$ is surjective. Let $\td s$ be such that $r(\td s)=\bar s$. Then $\td s$ induces an element of $\Gamma(X;\mc I_{\bar p}\mc C^*)$ that agrees with $s$ on $W$. 
\end{proof}

\subsubsection{\texorpdfstring{$\mc I_{\bar p}\mc C^*$}{IC} is quasi-isomorphic to \texorpdfstring{$\mc P_{D\bar p}$}{P}}

We next relate $\mc I_{\bar p}\mc C^*$ to $\mc I^{D\bar p}\mc S^*$ and $\mc P_{D\bar p}^*$ when
 $X$ is \emph{locally $(\bar p,R)$-torsion free}.

\begin{proposition}\label{P: IC is IS}
Let $\mf R$ be the constant coefficient system on $X-X^{n-1}$ with stalks in the \red{Dedekind domain $R$},  and let $\mc O$ be the $\Z$-orientation sheaf on $X-X^{n-1}$. Suppose $X$ is locally $(\bar p,R)$-torsion free.
Then 
the sheaf $\mc I_{\bar p}\mc C^*$  of $\bar p$-perversity intersection cochains with $R$ coefficients is quasi-isomorphic to the Deligne sheaf $\mc P^*_{D\bar p}(\mf R)$. Consequently, $\mc I_{\bar p}\mc C^*$  is also quasi-isomorphic to 
$\mc I^{D\bar p}\mc S^*(\mf R\otimes_{\Z} \mc O)$, the sheaf of $D\bar p$-perversity singular intersection chains with coefficients $\mf R\otimes \mc O$.
\end{proposition}
\begin{proof}
The last statement follows from the first by \cite[Theorem 3.6]{GBF23}. To prove the first, we employ \cite[Proposition 3.8]{GBF23}, which is the axiomatic characterization of Deligne sheaves with general perversities. These are analogous to the Goresky-MacPherson axioms \cite{GM2}, though as in \cite[Section V.2]{Bo} it is not necessary to make constructibility assumptions. 

Let $U_1=X-X^{n-1}$. When we restrict $I_{\bar p}S^*(X;R)$ to $I_{\bar p}S^*(U_1;R)$, we get simply $S^*(U_1;R)$, the ordinary singular cochain complex. It follows that the restriction of $\mc I_{\bar p}\mc C^*$ to $U_1$ is  the ordinary sheaf of singular cochains. In particular, it is quasi-isomorphic to the constant sheaf $\mf R_{U_1}$ on $U_1$  with stalks $R$. It is also evident that $\mc I_{\bar p}\mc C^i=0$ for $i<0$. Together these statements give the first axiom of \cite[Proposition 3.8]{GBF23} except for the bounded above condition. But, as in \cite[Remark V.2.7.b]{Bo}, we only need $\mc H^i(\mc I_{\bar p}\mc C^*)$ to be bounded above, and this will follow from the computations below.

Next, let $x\in Z$, where $Z$ is a singular stratum. Then 
\begin{align*}
H^i(\mc I_{\bar p}\mc C^*_x)=\dlim_{x\in U} \H^i(U;\mc I_{\bar p}\mc C^*)\cong \dlim_{x\in U} I_{\bar p}H^i(U;R).
\end{align*}
To compute the limit we are free to restrict $U$ to members of the cofinal family of distinguished neighborhoods of $x$ of the form $\R^{n-k}\times cL^{k-1}$. Using the Universal Coefficient Theorem and stratum preserving homotopy invariance \cite[Theorem 7.1.4 and Corollary 6.3.8]{GBF35},
\begin{align*}
I_{\bar p}H^i(U;R)&\cong  \Hom(I^{\bar p}H_i(U;R),R)\oplus \Ext(I^{\bar p}H_{i-1}(U;R),R)\\
&\cong  \Hom(I^{\bar p}H_i(cL;R),R)\oplus \Ext(I^{\bar p}H_{i-1}(cL;R),R).
\end{align*}
Since $X$ is locally $(\bar p,R)$-torsion free, $\Ext(I^{\bar p}H_{i-1}(cL;R),R)=0$ for $i=\codim(Z)-1-\bar p(Z)$, and we see that $H^i(\mc I_{\bar p}\mc C^*_x)$ vanishes for $i\geq \codim(Z)-1-\bar p(Z)$, i.e. for $i>D\bar p(Z)$, from the cone formula for intersection homology \cite[Theorem 6.2.13]{GBF35}. This is the second axiom.

Lastly, we need to show the attaching axiom holds. By \cite[page 50]{Bo}, this is equivalent to showing that 
$H^i(\mc I_{\bar p}\mc C^*_x)\cong \dlim \H^i(U-U\cap Z; \mc I_{\bar p}\mc C^*)$ when $x\in Z$ with $Z$ a singular stratum and when $i\leq D\bar p(Z)$. But again by stratified homotopy invariance and the intersection cohomology cone formula \cite[Proposition 7.1.5]{GBF35} we have that that $I_{\bar p}H^i(U;R)\cong  I_{\bar p}H^i(L;R)$, for $i\leq D\bar p(Z)$. More specifically, these isomorphisms are induced by the inclusion maps $L\into U-U\cap Z\into U$. Applying the naturality of intersection cohomology with respect to inclusion maps, and that $I_{\bar p}H^*(V;R)\cong \H^*(V;\mc I_{\bar p}\mc C^*)$ on open sets, 
it now follows from an easy argument that, in this degree range, $$H^i(\mc I_{\bar p}\mc C^*_x)\cong \dlim I_{\bar p}H^i(U;R)\cong \dlim I_{\bar p}H^i(U-U\cap Z;R)\cong \dlim \H^i(U-U\cap Z; \mc I_{\bar p}\mc C^*).$$ 
This completes our verification of the needed axioms.
\end{proof}

\begin{corollary}\label{C: IC is IS}
If $F$ is a field, $X$ is $F$-orientable, and we take as coefficients the constant sheaf with stalk $F$, then $\mc I_{\bar p}\mc C^*$ is quasi-isomorphic to $\mc I^{D\bar p}\mc S^*$ and is a $C$-ready and $C_c$-ready representative of $\mc P^*_{D\bar p}$. 
\end{corollary}

\section{Sheafification of the cup product}\label{S: sheaf cup}

An intersection cohomology cup product 
$$ I_{\bar p}H^i(X;F)\otimes I_{\bar q}H^j(X;F)\to I_{\bar r}H^{i+j}(X;F)$$
was introduced in \cite{GBF25} for coefficients in a field and for $\bar p,\bar q,\bar r$ satisfying $D\bar r\geq D\bar p+D\bar q$. In \cite[Chapter 7]{GBF35}, the cup product was extended to coefficients in a Dedekind domain $R$ assuming a locally torsion free condition (which is automatic for field coefficients). 
We now turn toward a sheaf theoretic description of this cup product.

\red{Let us first briefly review the definition of the cup product as given in  \cite[Section 7.2.2]{GBF35}. With $R$ coefficients mostly tacit, it consists of the composition
\begin{equation*}
\resizebox{.95\hsize}{!}{
\begin{diagram}
I_{\bar p}H^{*}(X)\otimes I_{\bar q}H^*(X)&\rTo&H^*(I_{\bar p}S^*(X)\otimes I_{\bar q}S^*(X))&\rTo^\Theta&H^*(\Hom(I^{\bar p}S_*(X)\otimes I^{\bar q}S_*(X),R))&\lTo^{\times^*}_\cong&I_{Q}H^*(X\times X)&\rTo^{d_2^*}& I_{\bar r}H^*(X).
\end{diagram}}
\end{equation*}
Here the map on the left takes the tensor product of cohomology classes to the product class, i.e.\ $[\alpha]\otimes [\beta]\to [\alpha\otimes \beta]$; the map $\Theta$ is induced at the cochain level so that $\Theta(\alpha\otimes \beta)(x\otimes y)=(-1)^{|x||\beta|}\alpha(x)\beta(y)$; the map $\times^*$ is the $\Hom$ dual of the chain level cross product, and $d_2^*$ is the $\Hom$ dual of the diagonal map. For ordinary singular cochains, this is precisely the classical definition of the cup product that can be found, for example, in Dold \cite[Section VII.8]{Dold}, as $\times$ is an Eilenberg-Zilber map in the sense of \cite[Section VI.12]{Dold}. The subtlety in the present context lies in demonstrating that there is a choice of perversity $Q$ on the product space so that all the maps remain well defined and so that the cross product induces a cohomology isomorphism.}

The following theorem says that the cup product in intersection cohomology can be realized sheaf theoretically by the Goresky-MacPherson sheaf pairing \cite[Section 5.2]{GM2} (see also \cite[Section V.9.C]{Bo}). 
 \green{As applications, we will prove in Section \ref{S: mult} that the intersection (co)homology de Rham theorem of Brasselet, Hector, and Saralegi \cite{BHS} (see also Saralegi \cite{Sa05} for general perversities) preserves multiplicative structures and that the singular cochain intersection cohomology cup product is isomorphic to the blown-up intersection cohomology cup product of Chataur, Saralegi-Aranguren, and Tanr\'e \cite{CST7} when they are both defined.}

\begin{theorem}\label{T: cup product}
Suppose that $R$ is a \red{Dedekind domain}, that $X$ is a locally $(\bar p,R)$-torsion free or locally $(\bar q,R)$-torsion free stratified pseudomanifold, and that $D\bar r\geq D\bar p+D\bar q$. There exists a  morphism $\td \cup: \mc I_{\bar p}\mc C^*\otimes \mc I_{\bar q}\mc C^*\to \mc I_{\bar r}\mc C^*$ in the derived category $D(X)$ such that the induced map on hypercohomology fits into the following commutative diagram, in which the vertical maps are isomorphisms induced by sheafification and the lower left horizontal map is induced by tensor product of sheaf sections:
\begin{diagram}[LaTeXeqno]\label{D: cup product}
I_{\bar p}H^*(X;R)\otimes I_{\bar q}H^*(X;R)&&\rTo^\cup && I_{\bar r}H^*(X;R)\\
\dTo^\cong&&&&\dTo_\cong\\
\H^*(X;\mc I_{\bar p}\mc C^*)\otimes \H^*(X;\mc I_{\bar q}\mc C^*)&\rTo  &
\H^*(X;\mc I_{\bar p}\mc C^*\otimes \mc I_{\bar q}\mc C^*)&\rTo^{\td \cup}&   \H^*(X;\mc  I_{\bar r}\mc C^*).
\end{diagram}
Furthermore, 
the morphism $\td \cup$ is a cohomological representative of the Goresky-MacPherson sheaf-theoretic intersection pairing  $\mc P_{\bar p}^*\otimes\mc P_{\bar q}^*\to \mc P_{\bar r}^*$ (see \cite[Section 5.2]{GM2} or \cite[Section V.9.C]{Bo}). In particular, at each $x\in X-X^{n-1}$, the map $\td \cup$ induces the product map $1\otimes 1\to 1$ of germs of $0$-cocycles. 
\end{theorem}
\begin{proof}
 
Suppose $D\bar r\geq D\bar p+D\bar q$ and that $Q=Q_{\{\bar p,\bar q\}}$ is the maximal $(\bar p,\bar q)$-compatible perversity as defined in \cite[7.2.4]{GBF35}. For any open $U\subset X$, we have the following diagram of maps; we leave the coefficients tacit:

\begin{diagram}[LaTeXeqno]\label{E: cup roof}
&&I_{\bar p}S^*(U)\otimes I_{\bar q}S^*(U)&&&&I_{Q}S^*(U\times U)\\
&\ldTo^=&&\rdTo^\Theta&&\ldTo^{\times^*}&&\rdTo^{d_2^*}\\
I_{\bar p}S^*(U)\otimes I_{\bar q}S^*(U)&& &
 &\Hom(I^{\bar p}S_*(U)\otimes I^{\bar q}S_*(U),R)&& &&      I_{\bar r}S^*(U).
\end{diagram}
The leftmost arrow is the identity, but we use it to get into the form of a roof diagram. The morphism $\Theta$ is defined by $\Theta(\alpha\otimes \beta)(x\otimes y)=(-1)^{|x||\beta|}\alpha(x)\beta(y)$ (see \cite[Section 7.2.2]{GBF35}), $\times^*$ is the $\Hom$ dual of the chain cross product (called $IAW$ in \cite{GBF35}), and $d_2^*$ is the $\Hom$ dual of the map induced by the diagonal $X\to X\times X$; this last map is allowable by \cite[Lemma 7.2.7]{GBF35} and the cross product is allowable by \cite[Theorem 6.3.19]{GBF35}. If $V\subset U\subset X$, then the entire diagram restricts by naturality, and we obtain a diagram of presheaves over $X$. Sheafifying yields a diagram of sheaves over $X$.
After sheafifying, the rightmost term is $\mc I_{\bar r}\mc C^*$, and the leftmost term is $\mc I_{\bar p}\mc C^*\otimes \mc I_{\bar q}\mc C^*$, which also represents $\mc I_{\bar p}\mc C^*\overset{L}{\otimes} \mc I_{\bar q}\mc C^*$ as each $\mc I_{\bar p}\mc C^*_x$ is torsion free and hence flat \cite[Proposition 4.20]{LAMMOD}. Furthermore, since $\mc I_{\bar p}\mc C^*$ and $\mc I_{\bar q}\mc C^*$ are each flabby, and hence soft and $c$-soft, $\mc I_{\bar p}\mc C^*\overset{L}{\otimes} \mc I_{\bar q}\mc C^*$ is soft and $c$-soft, and hence $C$-ready and $C_c$-ready, by \cite[Corollary II.16.31]{BR}.
 From the intersection homology K\"unneth theorem \cite[Theorem 6.4.14]{GBF35},
the cross product here is a chain homotopy equivalences, and so
 both arrows pointing down and to the left are quasi-isomorphisms. Thus the sheaf version of this diagram represents the concatenation of two morphisms in the derived category $D(X)$ and can be extended to a composition morphism $\td \cup\in \Mor_{D(X)}(\mc I_{\bar p}\mc C^*\otimes \mc I_{\bar q}\mc C^*,  \mc I_{\bar r}\mc C^*)$. 

We claim this morphism induces a  sheaf map compatible with the cup product. In particular, recall that for  a presheaf $P^*$, both the sheafification from $P^*$ to  $\ms S^*$ and the forming of injective resolutions  $\ms S^*\into \mc I^*$ are functorial between the appropriate chain homotopy categories. Applying $H^*$ to obtain a map $H^*(P^*(X))\to H^*(X;\mc I^*)=\H^*(X;\ms S^*)$ is also functorial. 
Thus, the above diagram yields a diagram of the following form, in which $\mc{\text{Sh}}$ is the sheafification functor over $X$.

\begin{equation*}
\resizebox{.95\hsize}{!}{
\begin{diagram}
I_{\bar p}H^{*}(X)\otimes I_{\bar q}H^*(X)&&&&H^*(I_{\bar p}S^{*}(X)\otimes I_{\bar q}S^*(X))&&&&I_{Q}H^*(X\times X)\\
\dTo&\rdTo&&\ldTo^=&&\rdTo^\Theta&&\ldTo^{\times^*}&&\rdTo^{d_2^*}\\
&&H^*(I_{\bar p}S^*(X)\otimes I_{\bar q}S^*(X))&&\dTo &
 &H^*(\Hom(I^{\bar p}S_*(X)\otimes I^{\bar q}S_*(X),R))&&\dTo && I_{\bar r}H^*(X)\\
\H^*(X;I_{\bar p}\mc C^*)\otimes \H^*(X;I_{\bar q}\mc C^*)&&&&\H^*(X;\mc I_{\bar p}\mc C^*\otimes \mc I_{\bar q}\mc C^*)&&&&\H^*(X; \mc{\text{Sh}}(I_{Q}S^*))\\
&\rdTo&\dTo&\ldTo^=&&\rdTo^\Theta&\dTo&\ldTo^{\times^*}&&\rdTo^{d_2^*}&\dTo\\
&&\H^*(X;\mc I_{\bar p}\mc C^*\otimes \mc I_{\bar q}\mc C^*)&& &
 &\H^*(X;\mc{\text{Sh}}(\Hom(I^{\bar p}C_*\otimes I^{\bar q}C_*,R)))&& &&  \H^*(X;\mc  I_{\bar r}\mc C^*).
\end{diagram}}
\end{equation*}

The maps on the left are simply the maps that take products of cohomology classes to the product cohomology class: $[\alpha]\otimes [\beta]\to[\alpha\otimes \beta]$. 
The composition across the top of the diagram is precisely the cup product as defined in \cite[Section 7.2.2]{GBF35}. The composition across the bottom represents the map on hypercohomology obtained from our morphism $\td \cup$. These bottom maps and cohomology groups technically depend on choices of injective resolutions, but since $\mc I_{\bar p}\mc C^*\otimes \mc I_{\bar q}\mc C^*$ and $\mc  I_{\bar s}\mc C^*$ are $C$-ready for any $\bar s$, we can assume $\H^*(X;\mc I_{\bar p}\mc C^*\otimes \mc I_{\bar q}\mc C^*)=H^*(\Gamma(X;\mc I_{\bar p}\mc C^*\otimes \mc I_{\bar q}\mc C^*))$ and $\H^*(X;\mc  I_{\bar s}\mc C^*)=H^*(\Gamma(X;\mc  I_{\bar s}\mc C^*))$, and then, as in Lemma \ref{L: rep}, the resulting cohomology map 
$$H^*(\Gamma(X;\mc I_{\bar p}\mc C^*))\otimes H^*(\Gamma(X;\mc I_{\bar q}\mc C^*))\to H^*(\Gamma(X;\mc I_{\bar p}\mc C^*\otimes \mc I_{\bar q}\mc C^*))\to H^*(\Gamma(X;\mc  I_{\bar r}\mc C^*))$$ along the bottom is independent of choices of resolutions of the other terms.
 The lefthand and righthand vertical morphisms are isomorphisms by Corollary \ref{C: cohomology}. 
  Thus this diagram implies the commutativity of \eqref{D: cup product}.

In addition, we observe that if $x\in U_1$, then restricting each $IS^*$ and $\mc I\mc C^*$ to a neighborhood $U$ of $x\in U_1$ yields complexes isomorphic to the usual singular cochain and singular cochain sheaf complexes  and the cup product becomes the standard cup product. In particular, the local cup product acts on locally constant $0$-cochains by multiplication, and it follows from Section \ref{R: uniqueness} and \cite[Section V.9.c]{Bo} that $\td \cup$ must represent the Goresky-MacPherson sheaf-theoretic intersection pairing  $\mc P_{\bar p}^*\otimes\mc P_{\bar q}^*\to \mc P_{\bar p}^*$, corresponding to some choices of quasi-isomorphisms $\mc P_{\bar s}^*\sim_{q.i.} \mc  I_{\bar s}\mc C^*$, for $\bar s=\bar p,\bar q,\bar r$. 
\end{proof}

\subsection{Applications: A multiplicative de Rham theorem for perverse differential forms and compatibility with the blown-up cup product}\label{S: mult}

\green{In this subsection, we apply Theorem \ref{T: cup product} to two applications. We first show that Saralegi's de Rham theorem for perverse differential forms \cite{BHS,Sa05} is compatible with the exterior product and cup product structures. Then we show compatibility between the singular intersection cohomology cup product and the blown-up intersection cohomology cup product of \cite{CST7}.}

\subsubsection{Perverse differential forms}

An approach to intersection (co)homology via perverse differential forms 
 first appeared in a paper by Brylinski \cite{Bry}, though he credits Goresky and MacPherson with the idea. Brylinski showed that for  perversities satisfying the Goresky-MacPherson conditions and on a Thom-Mather stratified space, 
a suitably defined de Rham intersection cohomology is $\Hom$ dual to intersection homology with real coefficients. Working on more general ``unfoldable spaces,'' Brasselet, Hector, and Saralegi later proved an analogous de Rham theorem in \cite{BHS}, showing that this result can be obtained by integration of forms on intersection chains, and this was extended to more general perversities by Saralegi in \cite{S1}. However, \cite{S1} contains an error in the case of perversities $\bar p$ satisfying $\bar p(Z)>\codim(Z)-2$ or $\bar p(Z)<0$ for some singular stratum $Z$. This error was corrected by Saralegi in \cite{Sa05}. \red{We refer the reader to these sources for definitions and background and prove here the following multiplicative de Rham theorem:}

\begin{theorem}\label{T: de Rham}
Let $X$ be an $\R$-oriented unfoldable stratified pseudomanifold, and suppose $D\bar r\geq D\bar p+D\bar q$.
 Let $\Omega_{\bar s}^*(X)$ be the complex of $\bar s$-perverse  differential forms on $X$. Then the following diagram commutes
\begin{diagram}[LaTeXeqno]\label{E: mult. de rham}
H^*(\Omega_{D\bar p}^*(X))\otimes H^*(\Omega_{D\bar q}^*(X))&\rTo^\wedge&H^*(\Omega_{D\bar r}^*(X))\\
\dTo^{\int\otimes \int}&&\dTo_\int \\
I_{\bar p}H^*(X;\R)\otimes I_{\bar q}H^*(X;\R)&\rTo^\cup&I_{\bar r}H^*(X;\R) .
\end{diagram}
\end{theorem}
Each vertical integration map in the statement of the theorem is technically a composition of the integration map $\int: H^*(\Omega_{D\bar p}^*(X))\to H^*(R_{\bar p}C^*(X;\R))$, where $R_{\bar p}C^*(X;\R)$ is the dual of the complex of ``liftable'' intersection chains \cite{Sa05}, with the inverse of the  isomorphism $H^*(I_{\bar p}S^*(X;\R))\to H^*(R_{\bar p}C^*(X;\R))$ induced by restriction.

The proof, which will occupy the remainder of this subsection, proceeds by first relating 
singular intersection (co)homology  to the (co)homology of liftable intersection chains and then relating the liftable chains to the perverse differential forms $\Omega^*$. 

We first establish a sheaf-theoretic version of liftable intersection cohomology and define chain and sheaf-theoretic versions of the cup product for liftable (co)chains. 

Let $X$ be an $\R$-oriented unfoldable stratified pseudomanifold. Let $\bar p$ be a general perversity. Let $\Omega_{\bar p}^*(X)$ be the complex of $\bar p$-intersection differential forms on $X$ (see \cite[Section 3]{Sa05}), and let $R^{\bar p}C_*(X,X_{D\bar p};\R)$ be the complex of liftable intersection chains as defined\footnote{Note: Saralegi uses the notation $\bar t-\bar p$ rather than $D\bar p$.} in Saralegi \cite[Section 2.3]{Sa05}. By \cite[Section 2.4]{Sa05}, the inclusion of $R^{\bar p}C_*(X,X_{D\bar p};\R)$ into Saralegi's relative singular intersection chain complex $S^{\bar p}C_*(X,X_{D\bar p};\R)$ (also defined in \cite{Sa05}) is a quasi-isomorphism, and $S^{\bar p}C_*(X,X_{D\bar p};\R)$ is isomorphic to our singular intersection chain complex  $I^{\bar p}S_*(X;\R)$ by \cite[Appendix A]{GBF20}. 
As for singular intersection chains, we can form the sheaf $\mc R_{\bar p}\mc C^*$ of liftable intersection cochains as the sheafification of $U\to R_{\bar p}C^*(U;\R):=\Hom(R^{\bar p}C_*(U, U_{D\bar p};\R),\R)$. This is well-defined assuming that we choose as the unfolding of each $U$ the restriction of a fixed unfolding of $X$. Since $R^{\bar p}C_*(U, U_{D\bar p};\R)$ is quasi-isomorphic to $S^{\bar p}C_*(U;\R)$, the restriction morphism $\mc I_{\bar p}\mc C^*\to \mc R_{\bar p}\mc C^*$ is a quasi-isomorphism.

\begin{lemma}
$\mc R_{\bar p}\mc C^*$ is flabby, $C$-ready, and $C_c$-ready.
\end{lemma}
\begin{proof}
We first observe that the presheaf $U\to R_{\bar p}C^*(U;\R)$ is conjunctive. The proof is the same as for $I_{\bar p}C^*(U;R)$ in Proposition \ref{P: conj}. To invoke  \cite[Proposition 6.5.2]{GBF35}, we use that the argument in \cite{GBF35} involves rewriting chains in terms of other chains whose simplices all come from linear subdivisions of the simplices in the original chains. By \cite{BHS} a simplex from a linear subdivision of a liftable simplex is liftable, and so \cite[Proposition 6.5.2]{GBF35} continues to hold replacing intersection chains with liftable intersection chains. The remainder of the proof of Proposition \ref{P: conj} is even simpler since we now employ field coefficients. 

The rest of the proof is now identical to that of Proposition \ref{P: flabby}, using the field coefficients to provide the needed splitting.
\end{proof}

It follows immediately that we have a 
 commutative diagram of isomorphisms

\begin{diagram}
I_{\bar p}H^*(X;\R)=H^*(I_{\bar p}S^*(X;\R))&\rTo&\H^*(X;\mc I_{\bar p}\mc C^*)=H^*(\Gamma(X;\mc I_{\bar p}\mc C^*))\\
\dTo&&\dTo\\
H^*(R_{\bar p}C^*(X;\R))&\rTo&\H^*(X;\mc R_{\bar p}\mc C^*)=H^*(\Gamma(X;\mc R_{\bar p}\mc C^*)).
\end{diagram}
That the bottom is an isomorphism follows from our existing knowledge that the other three sides are.

We now utilize these isomorphisms to define cup products 

$$\cup_R: H^*(R_{\bar p}C^*(X;\R))\otimes H^*(R_{\bar q}C^*(X;\R))\to H^*(R_{\bar r}C^*(X;\R))$$
and
$$\td \cup_R: H^*(\Gamma(X;\mc R_{\bar p}\mc C^*\otimes \mc R_{\bar q}\mc C^*))\to H^*(\Gamma(X;\mc R_{\bar r}\mc C^*))$$
for $D\bar r\geq D\bar p+D\bar q$ in terms of our existing cup products. 

Consider the cube\footnote{We leave the $\R$ coefficients tacit here and for the remainder of this subsection.}

{\footnotesize
\begin{diagram}[LaTeXeqno]\label{E: liftable cube}
H^*(I_{\bar p}C^*(X))\otimes H^*(I_{\bar q}C^*(X))&&\relax\rTo^{\cup}&&H^*(I_{\bar r}C^*(X))&&&\\
&\relax\rdTo&&&\relax\vLine&\relax\rdTo&&\\
\relax\dTo&&H^*(R_{\bar p}C^*(X))\otimes H^*(R_{\bar q}C^*(X))&\relax\rDashto^{\cup_R}&\relax\HonV&&H^*(R_{\bar r}C^*(X))&\\
&&\relax\dTo&&\relax\dTo&&&\\
H^*(\Gamma(X;\mc I_{\bar p}\mc C^*\otimes \mc I_{\bar q}\mc C^*))&\relax\hLine&\relax\VonH&\relax\rTo^{\td \cup}&H^*(\Gamma(X;\mc I_{\bar r}\mc C^*))&&\relax\dTo&\\
&\relax\rdTo&&&&\relax\rdTo&&\\
&&H^*(\Gamma(X;\mc R_{\bar p}\mc C^*\otimes \mc R_{\bar q}\mc C^*))&&\relax\rDashto^{\td \cup_R}&&H^*(\Gamma(X;\mc R_{\bar r}\mc C^*)).&
\end{diagram}}

\noindent We have just seen in Theorem \ref{T: cup product} that the back commutes, and the sides commute by the naturality of sheafification. \red{Furthermore, the horizontal maps on the sides are isomorphisms; note that tensor product $\mc R_{\bar p}\mc C^*\otimes \mc R_{\bar q}\mc C^*$ is soft by \cite[Corollary II.16.31]{BR} since $\mc R_{\bar p}\mc C^*$ is soft (since flabby implies soft) and torsion-free and so the bottom left horizontal map is a $C$-ready representation}. Thus we can complete the cube with the dashed arrows to obtain a commutative diagram, and we thus \emph{define} the cup products $\cup_R$ and $\td \cup_R$ on liftable intersection cohomology and its sheafification. This approach allows us to avoid consideration of whether all the maps in Diagram \eqref{E: cup roof} are consistent with liftability. We furthermore observe that since the bottom back map is induced by a morphism in the derived category, the quasi-isomorphisms between the $\mc I\mc C^*$ and $\mc R\mc C^*$ complexes imply that the hypercohomology map on the bottom front edge of the cube is also induced by a morphism in the derived category.

Next we connect the cup product for liftable intersection cochains with the exterior product of perverse differential forms using the integration map $\int: \Omega_{\bar p}^*(X)\to R_{D\bar p}C^*(X)$. See \cite{Sa05} for the precise definition, but the idea is basically the same as for manifolds: if $\alpha$ is an $i$-form, then $\int \alpha$ acts on a simplex by integrating $\alpha$ over it (or its interior) if the dimension of the simplex is $i$ and by $0$ otherwise. Taking (co)homology yields an isomorphism $\int: H^*(\Omega_{\bar p}^*(X))\to H^*(R_{D\bar p}C^*(X))$ by \cite[Theorem 3.2.2]{Sa05}. 

Restriction of liftable $\bar p$-forms provides a sheaf $U\to \Omega_{\bar p}^*(U)$. We shall denote this sheaf by $\td \Omega^*_{\bar p}$. By \cite[Proposition 3.1]{BHS}, this sheaf is fine (\cite{BHS} deals only with Goresky-MacPherson perversities, but the argument works for any perversity). Furthermore, integration of liftable intersection chains induces a sheaf map $\td\int: \td \Omega^*_{\bar p}\to \mc R_{D \bar p}\mc C^*$. It follows from Saralegi's isomorphisms \cite{Sa05} that this must be a quasi-isomorphism. Furthermore, by \cite[Section 3.1.2]{S1} the wedge product of forms induces a map $\td\wedge: \td \Omega_{\bar p}^*\otimes \td \Omega_{\bar q}^*\to\td \Omega_{\bar r}^*$ for $\bar p+\bar q\leq \bar r$.

Now, consider the diagram 
\begin{diagram}[LaTeXeqno]\label{E: mult derham}
\mc R_{D\bar p}\mc C^*\otimes \mc R_{D\bar q}\mc C^*&\lTo^{\td\int\otimes \td\int}&
\td \Omega_{\bar p}^*\otimes \td \Omega_{\bar q}^*&\rTo^{\td\wedge} &\td \Omega_{\bar p+\bar q}^*&\rTo^{\td\int} \mc R_{\bar D(\bar p+\bar q)}\mc C^*.
\end{diagram}
Since $\td\int\otimes \td\int$ is a quasi-isomorphism, this diagram represents an element of $\Mor_{D(X)}(\mc R_{D\bar p}\mc C^*\otimes \mc R_{D\bar q}\mc C^*, \mc R_{D(\bar p+\bar q)}\mc C^*)$. We claim  this map is $\td \cup_R$ on hypercohomology.

\begin{lemma}\label{L: wedge is cupR}
The composition $\td\int\circ \td\wedge\circ (\td\int\otimes \td\int)^{-1}$ equals
$$\td \cup_R: H^*(\Gamma(X;\mc R_{\bar p}\mc C^*\otimes \mc R_{\bar q}\mc C^*))\to H^*(\Gamma(X;\mc R_{\bar r}\mc C^*)).$$
\end{lemma}
\begin{proof}
As noted in Section \ref{R: uniqueness}, it is sufficient to show that these maps induce the same map 
$H^0(\mc R_{\bar p}\mc C^*_x)\otimes H^0(\mc R_{\bar q}\mc C^*_x)\to H^0(\mc R_{\bar r}\mc C^*_x)$
at each $x\in U_1=X-X^{n-1}$.  

Over $U_1$, the sheaves $\mc R_{\bar p}\mc C^*$ restrict to the sheaves $\mc C^*$ of ordinary singular cochains. Thus, as already noted for $\mc I_{\bar p}\mc C^*$ at the end of the proof of Theorem \ref{T: cup product}, the local cup product on cohomology $\cup=\cup_R$ at $x\in U_1$ restricts to multiplication of constant $0$-cocycles. 

Similarly,  over $U_1$, the complex  $\td \Omega_{\bar p}^*$ restricts to the standard sheaf of differential forms. In particular, at each $x\in U_1$, an element of $H^0((\td \Omega_{\bar p}^*)_x)$ is represented by a constant function in a neighborhood of $x$ and the exterior product corresponds to multiplication of these constant functions. If $f$ is such a $0$-form, $\int f$ acts on the standard positively-oriented generating singular $0$-simplex at $x$ by evaluation $f(x)$.   

Putting these facts together, the lemma follows. 
\end{proof}

Now we connect the cup product of liftable cochains to the wedge product of perverse forms.

\begin{lemma}
The following cube commutes (with $\R$ coefficients tacit):

{\footnotesize
\begin{diagram}[LaTeXeqno]\label{E: liftable cube2}
H^*(R_{\bar p}C^*(X))\otimes H^*(R_{\bar q}C^*(X))&&\relax\rTo^{\cup_R}&& H^*(R_{\bar r}C^*(X)) &&&\\
&\relax\luTo^{\int\otimes \int}   &&&\relax\vLine&\relax\luTo^\int&&\\
\relax\dTo^h&&H^*(\Omega_{D\bar p}^*(X))\otimes H^*(\Omega_{D\bar q}^*(X))&\relax\rTo^\wedge&\relax\HonV&&H^*(\Omega_{D\bar r}^*(X))&\\
&&\relax\dTo_f&&\relax\dTo_k&&&\\
\H^*(X;\mc R_{\bar p}\mc C^*\otimes \mc R_{\bar q}\mc C^*)&\relax\hLine&\relax\VonH&\relax\rTo^{\td\cup_R}&\H^*(X;\mc R_{\bar r}C^*)&&\relax\dTo_g&\\
&\relax\luTo^{\td\int\otimes \td\int}&&&&\relax\luTo^{\td\int}&&\\
&&\H^*(X;\td \Omega_{D\bar p}^*\otimes \td \Omega_{D\bar q}^*)&&\relax\rTo^{\td\wedge}&&\H^*(X;\td \Omega_{D\bar r}^*)&
\end{diagram}}

\end{lemma}

\begin{proof}
We have just seen in Lemma \ref{L: wedge is cupR} that the bottom square commutes.
The front and the left and right sides commute  due to the functoriality of sheafification and cohomology. The back square commutes by definition in consideration of our discussion surrounding the cube \eqref{E: liftable cube}.
The vertical morphisms on the right are isomorphisms by the fineness/flabbiness of the sheaves.  So, to get commutativity of the top square, we have
\begin{align*}
\int\circ\wedge&= \int\circ g^{-1}\circ\td \wedge\circ f\\
&=k^{-1}\circ\td\int\circ\td\wedge\circ f\\
&=k^{-1}\circ\td \cup_R\circ \td\int\otimes \td\int\circ f\\
&=k^{-1}\circ \td \cup_R\circ h\circ\int\otimes\int\\
&=\cup_R \circ \int\otimes \int.
\end{align*}
\end{proof}

\begin{proof}[Proof of Theorem \ref{T: de Rham}.] 
The theorem now follows from gluing together the top faces of the two cubes \eqref{E: liftable cube} and \eqref{E: liftable cube2}.
\end{proof}

\subsubsection{Blown-up intersection cohomology}

\green{Essentially the same proof as for Theorem \ref{T: de Rham} shows that the singular intersection cohomology cup product is isomorphic to the cup product of the \emph{blown-up intersection cohomology} of Chataur-Saralegi-Tanr\'e \cite{CST7}. We outline the modifications to the argument, but as this setting requires a certain amount of technology that is beside our main goal, we refer the interested reader to \cite{CST,CST7,CST20} for the necessary background and definitions. }

\green{
 For notation, some of which will be different from the references to avoid conflicts with our other notations, let $\td N^*_{\bar p}(X;R)$ denote the blown-up complex of $X$ \cite[Definition 3.2]{CST7}, let $\ms H_{\bar p}^*(X;R)$ be the corresponding blown-up intersection cohomology, and let $I^{\bar p}s_*(X;R)$ denote the subcomplex of $I^{\bar p}S_*(X;R)$ obtained by considering only chains made of filtered simplices, i.e. singular simplices $\sigma:\Delta^i\to X$ such that $\sigma^{-1}(X^j)$ is a face of $\Delta^i$ for each $j$. There are maps $\td N^*_{D\bar p}(X;R)\xr{\chi}\Hom(I^{\bar p}s_*(X;R),R)\leftarrow \Hom(I^{\bar p}S_*(X;R),R)=I_{\bar p}S^*(X;R)$. The map $\chi$ is constructed in \cite[Section 13]{CST7}. The other map is restriction.  The map $\chi$ is a quasi-isomorphism on locally $(\bar p,R)$-torsion free stratified pseudomanifolds by \cite[Theorem F]{CST7}. The restriction is a quasi-isomorphism as the inclusion $I^{\bar p}s_*(X;R)\into I^{\bar p}S_*(X;R)$ is a quasi-isomorphism by \cite[Theorem B]{CST-inv} and so the dual is also a quasi-isomorphism by the naturality of the Universal Coefficient Theorem and the Five Lemma. We let $\td \chi$ be the composite isomorphism $\ms H_{D\bar p}^*(X;R)\xr{\chi} H^*(\Hom(I^{\bar p}s_*(X;R),R))\xleftarrow{\cong} I_{\bar p}H^*(X;R)$. 
}

\begin{theorem}\label{T: CST cup}
\green{Let $R$ be a Dedekind domain and $X$ a locally $(\bar p,R)$-torsion free stratified pseudomanifold. Suppose $D\bar r\geq D\bar p+D\bar q$.  Then the following diagram commutes and the vertical maps are isomorphisms:}
\begin{diagram}
\ms H_{D\bar p}^*(X;R)\otimes \ms H^*_{D\bar q}(X;R)&\rTo^\cup&\ms H^*_{D\bar r}(X;R)\\
\dTo^{\td\chi\otimes \td\chi}&&\dTo_{\td \chi} \\
I_{\bar p}H^*(X;R)\otimes I_{\bar q}H^*(X;R)&\rTo^\cup&I_{\bar r}H^*(X;R) .
\end{diagram}
\end{theorem}
\begin{proof}
\green{The proof follows the same steps as that of Theorem \ref{T: de Rham}, replacing 
$\Omega_{\bar p}^*(X)$ and $R^{\bar p}C_*(X,X_{D\bar p};\R)$ respectively with 
$\td N^*_{\bar p}(X;R)$ and $I^{\bar p}s_*(X;R)$; we note only the necessary modifications. }

\green{
By \cite[Proposition 3.9]{CST20}, the sheaf $\mc I_{\bar p}c^*$ induced by the presheaf $U\to I_{\bar p}s^*(U;R):=\Hom(I^{\bar p}s_*(U;R);R)$ is soft and flat. So by Lemma \ref{L: C-ready}, it is $C$-ready. For the bottom left horizontal map in the analogue of the cube \eqref{E: liftable cube} we have that $\mc I_{\bar p}c^*\otimes \mc I_{\bar q}c^*$ is soft and hence $C$-ready again by \cite[Corollary II.16.31]{BR}, as $\mc I_{\bar p}c^*$ is torsion free; the referenced corollary is stated for $R$ a PID but its proof applies for $R$ a Dedekind domain as torsion-free modules over Dedekind domains are flat. }

\green{The sheafification $\mc N_{\bar p}^*$ of $\td N^*_{\bar p}(X;R)$ is flat and soft by \cite[Proposition 2.6]{CST20}, and $\chi$ induces a quasi-isomorphism $\mc N_{\bar p}^*\to \mc I_{D\bar p}c^*$ by \cite[Theorem F]{CST7}. Furthermore, the tensor product of two such quasi-isomorphisms is a quasi-isomorphism by the K\"unneth Theorem, which applies as all of our presheaves are torsion free and so flat. }

\green{
A version of Lemma \ref{L: wedge is cupR} follows from the observation that on the manifold $U_1$ we have identically $I_{\bar p}S^*(U_1;R)= I_{\bar p}s^*(U;R)=\td N_{D\bar p}^*(U_1;R)=S^*(U;R)$, the ordinary complex of singular cochains. Furthermore, the cup products all reduce to the standard cup product. This is obvious for $I_{\bar p}s^*(U;R)$; for $\td N^*(U_1;R)$ it follows by unwinding the definitions --- see the remarks at the end of \cite[Section 2]{CST20}. }

\green{Given these remarks, we conclude as in the proof of Theorem \ref{T: de Rham}.}
\end{proof}

\section{Classical duality via sheaf maps}\label{S: compatibility}

Throughout this section we choose our ground ring to be a field $F$, and we suppose  $X$ to be a compact $F$-oriented $n$-dimensional stratified pseudomanifold. We also let $\bar p$ and $\bar q=D\bar p$ be two complementary perversities, i.e.\ $\bar p+\bar q=\bar t$. 

Let  $$\alpha\to (-1)^{|\alpha|n}\alpha\cap \Gamma$$ be 
the Poincar\'e duality isomorphism from $I_{\bar p}H^i(X;F)$ to $I^{\bar q}H_{n-i}(X;F)$ given by the signed cap product with the fundamental class $\Gamma$; see \cite[Remark 8.2.2]{GBF35} for an explanation of the sign.  We will show that this Poincar\'e duality map is compatible with a hypercohomology map induced by a morphism  $\mathbb{O}\in \Mor_{D(X)}(I_{\bar p}\mc C^*, I^{\bar q}\mc S^*)$. 

As noted in Remark \ref{R: intro cap} of the Introduction, it is not clear that this can be achieved via a sheaf theoretic cap product induced by the cap product at the level of chains. However, since $I_{\bar p}\mc C^*$ and $I^{\bar q}\mc S^*$ are each quasi-isomorphic to the perversity $\bar q$ Deligne sheaf, it follows from Section \ref{R: uniqueness} that $\Mor_{D(X)}(I_{\bar p}\mc C^*,I^{\bar q}\mc S^*)\cong  F^m$, where $m$ is the number of connected components of $U_1=X-X^{n-1}$. In fact, up to equivalence in the derived category, each morphism is determined by the maps $F\cong H^0(I_{\bar p}\mc C^*_{x_j})\to H^0(I^{\bar q}\mc S^*_{x_j})\cong F$ as the points $\{x_j\}_{j=1}^m$ run over representative points, one in each  of the connected components of $U_1$. We define  $\mathbb{O}\in \Mor_{D^+}(I_{\bar p}\mc C^*,I^{\bar q}\mc S^*)$ to be the morphism that corresponds to the maps $H^0(I_{\bar p}\mc C^*_{x_j})\to H^0(I^{\bar q}\mc S^*_{x_j})$ that take the standard generator $0$-cocycle $1$ in each  $H^0(I_{\bar p}\mc C^*_{x_j})$ to the element of  $H^0(I^{\bar q}\mc S^*_{x_j})\cong I^{\bar q}H_n(X,X-x; F)$ consistent with the local orientation class determined by the given orientation on $X$. Since the maps over the $x_j$ are non-zero, $\mathbb{O}$ is a quasi-isomorphism. We will refer to $\mathbb{O}$ as the \emph{quasi-isomorphism consistent with the orientation}; the definition of $\mathbb{O}$ is discussed in more detail in Section \ref{S: O}.

The following theorem says that the Poincar\'e duality isomorphism given by the signed cap product with the fundamental class is compatible with the map on sheaf hypercohomology induced by $\mathbb{O}$.

\begin{theorem}\label{T: MAIN}
Let $X$ be a compact $F$-oriented $n$-dimensional stratified pseudomanifold, and let $\bar p,D\bar p$ be complementary perversities. The following diagram of isomorphisms commutes, where the vertical maps are induced by the sheafification of presheaf sections into sheaf sections,  the bottom map is induced by the quasi-isomorphism consistent with the orientation, and  the top map is the Poincar\'e duality map given by the signed cap product with the fundamental class $\Gamma$ determined by the orientation:

\begin{diagram}[LaTeXeqno]\label{D: theorem}
I_{\bar p}H^i(X; F)&\rTo^{(-1)^{in}\cdot\cap \Gamma}& I^{D\bar p}H_{n-i}(X; F)\\
\dTo^\sigma&&\dTo_{\sigma'}\\
\H^i(X;\mc I_{\bar p}\mc C^*) &\rTo^{\mathbb{O}} &  \H^i(X;\mc I^{D\bar p}\mc S^*).
\end{diagram}
\end{theorem}

When $X$ is a manifold, $\mc I_{\bar p}\mc C^*$ and $\mc I^{D\bar p}\mc S^*$ reduce to the sheaf complexes $\mc C^*$ and $\mc S^*$ of ordinary singular cochains and chains, and we obtain the following corollary:

\begin{corollary}\label{C: man duality}
Let $M$ be a compact $F$-oriented $n$-dimensional manifold. The following diagram of isomorphisms commutes, where the vertical maps are induced by the sheafification of presheaf sections into sheaf sections,  the bottom map is induced by the quasi-isomorphism consistent with the orientation, and  the top map is the Poincar\'e duality map given by the signed cap product with the fundamental class $\Gamma$ determined by the orientation:

\begin{diagram}
H^i(M; F)&\rTo^{(-1)^{in}\cdot\cap \Gamma}& H_{n-i}(M; F)\\
\dTo^\sigma&&\dTo_{\sigma'}\\
\H^i(M;\mc C^*) &\rTo^{\mathbb{O}} &  \H^i(M;\mc S^*).
\end{diagram}
\end{corollary}

The proof of Theorem \ref{T: MAIN} is given in Section \ref{S: PROOF} following some further preliminaries in Section \ref{S: O}.

\subsection{Orientations and canonical products}\label{S: O}

We continue to assume that $X$ is a compact $n$-dimensional $F$-oriented stratified  pseudomanifold.

As observed in Section \ref{R: uniqueness}, morphisms amongst Deligne sheaves with appropriate perversities, as well as maps from appropriate Deligne sheaves to the shifted Verdier dualizing sheaf complex $\D^*[-n]$, are determined uniquely (in the derived category) by the maps they induce on $H^0$ of stalks in  $U_1=X-X^{n-1}$. Similarly, various products are determined by their behavior on $H^0$ of these stalks. Thus if one wishes to describe these morphisms precisely, and in particular the maps they induce on hypercohomology, it is necessary to first ``orient'' these sheaves by choosing specific generators of these $H^0$ groups.

For the intersection chain sheaves, $H^0(\mc I^{\bar p}\mc S^*_x)\cong I^{\bar p}H_n(X,X-x;F)$, which for each $x\in U_1$ is canonically isomorphic to $H_n(X,X-x)$ by excision arguments. Thus the sheaf of such germs $\mc H^0(\mc I^{\bar p}\mc S^*)$ over $U_1$ is in fact the $F$-orientation sheaf over $U_1$. By definition of the $F$-orientation of $X$, each such stalk is assigned a preferred generator, which we may identify with $1\in F$  and think of as represented at each point by an $n$-chain generating the local orientation class. When $\bar p\geq \bar 0$, it is shown in \cite[Theorem 8.1.18]{GBF35}  that $X$ possesses 
  a fundamental class $\Gamma\in I^{\bar p}H_n(X;F)$, and the germs at $x\in U_1$ of any representative cycle for $\Gamma$ will also represent the preferred local generators.

For the sheaves of intersection cochains, we have over $U_1$ the cochain $1$ that takes the value $1\in F$ on any singular $0$-simplex of $U_1$. The germ of $1$ at $x\in U_1$, which we will denote $1_x$, restricts to the standard generator of $H^0(\mc I_{\bar q}\mc C^*_x)\cong I_{\bar q}H^0(x;F)\cong H^0(x;F)$.

By Proposition \ref{P: IC is IS}, we know that $\mc I_{\bar p}\mc C^*$ and $\mc I^{D\bar p}\mc S^*$ are quasi-isomorphic, and so by Section \ref{R: uniqueness} each quasi-isomorphism between them is determined by the map it induces $\mc H^0(\mc I_{\bar p}\mc C^*)|_{U_1}\to \mc H^0(\mc I^{D\bar p}\mc S^*)|_{U_1}$. As $X$ is orientable, these are each constant local systems with stalk $F$. For each perversity $\bar p$, we define $\mathbb{O}$ to be the morphism $\mc I_{\bar p}\mc C^*\to \mc I^{D\bar p}\mc S^*$ such that  $\mathbb{O}(1_x)$ is the local orientation class in $H^0(\mc I^{\bar p}\mc S^*_x)$ for each $x\in U_1$.

Similarly, by Section \ref{R: uniqueness}, if $\bar p\leq \bar t$, there are unique morphisms $\mc I^{\bar p}\mc S^*\to \D^*[-n]$ and $\mc I_{D\bar p}\mc C^*\to \D^*[-n]$ for each morphism of local systems on $U_1$. In fact, $\mc H^0(\D[-n])|_{U_1}=\mc H^{-n}(\D^*)|_{U_1}$ is also isomorphic to the $F$-orientation sheaf of $U_1$; see \cite[Section V.7.3]{Bo}. Since we have not chosen a particular geometric description of $\D^*$ (e.g. in terms of chains or cochains) but have rather defined $\D^*$ via a fixed but unidentified chosen resolution of the constant sheaf (see Section \ref{S: Verdier complex}), we are free to ``orient'' $\D^*$ by choosing and fixing an isomorphism $\mc H^0(\mc I^{\bar t}\mc S^*)|_{U_1}\to \mc H^0(\D^*[-n])|_{U_1}$. This determines a local generator for each $H^0(\D^*_x[-n])$ as the image of the local orientation in $H^0(\mc I^{\bar t}\mc S^*_x)$. This isomorphism of local systems over $U_1$ then determines a morphism  $\mathbb{K}:\mc I^{\bar t}\mc S^*\to \D^*[-n]$  and, by composition with $\mathbb{O}$, a unique morphism $\mathbb{L}:\mc I_{\bar 0}\mc C^*\to \D^*[-n]$ such that the following triangle commutes:
\begin{diagram}[LaTeXeqno]\label{D: O triangle}
\mc I_{\bar 0}\mc C^*\\
\dTo^{\mathbb{O}}& \rdTo(2,1)^{\mathbb{L}}&\D^*[-n].\\
\mc I^{\bar t}\mc S^*&\ruTo(2,1)_{\mathbb{K}}&
\end{diagram}

Next we ``orient'' certain pairings. Recall again from  Section \ref{R: uniqueness} that maps $\mc I^{\bar p}\mc S^*\otimes \mc I^{\bar q}\mc S^*\to \mc I^{\bar r}\mc S^*$ and $\mc I_{D\bar p}\mc C^*\otimes \mc I_{D\bar q}\mc C^*\to \mc I_{D\bar r}\mc C^*$, for $\bar p+\bar q\leq \bar r$, are determined uniquely by their induced morphisms 
 $\mc H^0(\mc I^{\bar p}\mc S^*)\otimes \mc H^0(\mc I^{\bar q}\mc S^*) \to \mc H^0(\mc I^{\bar r}\mc S^*)$ and $\mc H^0(\mc I_{D\bar p}\mc C^*)\otimes \mc H^0(\mc I_{D\bar q}\mc C^*)\to \mc H^0(\mc I_{D\bar r}\mc C^*)$ over $U_1$. We have already seen in Section \ref{S: sheaf cup} that the map corresponding to the local product $1_x\otimes 1_x\to 1_x$ is the sheafification $\td \cup$ of the intersection cohomology cup product, which represents the Goresky-MacPherson sheaf product. For the intersection chain sheaves, we define a representative of the  Goresky-MacPherson sheaf product by letting  $\td \psi: \mc I^{\bar p}\mc S^*\otimes \mc I^{\bar q}\mc S^*\to \mc I^{\bar r}\mc S^*$ be the unique morphism in $D(X)$ that takes the tensor product of local orientation classes at points $x\in U_1$ to a local orientation class.

Finally, there is one more ``orientation'' issue that requires an explicit isomorphism. 
It seems to be well-known  that for a locally compact space $\D^*[-n]$ is quasi-isomorphic to the sheaf of singular chains $\mc S^*$ (giving $\mc S^*$ the same indexing convention we use for complexes of sheaves of intersection chains). For a pseudomanifold $X$, we will in fact rederive this result below; see Remark \ref{R: Verdier is singular}. In particular then, for compact connected $X$, $\H^{0}(X;\D^*)\cong \H^n(X;\D^*[-n])\cong \H^n(X;\mc S^*)\cong H_0(X;F)\cong F$. This isomorphism will play a role below, but once again the construction of $\D^*$ we have used does not seem to provide a natural choice for these isomorphisms. It will turn out that a useful choice for this isomorphism will be forced on us within the proof of Theorem \ref{T: MAIN}; see the section labeled ``\textbf{Constants}'' beginning on page \pageref{P: constants}. We label this chosen isomorphism $\ell: \H^{0}(X;\D^*)\xr{\cong} F$.

\subsection{Proof of Theorem \ref{T: MAIN}}\label{S: PROOF}

Throughout the proof, we continue to assume that 
  $X$ is a compact $F$-oriented $n$-dimensional stratified pseudomanifold for a fixed field $F$ and write $\bar q=D\bar p$ so that  $\bar p+\bar q=\bar t$. We will also assume at first that $X$ is normal and connected, which implies $U_1=X-X^{n-1}$ is connected and thus $I^{\bar 0}H_n(X;F)\cong F$; see \cite[Lemma 2.6.3 and Theorem 8.1.18 ]{GBF35}. We complete the proof for general $X$ starting on page \pageref{P: not normal}.

With these assumptions, we consider the following diagram in which $i+j=n$. 

\begin{equation}\label{D: MAIN}
\resizebox{.95\hsize}{!}{
\begin{diagram}
I_{\bar p}H^i(X; F) & && &\rTo^{(-1)^{in}\cdot\cap \Gamma} &&&&I^{\bar q}H_j(X; F)\\
&\rdTo(4,2)^\nu && &I& &&\ldTo(4,2)^\kappa \ldTo(2,4)^r&\\
\dTo^\sigma& II&&&\Hom(I_{\bar q}H^j(X; F),F) && &&\\ 
 &&& \ruTo^{\sigma^*}&&\luTo^{e''}&VI & & \dTo_{\sigma'}\\
H^i(\mc I_{\bar p}\mc C^*(X))&\rTo^\chi&\Hom(\H^j(X;\mc I_{\bar q}\mc C^*),F)&&V&&H^{-j}(\Hom(I_{\bar q}S^*(X; F),F))&& \\
\dTo^=&III&\uTo(1,2)^e \qquad IV &\luTo(2,2)^{e'}&&\ruTo^h &  VII\dTo_{ \tau}&VIII & \\
\H^i(X;\mc I_{\bar p}\mc C^*)&\rTo^{\hat \Phi}&\H^i(X;\mc D\mc I_{\bar q}\mc C^*[-n])&\rTo^{\Lambda[-n]} &\H^i(X;\ms L\mc I_{\bar q}\mc C^*[-n])&\rTo^{ \eta[-n]}&\H^i(X;\ms P\mc I_{\bar q}\mc C^*[-n])&\lTo^{ \rho}&\H^i(X;I^{\bar q}\mc S^*)\\
\end{diagram}}
\end{equation}

The specific labeled maps and groups involved will be described in what follows.
We will consider the polygons in the diagram one at a time, considering the extent to which they commute and showing that each map is an isomorphism. 
In our first pass, we will show that all smaller polygons commute up to a constant. We will then show starting on page \pageref{P: constants} that with the proper choice of the isomorphism $\ell$ (see Section \ref{S: O}),   the signs will ``cancel out'' to provide the exact commutativity around the outside of the diagram and the composition of maps along the bottom will be induced by the sheaf morphism $\mathbb{O}$. 

Since the maps around the outside of the diagram will become the maps in the statement of Theorem \ref{T: MAIN}, we take care to ensure that these are all induced by chain maps. However, as the groups in the interior of the diagram are not necessarily cohomology groups of chain complexes, we will typically treat groups and maps in the interior degree by degree, not necessarily induced by chain maps.

For a sheaf complex $\mc A^*$, we will occasionally use the notation $\mc A^*(U):=\Gamma(U;\mc A^*)$. 

 \paragraph{VI.} 
For triangle VI of diagram \eqref{D: MAIN}, the map  $r$ is induced by the chain map  $I^{\bar q}S_*(X;F)\to \Hom^*(\Hom^*(I^{\bar q}S_*(X;F),F),F)=\Hom(I_{\bar q}S^*(X;F),F)$ defined so that if $x\in I^{\bar q}S_*(X;F)$ and $\alpha\in I_{\bar q}S^*(X;F)$, then  $f(x)(\alpha)=(-1)^{|\alpha|}\alpha(x)$. This map is 
 discussed in more detail in Appendix \ref{Appendix}, where it is shown that it is a degree $0$ chain map and that it induces an isomorphism on homology when $I^{\bar q}H_j(X;F)$ is finitely generated in all degrees  and trivial for sufficiently large $j$; these assumptions hold here by  \cite[Corollary 6.3.40]{GBF35} and \cite[Lemma 8.1.16]{GBF35}.

The map $e''$ is just the universal coefficient isomorphism as applied to $I_{\bar q}S^*(X;F)$, which is free since it is a module over the field $F$.  We define the map $\kappa$ to be the composition of $r$ and $e''$. Explicitly,  for a cycle $x$ representing an element of $I^{\bar q}H_j(X;F)$, the homomorphism $\kappa(x)$ acts on a cocycle $\alpha$ representing an element in $I_{\bar q}H^j(X; F)$ by $\kappa(x)(\alpha)=(-1)^{|\alpha|}\alpha(x)$.  

The triangle commutes by definition, and $\kappa$ is an isomorphism as $r$ and $e''$ are. 

\paragraph{I.}  We let $\nu: I_{\bar p}H^i(X; F) \to \Hom(I_{\bar q}H^j(X; F),F)$ be defined so that if $\alpha\in I_{\bar p}H^i(X; F)$ and $\beta\in I_{\bar q}H^j(X; F)$, then\footnote{As motivation for this sign, we note that it is consistent with interchanging the order of $\alpha$ and $\beta$ in the cup product and then applying $(-1)^n\cdot \cap \Gamma$, as the degree of $\beta\cup \alpha$ is $n$.} $\nu(\alpha)(\beta)=(-1)^{ij+n}\aug((\beta\cup \alpha)\cap \Gamma)$, where $\aug: I^{\bar t}H_0(X; F)\to F$ is the augmentation map that takes any $\bar t$-allowable singular $0$-simplex to $1$. This is well-defined as $\beta\cup\alpha\in I_{\bar 0}H^n(X;F)$. Since we have assumed $X$ connected and normal, the augmentation is an isomorphism (see Corollary 5.1.9 and Proposition 6.2.9 of \cite{GBF35}).

Then by the properties of cup and cap products \cite[Section 7.3.9]{GBF35}, 
\begin{align*}
\nu(\alpha)(\beta)&=
(-1)^{ij+n}\aug((\beta\cup \alpha)\cap \Gamma)\\
&=(-1)^{ij+n}\aug(\beta \cap (\alpha\cap \Gamma))\\
&=(-1)^{ij+n}\beta(\alpha\cap \Gamma)\\
&=(-1)^{ij+n+j}\kappa(\alpha\cap\Gamma)(\beta)\\
&=(-1)^{ij+n+j+in}((\kappa\circ (-1)^{in}\cdot \cap \Gamma)(\alpha))(\beta).
\end{align*}

Thus triangle $I$ commutes up to the sign $(-1)^{ij+n+j+in}=(-1)^{i(j+n)+n+j}=(-1)^{i(i)+n+j}=(-1)^{i+n+j}=(-1)^{2n}=1$, i.e.\ it commutes exactly. 

We note that $\cap \Gamma$ is an isomorphism by Poincar\'e duality \cite[Theorem 8.2.4]{GBF35}, $\kappa$ is an isomorphism as above, and thus $\nu$ is also an isomorphism.

\paragraph{III.}

Recall that  we assume that we have chosen a fixed  arbitrary $c$-soft resolution of the constant sheaf with stalk $F$ in order to define $\D^*$ and that we have fixed in Section \ref{S: O} (though not yet specified) the isomorphism $\ell: \H^{0}(X;\D^*)\cong F$.

We will demonstrate square III commutes up to a constant that depends on the choice of $\ell$. We will return to these specific choices when we consider  the issue of precise commutativity more carefully below, starting on page \pageref{P: constants}.

Let $\Phi:\mc I_{\bar p}\mc C^*\otimes \mc I_{\bar q}\mc C^*\to \D^*[-n]$ be the composition of the sheaf cup product $\td \cup:\mc I_{\bar p}\mc C^*\otimes \mc I_{\bar q}\mc C^*\to \mc I_{\bar 0}\mc C^*$ in the derived category (see Theorem \ref{T: cup product}) with the  morphism $\mathbb{L}: \mc I_{\bar 0}\mc C^*\to \D^*[-n]$ of Section \ref{S: O}.

Since $\D^*$ as defined in \cite{Bo} is injective by \cite[Corollary V.7.6]{Bo}, $\Phi$ can be represented by an actual degree $0$ chain map of sheaf complexes  \cite[Section V.5.17]{Bo} (which we will also call $\Phi$). The map $\Phi$ induces a degree $0$ adjoint
$\hat \Phi:  \mc I_{\bar p}\mc C^*\to \SHom(\mc I_{\bar q}\mc C^*,\D^*[-n])=\mc D\mc I_{\bar q}\mc C^*[-n]$. To define $\hat \Phi$ precisely, given a section $s\in \mc I_{\bar p}\mc C^i(U)$, $\hat \Phi(s)$  must be  an element of $\SHom^i(\mc I_{\bar q}\mc C^*,\D^*[-n])(U)=\Hom^i(\mc I_{\bar q}\mc C^*|_U,\D^*[-n]|_U)$, i.e.\ a degree $i$ homomorphism of sheaf complexes (though it need not be a chain map). If $s\in \mc I_{\bar p}\mc C^i(U)$, then we can define such a homomorphism  $\hat \Phi(s):\mc I_{\bar q}\mc C^*|_U\to \D^*[-n]|_U$ as follows:  if $t\in \mc I_{\bar q}\mc C^*(V)$ for $V\subset U$, let $s|_V\td\otimes t$ be the image of $s|_V\otimes t$ under sheafification in  $(\mc I_{\bar p}\mc C^*\otimes \mc I_{\bar q}\mc C^*)(V)$ (which we recall is not necessarily equal to $\mc I_{\bar p}\mc C^*(V)\otimes \mc I_{\bar q}\mc C^*(V)$), and 
then let $\hat \Phi(s)(t)=\Phi(s|_V\td\otimes t)\in \D^*[-n](V)$.  With this definition, $\hat \Phi(s)$ commutes with restrictions, since $\Phi$ is a sheaf morphism, and so $\hat \Phi(s)$ provides a homomorphism of sheaf complexes over $U$ as desired. Furthermore,  restrictions of $s$ commute with restrictions of $\hat\Phi(s)$, and so $\hat \Phi$ is a homomorphism of sheaves. In fact $\hat \Phi$ is a degree $0$ chain map: 
if $s\in  \mc I_{\bar p}\mc C^i(U)$ and $t\in \mc I_{\bar q}\mc C^*(V)$ for $V\subset U$, then

\begin{align*}
\hat \Phi(ds) (t)&=\Phi(ds|_V\td\otimes t)\\
&=\Phi(d(s|_V\td \otimes t)-(-1)^{i}(s|_V\td \otimes dt))\\
&=d\Phi(s|_V\td\otimes t) -(-1)^{i}\Phi(s|_V\td\otimes dt) \\
&= d(\hat \Phi(s)(t))   -(-1)^{i}\hat\Phi(s)(d(t))\\
&=(d \circ  \hat \Phi(s)-(-1)^{i}\hat\Phi(s)\circ d)(t)\\
&=((d\hat\Phi)(s))(t).\\
\end{align*}

\begin{lemma}
$\hat \Phi$ is a quasi-isomorphism. 
\end{lemma}
\begin{proof}
By Corollary \ref{C: IC is IS},  $\mc I_{\bar p}\mc C^*$ is quasi-isomorphic to $\mc I^{\bar q}\mc S^*$, and so by \cite[Theorem V.7.8.i]{Bo}, also 
$\mc D\mc I_{\bar q}\mc C^*[-n]$ is quasi-isomorphic to $\mc D\mc I^{\bar p}\mc S^*[-n]$, which is quasi-isomorphic to $\mc I^{\bar q}\mc S^*$ by \cite[Theorem 4.3]{GBF23}. Therefore $\mc I_{\bar p}\mc C^*$ and $\mc D\mc I_{\bar q}\mc C^*[-n]$
 are both  quasi-isomorphic to the perversity $\bar q$ Deligne sheaf. So to show that $\hat \Phi$ is a quasi-isomorphism, it suffices by Section \ref{R: uniqueness} to show that $\hat \Phi$ is a quasi-isomorphism at some $x\in U_1=X-X^{n-1}$, since the only non-quasi-isomorphism between these sheaves induces the zero map on the cohomology of the stalk at such an $x$.

So suppose $x\in U_1$, and let $U$ be a Euclidean neighborhood of $x$.  As observed in the proof of Theorem \ref{T: cup product}, on $U$ the intersection cohomology cup product reduces to the standard cup product, and in particular the stalk map $H^*(I_{\bar p}\mc C^*_x)\otimes H^*(I_{\bar q}\mc C^*_x)\to H^*(I_{\bar 0}\mc C^*_x)$ corresponds to multiplication $F\times F\to F$ in degree $0$, while $H^*(I_{\bar 0}\mc C^*_x)\to H^*(\D[-n]_x)$ corresponds to an automorphism of $F$ in degree $0$ (recall $\D^*[-n]|_{U_1}$ is quasi-isomorphic to the constant sheaf $\mc F_{U_1}$). Since the stalk cohomology for each of these sheaves is isomorphic to the section cohomology on $U$, and in fact there is a cofinal system of such euclidean neighborhoods that is essentially constant on cohomology, corresponding statements hold over $U$. 

Now, over $U$, $\hat \Phi$  takes a cycle representing a non-zero element $\alpha\in H^0(I_{\bar p}S^*(U))\cong F$ to an element of $\Hom^0(\mc I_{\bar q}\mc C^*|_U,\D^*[-n]|_U)$ representing a degree $0$ chain map. In particular, $\hat \Phi(\alpha)$ induces a map $I_{\bar q}S^*(U)\to \D^*[-n](U)$ such that for a cocycle $\beta\in I_{\bar q}S^0(U)$, $\hat \Phi(\alpha)(\beta)=\Phi(\alpha\otimes \beta)$, which if $\beta$ is also cohomologically non-trivial yields a non-zero element of $H^0(U;\D^*[-n])\cong H^0(U;F)\cong F$. Since all maps involved commute with restrictions and, in fact, yield cohomology isomorphism on restriction to smaller euclidean neighborhoods about $x$, it follows that $\hat \Phi(\alpha)$ induces a quasi-isomorphism $\mc I_{\bar q}\mc C^*|_U\to \D[-n]|_U$ and, on restriction, a quasi-isomorphism $\mc I_{\bar q}\mc C^*|_V\to \D[-n]|_V$ for all Euclidean $V$ with $x\in V\subset U$. In particular, $\hat \Phi(\alpha)$ is not chain-homotopically trivial, and thus it is not a boundary in the $\Hom^*$ complex, and so it represents a non-zero element of $\H^0(U;\SHom(\mc I_{\bar q}\mc C^*,\D[-n]))\cong \H^0(U;\mc D\mc I_{\bar q}\mc C^*[-n])\cong \H^0(U;\mc C^*)\cong H^0(U;F)\cong F$. Again because these formulas all commute with restriction and yield isomorphisms on all cohomology groups over Euclidean neighborhoods of $x$, it follows that $\hat \Phi$ yields a cohomology isomorphism at each stalk $x\in U_1$. 
\end{proof}

Now, suppose $i+j=n$. We construct a diagram
\begin{diagram}[LaTeXeqno]\label{D: cup duality}
&&\Hom(\H^j(X;\mc I_{\bar p}\mc C^*),F)\\
\H^i(X;\mc I_{\bar p}\mc C^*)&\ruTo(2,1)^\chi &\uTo^e\\
 &\rdTo(2,1)^{\hat \Phi}&H^i(\Hom(\mc I_{\bar q}\mc C^*,\D^*[-n])). 
\end{diagram}
 We define the vertical map $e$ to take a cocycle $\eta\in \Hom^i(\mc I_{\bar q}\mc C^*,\D^*[-n])$, which corresponds to a degree $i$ chain map  $\mc I_{\bar q}\mc C^*\to \D^*[-n]$,  and output the induced map $\H^j(X;\mc I_{\bar q}\mc C^*)=H^j(\Gamma(X;\mc I_{\bar q}\mc C^*))\to H^n(\Gamma(X;\D^*[-n]))\xrightarrow{\ell} F$. This depends on our previous choice of isomorphism  $\ell: H^n(\Gamma(X;\D^*[-n]))=H^0(\Gamma(X;\D^*))\cong F$. 
We define the top diagonal $\chi$ to  take a cohomology class represented by a cocycle $\alpha\in \mc I_{\bar p}\mc C^*(X)$ to the composition  $\beta\to \Phi(\alpha\td\otimes\beta) \xrightarrow{\ell} F$, where $\alpha\td\otimes\beta$ stands for the image in $(\mc I_{\bar p}\mc C^*\otimes \mc I_{\bar q}\mc C^*)(X)$ of $\alpha\otimes \beta\in \mc I_{\bar p}\mc C^*(X)\otimes \mc I_{\bar q}\mc C^*(X)$ under sheafification. As $\alpha$ and $\beta$ are both assumed to be cocycles (representing cohomology classes) we do have that  $\Phi(\alpha\td\otimes\beta)$ represents an element of $\H^n(X;\D^*[-n])$.

\begin{lemma}
Diagram \eqref{D: cup duality} commutes. 
\end{lemma}
\begin{proof}
Let $\alpha$ be a cocycle in $\mc I_{\bar p}\mc C^i(X)$. By definition, and the assumption that $\alpha$ is a cocycle,  $\hat \Phi(\alpha)$ is a degree $i$ chain map $\mc I_{\bar q}\mc C^*\to\D^*[-n]$, and so it induces a map on homology $H^*(\mc I_{\bar q}\mc C^*(X))\to H^{*+i}(\D^*[-n](X))$ (in fact, this is a map of hypercohomology, as $\mc I_{\bar q}\mc C^*$ and $\D^*$ are both $C$-ready).
Therefore, if $\beta$ is a cycle in $\mc I_{\bar q}\mc C^j (X)$, then $((e\hat \Phi))(\alpha)(\beta)$ is precisely the image of $\Phi(\alpha\td \otimes \beta)$ under  $\ell: H^n(\Gamma(X;\D^*[-n]))\cong F$. But this is precisely what $\chi(\alpha)$ does to $\beta$ by definition.
\end{proof}

This proposition suffices to demonstrate the commutativity of 
 square III in diagram \eqref{D: MAIN} since the lefthand map is an identity, as $\mc I_{\bar p}\mc C^*$ is flabby, and $\H^i(X;\mc D^*\mc I_{\bar q}\mc C^*[-n])=H^i(\Gamma(X;\SHom^*(\mc I_{\bar q}\mc C^*,\D^*)[-n]))=H^i(\Hom^*(\mc I_{\bar q}\mc C^*,\D^*[-n]))$ by the flabbiness of  $\mc D^*\mc I_{\bar q}\mc C^*=\SHom(\mc I_{\bar q}\mc C^*,\D^*)$. 

$\hat \Phi$ is an isomorphism since it is induced by a quasi-isomorphism. We will see below in proving the commutativity of II that $\chi$ is an isomorphism, hence so is $e$.

\paragraph{II.} We continue to let $\nu,\chi$ be as defined above. We let $\sigma$ stand for the map that takes cochains in $I_{\bar p}C^*(X; F)$  to their sections in $\mc I_{\bar p}\mc C^*(X)$. Similarly, $\sigma^*$ is induced by the Hom-dual of such a map on perversity $\bar q$ intersection cochains.

Now, let $\alpha\in I_{\bar p}S^i(X; F)$ and $\beta\in I_{\bar q}S^j(X; F)$
be  cocycles. Then, by definition, $\sigma^*\chi\sigma(\alpha)$ acts on $\beta$ by taking it to $\Phi(\sigma(\alpha)\td\otimes \sigma(\beta))$, which is  the composition on $\sigma(\alpha)\td\otimes \sigma(\beta)$ of the sheaf cup product $\td \cup$, which has image in  $\H^n(X;\mc I_{\bar 0}\mc C^*)$, with the isomorphisms $\H^n(X;\mc I_{\bar 0}\mc C^*)\xr{\mathbb{L}} \H^n(X;\D^*[-n])\xr{\ell} F$ - see Section \ref{S: O}. But by Theorem \ref{T: cup product}, the sheaf cup product on $\sigma(\alpha)\td\otimes \sigma(\beta)$ is precisely the image in $\H^n(X;\mc I_{\bar 0}\mc C^*)$ under sheafification of the singular cup product $\alpha \cup \beta$. Therefore $\sigma^*\chi\sigma(\alpha)$ is equal to the composite map 
$$ I_{\bar q}H^j(X; F)\xrightarrow{\alpha\cup \cdot} I_{\bar 0}H^n(X; F) \xrightarrow{\sigma} \H^n(X;\mc I_{\bar 0}\mc C^*) \xrightarrow{\mathbb{L}} \H^n(X;\D^*[-n])\xrightarrow{\ell} F.$$

On the other hand, recall that  $\nu(\alpha)(\beta)$ is the image of $(-1)^{ij}\beta\cup \alpha\in I_{\bar 0}H^n(X; F)$ under the cap product $(-1)^n\cdot\cap \Gamma: I_{\bar 0}H^n(X; F)\to I^{\bar t}H_0(X; F)$, which is an isomorphism by Poincar\'e duality, and the augmentation isomorphism $\aug:I^{\bar t}H_0(X; F)\to F$. Thus 
$\nu(\alpha)$  is the composition 
$$ I_{\bar q}H^j(X; F)\xrightarrow{\cdot \cup (-1)^{ij}\alpha} I_{\bar 0}H^n(X; F) \xrightarrow{(-1)^n\cdot \cap\Gamma} I^{\bar t}H_0(X; F)\xrightarrow{\aug} F.$$

 Since  $\alpha\cup\beta=(-1)^{ij}\beta\cup \alpha$, and since the maps $I_{\bar 0}H^n(X; F) \xrightarrow{\sigma} \H^n(X;\mc I_{\bar 0}\mc C^*) \xrightarrow{\mathbb{L}} \H^n(X;\D^*[-n])$ and $I_{\bar 0}H^n(X; F) \xrightarrow{(-1)^n\cdot \cap\Gamma} I^{\bar t}H_0(X; F)\xrightarrow{\aug}F$
are all isomorphisms of modules isomorphic to $F$, we see that $\nu(\alpha)$ and  $\sigma^*\chi\sigma(\alpha)$ differ only up to a unit in $F$. More particularly, they differ by  the unit representing the automorphism

\begin{equation}\label{E: osign}
F\xrightarrow{\aug^{-1}} I^{\bar t}H_0(X; F)\xrightarrow{(-1)^n(\cdot\cap\Gamma)^{-1}}I_{\bar 0}H^n(X; F)\xrightarrow{\sigma} \H^n(X;\mc I_{\bar 0}\mc C^*)\xrightarrow{\mathbb{L}}\H^n(X;\D^*[-n])\xr{\ell} F.
\end{equation}

This automorphism clearly depends on the choices we have made. We will return to this issue below on page \pageref{P: constants}. For now, we observe that  these automorphisms do not depend on $\alpha$, and we conclude that square II commutes up to a constant that does not depend on $i$ or $j$.

$\sigma$ is an isomorphism by \cite[Section 6]{GBF25}, and hence so is the dual $\sigma^*$. We saw in the section on rectangle I that $\nu$ is an isomorphism. It follows that $\chi$ is an isomorphism.

 \paragraph{IV.}

  Next we get to work on  triangle IV of diagram \eqref{D: MAIN}. We begin by defining $\ms L\mc I_{\bar q}\mc C^*$. 

We let $\ms L\mc I_{\bar q}\mc C^*$ (or $\ms L^*\mc I_{\bar q}\mc C^*$ if we need to emphasize the indexing of this sheaf complex) be the sheaf complex $U\to \Hom(\Gamma_c(U; \mc I_{\bar q}\mc C^*), F)$.  Since $\mc I_{\bar q}\mc C^*$ is $c$-soft, this is a sheaf by \cite[Proposition V.1.2]{IV}. In \cite{Bo}, this is essentially Borel's initial definition of the Verdier dual, except that we do not need to take a tensor product first of $\mc I_{\bar q}\mc C^*$ with a flat c-soft resolution of the constant sheaf since $\mc I_{\bar q}\mc C^*$ is already c-soft; see \cite[Section V.7.7]{Bo}, particularly the penultimate paragraph. This allows us to represent this version of the Verdier dual complex more simply, and we use the different notation to emphasize this point.

Note that $\ms L\mc I_{\bar q}\mc C^*$ is flabby as well since $\Gamma_c(U; \mc I_{\bar q}\mc C^*)\to \Gamma_c(X; \mc I_{\bar q}\mc C^*)$ is always injective and $\Hom(\cdot, F)$ is exact for $F$ a field. Also $H^i(\ms L\mc I_{\bar q}\mc C^*_x)$ vanishes for $|i|$ sufficiently large by the local intersection cohomology computations.
Thus $\ms L\mc I_{\bar q}\mc C^*$ is $C$-ready and $C_c$-ready. 

Now we turn to defining a map $\Lambda: \mc D\mc I_{\bar q}\mc C^*\to \ms L\mc I_{\bar q}\mc C^*$.

 We continue to employ the fixed isomorphism $\ell: \H^0(X;\D^*)\to F$ chosen earlier. Furthermore, if we consider the definition of $\D^*$ given in \cite[Section V.7]{Bo} as $\D^*(U)=\Hom^*(\Gamma_c(U;\mc K^*),F)$ for some fixed $c$-soft resolution $\mc F\to \mc K^*$ of the constant sheaf with stalk $F$, we see that $\D^i=0$ if $i>0$. If we think of $F$ as a chain complex that is non-trivial only in dimension $0$, then there is a chain map  $\gamma: \Gamma_c(X;\D^*)\to F$ that takes elements of $\Gamma_c(X;\D^i)$ to $0$ if  $i<0$ and that takes elements of $\Gamma_c(X;\D^0)$, which are all cycles, to the images of their cohomology classes under  $\ell: H^0(\Gamma_c(X;\D^*))\to F$; it is easy to observe that $\gamma$ is a chain map.

Using $\gamma$, we construct the chain map $\Lambda$ from $\mc D\mc I_{\bar q}\mc C^*=\SHom(\mc I_{\bar q}\mc C^*,\D^*)$ to $\ms L\mc I_{\bar q}\mc C^*$. 
We define $\Lambda$ on the open set $U\subset X$ as follows:
If $f\in \mc D^k\mc I_{\bar q}\mc C^*(U)=\Hom^k(\mc I_{\bar q}\mc C^*|_U,\D^*|_U)$, let $\lambda_f:\Gamma_c(U;\mc I_{\bar q}\mc C^*)\to \Gamma_c(U;\D^*)$ be the induced degree $k$ homomorphism  on compactly supported sections. Let $j_U:\Gamma_c(U;\D^*)\to \Gamma_c(X;\D^*)=\Gamma(X;\D^*)$ be induced by inclusion, using that $X$ is compact, and let $\Lambda(f)=\gamma j_U \lambda_f$. Observe that $\Lambda(f)\in\Hom^k(\Gamma_c(U;\mc I_{\bar q}\mc C^*),F).$ 

\begin{lemma}
$\Lambda$ is a degree $0$ chain map of sheaf complexes.
\end{lemma}
\begin{proof}
We first observe that $\Lambda$ is a sheaf map in each degree. Indeed, let $f\in \mc D\mc I_{\bar q}\mc C^k(U)$, and suppose $V\subset U$. Then $f$ restricts to a morphism $f|_V:\mc I_{\bar q}\mc C^*|_V\to \D^*|_V$ and  $\Lambda(f|_V)=\gamma j_V\lambda_{f|_V}$, where  $\lambda_{f|_V}$ is the map $\Gamma_c(V;\mc I_{\bar q}\mc C^*)\to \Gamma_c(V;\D^*)$ induced by $f|_V$. On the other hand, the restriction of $\Lambda(f)$ to $V$ yields the map $\Lambda(f)|_V:\Gamma_c(V;\mc I_{\bar q}\mc C^*)\to F$, which is, by definition of $\ms L^*\mc I_{\bar q}\mc C^*$, the composition $\Gamma_c(V;\mc I_{\bar q}\mc C^*)\xrightarrow{i} \Gamma_c(U;\mc I_{\bar q}\mc C^*)\xrightarrow{\Lambda(f)} F$. But we clearly have a commutative diagram
\begin{diagram}
\Gamma_c(V;\mc I_{\bar q}\mc C^*)&\rTo^i&\Gamma_c(U;\mc I_{\bar q}\mc C^*) \\
\dTo^{\lambda_{f|_V}}&&\dTo^{\lambda_f}\\
\Gamma_c(V;\D^*)&\rTo& \Gamma_c(U;\D^*)\\
\dTo^{j_V}&\ldTo^{j_U}&\\
\Gamma(X;\D^*)&\rTo^\gamma&F.
\end{diagram}
Here the composition $\gamma j_V\lambda_{f|_V}$ is simply $\Lambda(f|_V)$, while the compositions $\gamma j_U\lambda_{f}i$ is $\Lambda(f)|_V$. 
So it follows that $\Lambda$ computes with restrictions and so is a map of sheaves in each degree.

Next we show that  $\Lambda$ is a chain map. To see this, we compute over $U$. Let $f\in \Hom^i(\mc I_{\bar q}\mc C^*|_U,\D^*|_U)$ and recall that $f$ is not necessarily a chain  map. Then by definition\footnote{In an attempt to avoid both confusion and cluttered notation, when working with $\Hom^*(A^*,B^*)$, we will use $df$ to denote the boundary of $f\in \Hom^*(A^*,B^*)$ and $d\circ f$ to denote the composition of $f$ followed by the boundary in $B^*$. Similarly $f\circ d$ is the composition of the boundary of $A$ with $f$.}, $df=d\circ f-(-1)^{i}f\circ d$, and we see $\Lambda(df)=\gamma j_U \lambda_{d\circ f-(-1)^{i}f\circ d}=\gamma j_U\lambda_{d\circ f}-(-1)^{i}\gamma j_U\lambda_{f\circ d}$. Since any boundary in $\Gamma_c(U;\D^*)$ is taken to $0$ in $F$ by the chain maps $\gamma j_U$, it follows that $\Lambda(df)=-(-1)^{i}\gamma j_u\lambda_{f\circ d}$ (which can be non-zero because $f$ is not necessarily a chain map). On the other hand, $d(\Lambda(f))=d(\gamma j_U\lambda_f)=\gamma j_U(d\lambda_f)=\gamma j_U(d\circ\lambda_f-(-1)^{i}\lambda_f\circ d)$, since $\gamma$ and $j_U$ are chain maps. But again, $\gamma j_U$ is $0$ on coboundaries, so $d(\Lambda(f))=-\gamma j_U((-1)^{i}\lambda_f\circ d)$. But it is evident that $(-1)^{i}\lambda_f\circ d=\lambda_{(-1)^{i}f\circ d}$ as applied to elements of $\Gamma_c(U,\mc I_{\bar q}\mc C^*)$. Thus $\Lambda$ is a chain map.
\end{proof}

\begin{lemma}
$\Lambda$ is a quasi-isomorphism.
\end{lemma}
\begin{proof}
We already know abstractly that $\mc D\mc I_{\bar q}\mc C^*$ and $\ms L^*\mc I_{\bar q}\mc C^*$ are quasi-isomorphic by \cite[Section V.7.B]{Bo}. Then since $\mc D\mc I_{\bar q}\mc C^*[-n]$ is quasi-isomorphic to the Deligne sheaf $\mc P^*_{\bar p}$, so is $\ms L^*\mc I_{\bar q}\mc C^*[-n]$.
Since $X$ is connected and normal, the set of morphisms in the derived category between these sheaf complexes is isomorphic (as a set) to $F$ by Section \ref{R: uniqueness}, and all of these are quasi-isomorphisms except for the trivial $0$ morphism. So to show that $\Lambda$ is a quasi-isomorphism it suffices to show that $\Lambda$ is not trivial as a map in the derived category. For this, it suffices to check that there is an open $V\subset  X$ such that  $\Lambda$ induces a non-trivial morphism on the hypercohomology over $V$, for if $\Lambda$ is the $0$ map (up to quasi-isomorphism), then so is $\Lambda|_V$, which would have to induce a trivial map on cohomology.

On $U_1$, $\mc I_{\bar q}\mc C^*=\mc C^*$, the sheaf of singular cochains, which is a resolution of the constant sheaf $\mc F_{U_1}$. By \cite[Section V.7.3]{Bo}, we know that $\D^*[-n]|_{U_1}$ is a resolution of the orientation sheaf, which is also isomorphic to $\mc F_{U_1}$. Thus applying \cite[Lemma V.9.13]{Bo},  $\mc I_{\bar q}\mc C^*|_{U_1}$ and $\D^*[-n]|_{U_1}$ are isomorphic in the derived category of sheaves on $U_1$, i.e. they are quasi-isomorphic on $U_1$ say by a degree zero quasi-isomorphism $g'$.
Since $\D^*$ is injective (and so also $\D^*[-n]|_{U_1}$ is injective), there is by \cite[Section V.5.16]{Bo} an actual degree $0$ chain map $\mc I_{\bar q}\mc C^*|_{U_1}\to \D^*[-n]|_{U_1}$ giving this quasi-isomorphism, and this corresponds to  a degree $-n$ chain map and quasi-isomorphism $g: \mc I_{\bar q}\mc C^*|_{U_1}\to \D^*|_{U_1}$. Since $g$ is a chain map, it represents a cycle in $\SHom^{-n}(\mc I_{\bar q}\mc C^*,\D^*)(U_1)\cong(\mc D\mc I_{\bar q}\mc C^*)^{-n}(U_1)$.

Furthermore, since both $\mc I_{\bar q}\mc C^*|_V$ and  $\D^*[-n]|_{V}$ are $C_c$-ready for any open $V\subset U_1$, $g$ induces an isomorphism $\H_c^*(V;\mc I_{\bar q}\mc C^*)\to \H^{*-n}_c(V;\D^*)$ and, in particular, $\lambda_{g|_V}$ induces a cohomology isomorphism $\H_c^n(V;\mc C^*)\xrightarrow{\cong} \H^0_c(V;\D^*)$. If $V$ is homeomorphic to $\R^n$, then $\H^0_c(V;\D^{*})\cong H_0^c(V;F)\cong F$
and furthermore $\gamma j_{V}$ induces an isomorphism $\H^0_c(V;\D^*)\to F$ since $j_V: \H^0_c(V;\D^*)\to  \H^0_c(X;\D^*)$ corresponds to the isomorphism $H_0^c(V;F)\to H_0^c(X;F)$ induced by inclusion and $\gamma$ induces our fixed isomorphism $\H^0_c(X;\D^*)=H^0(\Gamma(X;\D^*))\to F$. 
 It follows that $\Lambda(g|_V)=\gamma j_V\lambda_{g|_V}$ is not homotopic to zero as a degree $-n$ chain map $\Gamma_c(V;\mc C^*)\to F$, and so $\Lambda(g|_V)\neq 0\in\H^{-n}(V;\ms L\mc I_{\bar q}\mc C^*)$. Thus $\Lambda$ cannot be $0$ on hypercohomology over $V$, and so $\Lambda$ is non-trivial as claimed.
\end{proof}

\begin{lemma}
Triangle IV of diagram \eqref{D: MAIN} commutes, where $e$ and $\Lambda$ are as above and $e'$ takes a cycle in $\ms L^i\mc I_{\bar q}\mc C^*[-n](X)=\ms L^{-j}\mc I_{\bar q}\mc C^*(X)$, which corresponds to a chain map in $\Hom^{-j}(\Gamma_c(X;\mc I_{\bar q}\mc C^*),F)$, to the induced map on cohomology $\Hom(H^j(\Gamma_c(X;\mc I_{\bar q}\mc C^*)),F)$. 
\end{lemma}
\begin{proof}
This follows directly from the definitions of the maps involved. Note that the shifted $\Lambda[-n]$ acts on elements just as $\Lambda$ does (though with degrees shifted). 
\end{proof} 
 
 Since $\Lambda$ is a sheaf quasi-isomorphism, the induced hypercohomology map is an isomorphism, and we already know $e$ is an isomorphism, hence $e'$ is also an isomorphism. Alternatively, $e'$ is precisely the universal coefficient map and so an isomorphism by the Universal Coefficient Theorem.

\paragraph{V.}  Recall that $\Gamma(X;\ms L^i\mc I_{\bar q}\mc C^*[-n])=\Hom^{i-n}(\mc I_{\bar q}\mc C^{*}(X),F)=\Hom(\mc I_{\bar q}\mc C^{j}(X),F)$. We have used here that $X$ is compact, that $i+j=n$, and that $F$ is treated as a complex in degree $0$. We define the map $h:\H^i(X;\ms L\mc I_{\bar q}\mc C^*[-n])=\H^{-j}(X;\ms L\mc I_{\bar q}\mc C^*)\to H^{-j}(\Hom(I_{\bar q}C^*(X; F),F))$ to be induced by
the $\Hom(\cdot,F)$ dual of the surjection $I_{\bar q}S^*(X; F)\to \mc I_{\bar q}\mc C^*(X)$, which exists by \cite[Theorem I.6.2]{BR} because the presheaf $U\to I_{\bar q}S^*(U; F)$ is conjunctive by Proposition \ref{P: conj}. Furthermore, this surjection is a quasi-isomorphism since the complex of intersection cochains with $0$-support is quasi-isomorphic to $0$, as also observed in Section \ref{S: sheaf}. It follows as in \cite[Theorem 45.6]{MK} that $h$ is an isomorphism since the dual of a quasi-isomorphism of free complexes is a quasi-isomorphism. Furthermore, the commutativity of square V and the fact that $e''$ is also an isomorphism is simply an application of the  universal coefficient theorem and its naturality; notice that all complexes are free as vector spaces, so these universal coefficient theorems require no special finiteness hypotheses.

\paragraph{VIII.}
 
For polygons VII and VIII, we need to define 
 the sheaf $\ms P\mc I_{\bar q}\mc C^*$. We let it be the sheafification of the presheaf 
$P I_{\bar q} C^*$ defined by $U\to  \Hom(I_{\bar q}S^*(X, X-\bar U; F),F)$. Notice that if $V\subset U$, then there is an injection $I_{\bar q}S^*(X, X-\bar V; F)\into I_{\bar q}S^*(X, X-\bar U; F)$, so the restriction map of the presheaf is just the $\Hom(\cdot,F)$ dual of this injection. $\ms P\mc I_{\bar q}\mc C^*$ plays the role of the (shifted) double dual of $\mc I^{\bar q}\mc S^*$.

\begin{lemma}
The sheaf complex $\ms P\mc I_{\bar q}\mc C^*$ is homotopically fine.
\end{lemma}
\begin{proof}
We must show that if $\mc{U}=\{U_k\}$ is a locally-finite cover of $X$ then there exist endomorphisms $1_k$ and $\ms D$ of $\ms P\mc I_{\bar q}\mc C^*$ such that $|1_k|\subset \bar U_k$ and $\sum 1_k=\text{id} -d \ms D -\ms Dd$, where $d$ is the coboundary map of $\ms P\mc I_{\bar q}\mc C^*$, i.e. $\sum 1_k$ is chain homotopic to the identity. The $1_k$ need not be chain maps; see \cite[Section 6]{SW}. 

In \cite[Proposition 3.5]{GBF10}, it is shown that the sheaf complex $\mc I^{\bar q}\mc S^*$ of intersection chains is homotopically fine. As part of the proof, it is show that it is possible to construct maps $g_k: I^{\bar q}S_*(X; F)\to I^{\bar q}S_*(X; F)$ so that the support of $g_k$ is contained in $U_k$ and $\sum g_k$ is chain homotopic to the identity by a chain homotopy $D$. The maps $g_k$ and $D$ induce maps on the quotients  $I^{\bar q}S_*(X, X-\bar U; F)$. We consider the double duals $g^{**}$ and $D^{**}$ on $\Hom(\Hom(I^{\bar q}S_*(X, X-\bar U; F),F),F)=P I_{\bar q} C^*(U)$. Since $g_k$ and $D$ restrict in the appropriate way for $V\subset U$, so do their double duals, and they induces maps of presheaves and hence maps of sheaves. Let $1_k$ and $\ms D$ be the induces sheaf maps. We claim they satisfy the necessary properties to make $\ms P\mc I_{\bar q}\mc C^*$  homotopically fine.

First we show the support of $1_k$ is in $\bar U_k$. Let $x\in X-\bar U_k$, and let $V$ be a neighborhood of $x$ such that $\bar V\cap \bar U_k=\emptyset$; such a $V$ can be found by taking a distinguished neighborhood of $x$ in $X-\bar U_k$ and then letting $V$ be an appropriate smaller distinguished neighborhood. Let $s\in P I_{\bar q} C^*(V)$, and let $\alpha\in  \Hom(I^{\bar q}S_*(X, X-\bar V; F),F)$. Consider now $(g_k^{**}s)(\alpha)=s(g_k^*(\alpha))$. The cochain $g_k^*(\alpha)$ acts on chains $\xi\in I^{\bar q}S_*(X, X-\bar V; F)$ by $g_k^*(\alpha)(\xi)=\alpha(g_k(\xi))$. But the support of $g_k(\xi)$ lies in $\bar U_k$, while $\alpha$ kills any chain with support outside of $\bar V$. Thus $g_k^*(\alpha)=0$, so $(g_k^{**}s)$ must always be $0$ for $s\in P I_{\bar q} C^*(V)$. Therefore the support of $g_k^{**}$ lies in $\bar U_k$. 

Next, consider the equation $\sum g_k=\text{id} -\bd  D - D\bd$ on $I^{\bar q}S_*(X; F)$ that shows that $\sum g_k$ is chain homotopic to the identity. By dualizing twice, $\text{id}^{**}$ is chain homotopic to $(\sum g_k)^{**}$ by \cite[Lemma A.2.2]{GBF35}. But $\text{id}^{**}$ is the identity on the double duals, and for any sufficiently small $U$ the sum $(\sum g_k)^{**}$ is equal to $\sum g_k^{**}$ since the covering $\{U_k\}$ is locally finite and so only a finite number of the $g_k$ are non-zero when restricted to such a $U$. Next we take limits and  sheafify, which provides the desired chain homotopy at the sheaf level. \qedhere

\end{proof}

Each $H^i(\ms P\mc I_{\bar q}\mc C^*_x)$ vanishes for $|i|$ sufficiently large by the Universal Coefficient Theorem and the standard local intersection cohomology computations. It follows from  Lemma \ref{L: C-ready} that $\ms P\mc I_{\bar q}\mc C^*$ is $C$-ready, and so we have the following:

\begin{corollary}
$\H^*(X;\ms P\mc I_{\bar q}\mc C^*)\cong H^*(\Gamma(X;\ms P\mc I_{\bar q}\mc C^*))$.
\end{corollary}

Now, recall that the map $r$ of diagram \eqref{D: MAIN} is defined so that for $x\in I^{\bar q}H_j(X;F)$ and $\alpha\in I_{\bar q}H^j(X;F)$, $r(x)$ acts on $\alpha$ by $r(x)(\alpha)=(-1)^{j}\alpha(x)$. To define a sheaf map $\rho: \mc I^{\bar q}\mc S^*\to \ms P\mc I_{\bar q}\mc C^*[-n]$, we  need a sheafified version of $r$. The main issue is that we need to be careful about the various shifts of indexing that are involved. 

In the notation of \cite{GBF10}, let $K^{\bar q}S^*$ be the presheaf $U\to I^{\bar q}S_{n-*}(X,X-\bar U; F)$. As noted in Section \ref{S: SIC}, since we are technically shifting by $-n$ (in cohomological indexing), the boundary maps of this complex are given a sign $(-1)^n$ compared with the usual boundary maps for $I^{\bar q}S_{*}(X,X-\bar U; F)$.
Then $K^{\bar q}S^*$ also sheafifies to $\mc I^{\bar q}\mc S^*$ by \cite[Lemma 3.1]{GBF10}. We define a degree $0$ map of presheaves $\bar r:K^{\bar q}S^*\to PI_{\bar q}C^*[-n]$ as follows: Suppose $x\in K^{\bar q}S^i(U)= I^{\bar q}S_{n-i}(X,X-\bar U; F)$. Then $\bar r(x)$ must be an element of  $(PI_{\bar q}C^*[-n])^i(U)=PI_{\bar q}C^{i-n}(U)=\Hom(I_{\bar q}S^{n-i}(X,X-\bar U;F),F)$. If $\alpha\in I_{\bar q}S^{n-i}(X,X-\bar U;F)$, we let\footnotemark $\bar r(x)(\alpha)=(-1)^{n-i}\alpha(x)$. Thus for fixed $U$, $\bar r$ is simply the double dual map called $f$ in the Appendix. 

\footnotetext{To further justify this sign,  we could instead define a presheaf complex $J^{\bar q}S^*$ by $U\to I^{\bar q}S_{-*}(X,X-\bar U; F)$. Then we would let $\mf r: J^{\bar q}S^*\to PI_{\bar q}C^*$ be defined so that for
$x\in J^{\bar q}S^i(U)= I^{\bar q}S_{-i}(X,X-\bar U; F)$ we have $\mf r(x)\in PI_{\bar q}C^{i}(U)=\Hom(I_{\bar q}S^{-i}(X,X-\bar U;F),F)$ acting on $\alpha\in I_{\bar q}S^{-i}(X,X-\bar U;F)$ by 
 $\mf r(x)(\alpha)=(-1)^{i}\alpha(x)$. In this case the sign agrees with the natural degrees of the chains and cochains, as we would expect. The map $\bar r$ is then the $[-n]$ shift of $\mf r$, which is still the same map degree-wise and so doesn't change the sign in the behavior of the cochain $\mf r(x)$.}

It is not hard to check that $\bar r$ is a map of presheaves as, for open $V\subset U$, the  commutativity of the following diagram commutes:

\begin{diagram}
I^{\bar q}S_{j}(X,X-\bar U; F)&\rTo^{\bar r} &\Hom(I_{\bar q}S^{j}(X, X-\bar U; F),F)\\
\dTo&&\dTo\\
I^{\bar q}S_{j}(X,X-\bar V; F)&\rTo^{\bar r} &\Hom(I_{\bar q}S^{j}(X, X-\bar V; F),F).
\end{diagram} 

Now let us verify that $\bar r$ is a degree $0$ chain map. To simplify the notation, let $A=K^{\bar q}S^*$ and $B=PI_{\bar q}C^*$, and let $d_A$ and $d_B$ be the corresponding coboundary maps. Recall from Section \ref{S: SIC} that $d_A=(-1)^n\bd$.  
 Let $x\in A^i(U)=I^{\bar q}S_{n-i}(X,X-\bar U; F)$. Then $d_{B[-n]}\bar r(x)$ and $\bar r(d_Ax)$ both live in 
$(B[-n])^{i+1}=B^{i+1-n}=\Hom(I_{\bar q}S^{n-i-1}(X,X-\bar U;F),F)$. Let $\alpha\in I_{\bar q}S^{n-i-1}(X,X-\bar U;F)$.
 Then $\bar r(d_Ax)(\alpha)=(-1)^{n-i-1}\alpha(d_Ax)=(-1)^{n-i-1}\alpha((-1)^n\bd x)=(-1)^{i+1}\alpha(\bd x)$. On the other hand, 
\begin{align*}
(d_{B[-n]}(\bar r(x)))(\alpha)&=(-1)^n (d_{B}(\bar r(x)))(\alpha)\\
&=(-1)^n(-1)^{i-n+1}\bar r(x)(d\alpha)\\
&=(-1)^n(-1)^{i-n+1}(-1)^{n-i}(d\alpha)(x)\\
&=(-1)^n(-1)^{i-n+1}(-1)^{n-i}(-1)^{n-i-1+1}\alpha(\bd x)\\
&=(-1)^{i+1}\alpha(\bd x).
\end{align*}
The first equality just uses the definition of the coboundary of a shifted complexes: $d_{B[-n]}=(-1)^nd_B$. The second equality uses the definition of the coboundary for an element of $\Hom^{i-n}(I_{\bar q}S^*(X,X-\bar U;F),F)$. The third equality is the definition of $\bar r$. The fourth equality is the definition of the coboundary on $I_{\bar q}S^*$, and the last equality simplifies the signs. Comparing with $\bar r(d_Ax)(\alpha)$, we see that $\bar r$ is a degree $0$ chain map of presheaves.

Sheafifying $\bar r$ yields the sheaf map $\rho: \mc I^{\bar q}\mc S^*\to \ms P\mc I_{\bar q}\mc C^*[-n]$.
Furthermore, $\rho$ is  a quasi-isomorphism of sheaves: For any (sufficiently small) distinguished neighborhood $U$ of $x$, the map $\bar r$ is guaranteed to be an isomorphism as in Lemma \ref{L: coco UCT}. But the distinguished neighborhoods of $X$ form a 
 cofinal system.  This ensures that we obtain a quasi-isomorphism at each stalk of the induced sheaves.

Now we turn to polygon VIII of Diagram \eqref{D: MAIN}. We define $\tau$ to be induced by the sheafification of global section of $P I_{\bar q} C^*$ to global sections of $\ms P\mc I_{\bar q}\mc C^*$; note that $\H^i(X;\ms P\mc I_{\bar q}\mc C^*[-n])=\H^{-j}(X;\ms P\mc I_{\bar q}\mc C^*)$. 
We also observe that the map $r$ in Diagram \eqref{D: MAIN} corresponds to the cohomology map induced by $\bar r$ on global sections. 
The commutativity of VIII then follows from the naturality of sheafification of global sections and of cohomology. Note that the degrees of the various maps don't come into play in checking this commutativity.

Since $\rho$ is a quasi-isomorphism of sheaves,  it must induce an isomorphism on hypercohomology. 
The sheafification induced map $\sigma'$ is an isomorphism as the presheaf $I^{\bar q}S^*$ is a monopresheaf and conjunctive for coverings and $\mc I^{\bar q}\mc S^*$ is homotopically fine. We have previously observed that $r$ is an isomorphism. It follows that $\tau$ is an isomorphism.

  \paragraph{VII.} 
  
  We begin by constructing $\eta:\ms L\mc I_{\bar q}\mc C^*\to \ms P\mc I_{\bar q}\mc C^*$, which requires a construction at the presheaf level. This map will essentially be the dual of an inclusion, but this requires a bit of work as there are not always appropriate maps $LI_{\bar q}C^*(U)=\Hom(\Gamma_c(U;\mc I_{\bar q}\mc C^*),F)  \to P I_{\bar q}C^*(U)=\Hom(I_{\bar q}S^*(X, X-\bar U; F),F)$ for every $U$.

If we have two presheaves $S^*$ and $T^*$ on a space $X$, then a map of presheaves $S^*\to T^*$ induces a map of the induced sheaves $\mc S^*\to \mc T^*$, but notice that it is possible to get a map of sheaves with less starting data. In particular, suppose ever 
point $x\in X$ has a cofinal system of neighborhoods $\{U_i\}$ with $U_{i+1}\subset U_i$ for all $i$. Then to define a map of sheaves, it is sufficient to have for each $i$ a commutative diagram
\begin{diagram}
S^*(U_i)&\rTo &T^*(U_{i+1})\\
\dTo &&\dTo \\
S^*(U_{i+1}) &\rTo& T^*(U_{i+2})
\end{diagram}
in which the vertical maps are the presheaf restriction maps. To see this, note that this information is sufficient to define a compatible sequence of maps $S^*(U_i)\to \dlim_j T^*(U_j)=\mc T^*_x$ by taking the direct limit over $T$, and that this is enough then to get a map $\dlim_i S(U_i)=\mc S^*_x\to \mc T^*_x$. Altogether, this provides a sheaf map. If the maps in the diagram are all quasi-isomorphisms, then the induced sheaf map is a quasi-isomorphism.

Now each point $x\in X$ has a cofinal system of distinguished neighborhoods $U\cong \R^{n-k}\times cL$. For each $x$, we can fix one such $U$. Then let $D^{n-k}_r$ denote the open $n-k$ ball of radius $r$. Let $c_rL$ denote the open cone on $L$ of radius $r$, i.e. $c_rL=\frac{[0,r)\times L}{0\times L}$. We can assume $U\cong D^{n-k}_1\times c_1L$. For $0<r<1$, let $U_{r}$ be the homeomorphic image in $U$ of $D^{n-k}_r\times c_rL$. Then we claim that for $0<t<s<r<1$ there is a commutative diagram of quasi-isomorphisms

\begin{diagram}
\Gamma_c(U_r;\mc I_{\bar q}\mc C^*)&\lTo &I_{\bar q}S^*(X, X-\bar U_s; F)\\
\uTo &&\uTo \\
\Gamma_c(U_s;\mc I_{\bar q}\mc C^*) &\lTo& I_{\bar q}S^*(X, X-\bar U_t; F).
\end{diagram}
Dualizing this diagram by $\Hom(\cdot,F)$ and applying the universal coefficient theorem will yield our desired quasi-isomorphism of sheaves $\eta:\ms L\mc I_{\bar q}\mc C^*\to \ms P\mc I_{\bar q}\mc C^*$.

To verify the claim, we first note that the vertical maps of the diagram are both inclusions, while the horizontal maps are induced by sheafification of cochains. If $\alpha\in I_{\bar q}S^*(X, X-\bar U_s; F)$, then the support of $\alpha$ under sheafification must be contained within the compact set $\bar U_s\subset U_r$. Hence the diagram is well-defined, and commutativity is evident from the naturality of sheafification (which follows from the functoriality of direct limits). So we need only verify that the maps are quasi-isomorphisms. For the righthand vertical map, this is a consequence of stratum-preserving homotopy invariance of intersection homology. So it will suffice to show that the horizontal maps are quasi-isomorphisms, for which we can use the same argument for each; it will follow from the commutativity that the lefthand vertical map is also a quasi-isomorphism. 

Let us factor the morphism on cohomology induced by the horizontal maps as

\begin{align*}
I_{\bar q}H^*(X, X-\bar U_s; F)&\to I_{\bar q}H^*(U_r, U_r-\bar U_s; F)\\
&\to  I_{\bar q}H_c^*(U_r; F)\\
&\to H^*(\Gamma_c(U_r;\mc I_{\bar q}\mc C^*)).
\end{align*}
The first map is the excision isomorphism. The second map is the isomorphism induced by the natural map $I_{\bar q}H^*(U_r, U_r-\bar U_s; F)\to \displaystyle\dlim_{u\to r^-}I_{\bar q}H^*(U_r, U_r-\bar U_u; F)\cong I_{\bar q}H_c^*(U_r; F)$; see \cite[Section 7.3]{GBF35}. See \cite[Section 3.3]{Ha} for a discussion of this approach to cohomology with compact supports. The direct system is a system of isomorphisms due to stratum-preserving homotopy invariance of intersection homology. Finally, the last isomorphism is by Corollary \ref{C: cohomology}.

This establishes the claim and hence the quasi-isomorphism $\eta$.

\begin{remark}\label{R: Verdier is singular}
Together, $\eta[-n]$ and the quasi-isomorphism $\rho$ we constructed in discussing square VIII determine a quasi-isomorphism between $\ms L\mc I_{\bar q}\mc C^*[-n]$ and $\mc I^{\bar q}\mc S^*$. However, if we replace  $\mc I^{\bar p}\mc C^*$ and $\mc I_{\bar q}\mc S^*$, respectively, with $\mc C^*$ and $\mc S^*$, the sheaves of ordinary singular cochains and chains, then the same arguments give a quasi-isomorphism between $\ms L\mc C^*[-n]$ and $\mc S^*$. But by the same arguments as Proposition \ref{P: flabby}, $\mc C^*$ is a soft flat resolution of the constant sheaf $\mc F$ with stalk $F$. Thus $\ms L\mc C^*$, whose sections are $\ms L\mc C^*(U)= \Hom(\Gamma_c(U;\mc C^*),F)$, represents the Verdier dualizing sheaf $\D^*$ as defined in \cite[Section V.7.1]{Bo}. For pseudomanifolds and with field coefficients this provides another proof 
 of the well-known fact that $\H^*(X;\D^*[-n])\cong \H^{*}(X;\mc S^*)\cong H_{n-*}^{\infty}(X;F)$. \hfill\qedsymbol
\end{remark}

Now we will show that the diagram

\begin{diagram}
&&H^{-j}(\Hom(I_{\bar q}S^*(X; F),F))\\
&\ruTo^h&\dTo^\tau\\
\H^{-j}(X;\ms L\mc I_{\bar q}\mc C^*)&\rTo^{ \eta}&\H^{-j}(X;\ms P\mc I_{\bar q}\mc C^*)
\end{diagram}
commutes. This is simply  triangle VII of diagram \eqref{D: MAIN}  rewriting $\H^i(X;\mc A^*[-n])$ as $\H^{i-n}(X;\mc A^*)=\H^{-j}(X;\mc A^*)$ for each of the bottom terms, with $\mc A^*$ as appropriate.
Recall that $X$ is compact so $\Gamma_c(X;\mc I_{\bar q}\mc C^*)=\Gamma(X;\mc I_{\bar q}\mc C^*)=\mc I_{\bar q}\mc C^*(X)$ and that   $h$ is the $\Hom$ dual of the sheafification map $I_{\bar q}S^*(X; F)\to \mc I_{\bar q}\mc C^*(X)$. The map $\tau$ is the sheafification $\Hom(I_{\bar q}S^*(X; F),F)\to \Gamma(X;\ms P\mc I_{\bar q}\mc C^*)$. 
We  check commutativity. 

Given an $\alpha\in \Hom(\Gamma_c(X;\mc I_{\bar q}\mc C^*),F)$, then $h(\alpha)$ acts on $\beta\in I_{\bar q}S^*(X; F)$ by taking $\beta$ to its sheaf section and then applying $\alpha$. The image of $h(\alpha)$ under $\tau$ is a section of $\ms P\mc I_{\bar q}\mc C^*$. At the point $x$, the germ $\tau h(\alpha)_x$ is represented by an element of $\Hom(I_{\bar q}S^*(X,X-\bar U; F),F)$, for some neighborhood $U$ of $x$, obtained by restricting $h(\alpha)$ to act on cochains that vanish outside of $\bar U$. But on such cochains, $h(\alpha)$ still acts by sheafifying the cochain to a section in $\Gamma_c(X;\mc I_{\bar q}\mc C^*)$ and then applying $\alpha$.  
Now what is $\eta(\alpha)$? The map $\eta$ is defined locally by dualizing sheafification maps $I_{\bar q}S^*(X, X-\bar U_s; F) \to \Gamma_c(U_r;\mc I_{\bar q}\mc C^*)$ for $r>s$. 
Thus a germ at $x$ of the global section $\eta(\alpha)$  is represented by an element of $\Hom(I_{\bar q}S^*(X, X-\bar U_s; F),F)$ for some $s$, and  this acts on cochains that vanish in $X-\bar U_s$ by sheafifying them to elements of $\Gamma_c(U_r;\mc I_{\bar q}\mc C^*)\subset \Gamma_c(X;\mc I_{\bar q}\mc C^*)$, for sufficient small $r$, and then 
 applying the restriction of  $\alpha$. But this is exactly what $\tau h(\alpha)_x$ does, noting that we can represent a germ of $\ms P\mc I_{\bar q}\mc C^*_x$ in $\Hom(I_{\bar q}S^*(X,X-\bar U_s; F),F)$ for some $U_s$. 
Thus $\tau h=\eta$ and VII commutes. 

Finally, we have already shown that all maps of VII are isomorphisms.

\paragraph{Constants.}\label{P: constants}

So far we have succeeded in showing that diagram \eqref{D: MAIN} commutes up to a constant. However, we would like to be even more specific. For one thing, we have not determined the constant represented by the composition
\begin{equation}\label{E: II}
F\xrightarrow{\aug^{-1}} I^{\bar t}H_0(X; F)\xrightarrow{(-1)^n(\cdot\cap\Gamma)^{-1}}I_{\bar 0}H^n(X; F)\xrightarrow{\sigma} \H^n(X;\mc I_{\bar 0}\mc C^*)\xrightarrow{\mathbb{L}}\H^n(X;\D^*[-n])\xr{\ell} F
\end{equation}
appearing in the commutativity discussion of square II of diagram \eqref{D: MAIN}. For another, we have not yet determined that the composition morphism  $\H^i(X;\mc I_{\bar p}\mc C^*)\to \H^i(X;\mc I^{\bar q}\mc S^*)$ along the bottom of diagram \eqref{D: MAIN} is induced by the quasi-isomorphism $\mathbb{O}$ consistent with the orientation. It turns out that if we  choose $\ell$ such that the  composition \eqref{E: II} is multiplication by $1$, i.e.\ it is the identity, then 
the morphism $\H^i(X;\mc I_{\bar p}\mc C^*)\to \H^i(X;\mc I^{\bar q}\mc S^*)$ becomes consistent with the orientation\footnote{It makes sense that the composition \eqref{E: II} should be ``off'' by a sign $(-1)^n$ as the correct Poincar\'e duality map $I_{\bar 0}H^n(X; F)\to I^{\bar t}H_0(X; F)$ is $(-1)^n\cdot \cap \Gamma$.}.

To show this,  let $U$ be a Euclidean neighborhood of $x$ in $U_1=X-X^{n-1}$. Note that when we restrict to such a neighborhood, $I_{\bar p}S^*(U;F)=S^*(U;F)$ for any $\bar p$ so that we may use intersection cochains and ordinary singular chains interchangeably in what follows whenever we restrict to $U$ or to any other subset of $U_1$. Let $1$ be the  cocycle in $S^0(U;F)$ that evaluates to $1\in F$ on every singular $0$-simplex. Then the germ of $1$ at $x$ represents $1_x\in \mc I_{\bar p}\mc C^0_x$.
Now consider $\eta[-n]\circ \Lambda[-n]\circ\hat \Phi(1_x)\in \ms P^0\mc I_{\bar q}\mc C^*_x[-n]$, and let us call this $E(1_x)$.  On the other hand, for $V$ a neighborhood of $x$ such that $\bar V\subset U$, let $\gamma_x\in I^{\bar q}S_n(X,X-\bar V; F)$ be a cycle representing the orientation class.  The cycle  $\gamma_x$ determines an element that we also label $\gamma_x$ in $\mc I^{\bar q}\mc S^0_x$. Then $\rho(\gamma_x)$ is also in $\ms P^0\mc I_{\bar q}\mc C^*_x[-n]$, and we must compare $E(1_x)$ and  $\rho(\gamma_x)$ as elements of $H^0( \ms P^0\mc I_{\bar q}\mc C^*_x[-n])$.

Now $\ms P^0\mc I_{\bar q}\mc C^*_x[-n]= \ms P^{-n}\mc I_{\bar q}\mc C^*_x$, whose germs are represented by homomorphisms of degree $-n$ from 
$I_{\bar q}S^*(X,X-\bar V; F)$ to $F$ for sufficiently small open $V$. In particular, these maps act non-trivially only on intersection cochains of degree $n$.
Since it suffices to consider the cohomology classes $H^0$ at $x$, notice that $H^*(\Hom(I_{\bar q}S^*(X,X-\bar V; F),F))\cong \Hom(I_{\bar q}H^*(X,X-\bar V; F),F)$ by the Universal Coefficient Theorem, so to identify elements of $H^0( \ms P^0\mc I_{\bar q}\mc C^*_x[-n])$, it is enough to look at how representatives of germs act on cocycles. In particular, we assume from here on that $\beta$ is a cocycle representing an element of $I_{\bar q}H^n(X,X-\bar V; F)$. 

Tracing through the definitions of $\eta$, $\Lambda$, and $\hat \Phi$, let us see how $E(1_x)$ acts on a cocycle  $\beta\in I_{\bar q}S^n(X,X-\bar V; F)$. First we form $\Phi(\td 1\td \otimes \td \beta)\in  \Gamma_c(U;\D^0)$, where $\td \beta$ and $\td 1$ are the sheafifications of $\beta$ and $1$ in  $\Gamma(U;\mc I^{\bar q}\mc C^n)$ and $\Gamma_c(U;\mc I^{\bar p}\mc C^0)$, respectively, and $\td \otimes$ represents the sheaf tensor product.
The cocycle section $\Phi(\td 1\td \otimes \td \beta)=\mathbb{L}(\td 1\td \cup\td \beta)\in  \Gamma_c(U;\D^0)$ is then taken to the class it represents under the composition $\H_c^0(U;\D^*)\to \H_c^0(X;\D^*) \xr{\ell} F$.

By contrast,  $\rho(\gamma_x)$ acts on the same $\beta$  by $\rho(\gamma_x)(\beta)=(-1)^n\beta(\gamma_x)$, which  is equal to  $(-1)^n\aug(\beta\cap \gamma_x)$ by \cite[Proposition 7.3.25]{GBF35}.

Now, the comparison between $E(1_x)$ and $\rho(\gamma_x)$ is very close to being a localized version of the composition \eqref{E: II}.
To make this precise, we consider the following diagram (coefficients tacit):

\begin{equation*}
\resizebox{.9\hsize}{!}{\begin{diagram}
F\lTo^{\aug} &  &H_0(U)&\lTo^{(-1)^n\cdot\cap \gamma_x}&H^n(U,U-\bar V)&\rTo^{\sigma}& \H^n_c(U;\mc I_{\bar 0}\mc C^*)&\rTo^{\mathbb{L}} &\H^n_c(U;\D^*[-n])\\
&\luTo^{\aug}&\dTo&&\uTo^\cong\\
&& I^{\bar t}H_0(X) &\lTo^{(-1)^n\cdot\cap\Gamma}&I_{\bar 0}H^n(X,X-\bar V)&&\dTo&&\dTo\\
&&&\luTo^{(-1)^n\cdot\cap\Gamma}&\dTo\\
 &&&&I_{\bar 0}H^n(X)&\rTo^{\sigma}& \H^n(X;\mc I_{\bar 0}\mc C^*)&\rTo^{\mathbb{L}} &\H^n(X;\D^*[-n])&\rTo^\ell& F\\
\end{diagram}}
\end{equation*}

This diagram commutes: The left triangle commutes by naturality of augmentation. The small square and other triangle commute by naturality of the cap product \cite[Proposition 7.3.6]{GBF35} (identifying ordinary and intersection (co)homology on the manifold $U$ as well as letting $\gamma_x$ and $\Gamma$ stand also for the elements of $H_n(U,U-\bar V;F)$ and $I^{\bar 0}H_n(X,X-\bar V;F)$ they represent).
The maps counterclockwise around the pentagon come by sheafifying a cocycle of $I_{\bar 0}S^n(X,X-\bar V;F)$. Since the support of such a cocycle must be compact in $U$, we can equivalently restrict first to $U$, sheafify, and then include back into $X$, which gives us the clockwise procedure. Thus the pentagon commutes. Commutativity of the right rectangle is clear. 

The composition from $F$ to $F$ along the bottom is just the composition \eqref{E: II}. As we've seen, if we start with $\beta$ representing an element of $I_{\bar q}H^j(X,X-\bar V;F)$, then $E(1_x)$ acts on it by first taking it to $\td 1 \td\cup \td \beta\in \Gamma_c(U;\mc I_{\bar 0}\mc C^*)$ and then proceeding around the diagram to the right. Note that $\td 1 \td\cup \td \beta$ is the image of $1\cup \beta|_U$ under the sheafification. Although we start with $\beta$ as a perversity-$\bar q$ cochain, once restricted to $U$ it is an ordinary cochain. So in particular $1\cup \beta|_U=\beta|_U\in H^n(U,U-\bar V;F)$ so we can start here in the diagram instead and go right to compute $E(1_x)(\beta)$. We have also seen that $\rho(\gamma_x)$ acting on $\beta$ is $(-1)^n\beta(\gamma_x)$. As we can choose our representative for $\gamma_x$ to be supported in $U$, this is equal to the image of $\beta$ going left from $H^n(U,U-\bar V;F)$ to $F$. Since the diagram commutes, for $E(1_x)$ and $\rho(\gamma_x)$ to act equally on $\beta$, we need the the composition from $F$ to $F$ along the bottom to be the identity, which can be achieved choosing $\ell$ appropriately. 

With this choice of $\ell$, the composition along the bottom of Diagram \eqref{D: MAIN} is induced by $\mathbb O$ and square II commutes.

\bigskip

Now putting together our previous computations, all parts of diagram \eqref{D: MAIN} commute. Thus the diagram of Theorem \ref{T: MAIN} commutes.

This completes the proof for $X$ connected and normal. The proof clearly extends to $X$ normal and consisting of a finite number of connected components.

\paragraph{If $\mathbf{X}$ is not normal.}\label{P: not normal}
Suppose now $X$ is compact and oriented, but not necessarily normal. Let $p:\hat X\to X$ be the normalization; see \cite{Pa03}. Recall that $p$ restricts to a homeomorphism on $X-X^{n-1}$, so in particular we can choose an orientation for $\hat X$ such that $p$ is orientation preserving. Consider the diagram

\begin{diagram}[notextflow]
   I_{\bar p}H^i (\hat X; F)  &    &\rTo^{(-1)^{in}\cdot\cap \Gamma_{\hat X}} &      &   I^{\bar q}H_{n-i} (\hat X; F)   \\
      & \rdTo_{\sigma} &      &      & \vLine& \rdTo^{\sigma'}  \\
\uTo^{p^*} &    &   \H^i (\hat X;\mc I_{\bar p}\mc C_{\hat X}^*)   & \rTo^{\mathbb{O_{\hat X}}}  & \HonV   &    &  \H^i(\hat X;\mc I^{\bar q}\mc S_{\hat X}^*) \\
      &    & \uTo^{p^*}  &      & \dTo_{p}   \\
   I_{\bar p}H^i ( X; F) & \hLine & \VonH   & \rTo^{(-1)^{in}\cdot\cap \Gamma_X} &  I^{\bar q}H_{n-i} ( X; F)  &    & \dTo_{p} \\
      & \rdTo_{\sigma} &      &      &      & \rdTo^{\sigma'}  \\
      &    &   \H^i ( X;\mc I_{\bar p}\mc C_{X}^*)   &      & \rTo^{\mathbb{O_X}} &    &  \H^i(X;\mc I^{\bar q}\mc S_{ X}^*) .\\
\end{diagram}

The top square commutes by our proof for compact normal pseudomanifolds. It is the bottom square we wish to prove commutes, which we will do by explaining the sides of the cube and seeing that they commute. 

The vertical maps on the back face are induced by the normalization $p:\hat X\to X$. To see that the back face commutes, we first observe that if $\Gamma_{\hat X}$ is the fundamental class for $\hat X$ consistent with the orientation, then $p\Gamma_{\hat X}=\Gamma_X$, which follows from \cite[Theorem 8.1.18]{GBF35} and our assumptions about $p$ and $\hat X$. 
The back face then commutes by the naturality of the cap product \cite[Proposition 7.3.6]{GBF35}.

For the righthand side, notice that we have a sheaf map $p_*\mc I^{\bar q}\mc S^*_{\hat X}\to \mc I^{\bar q}\mc S^*_{ X}$ because the germs of the former sheaf at $x\in X$ are represented in neighborhoods $U$ of $x$ by elements of $I^{\bar q}S_{n-*}(\hat X,\hat X-\overline{p^{-1}(U)}; F)$, and these map under $p$ to elements of $I^{\bar q}S_{n-*}(X, X-\bar U; F)$ representing germs of $\mc I^{\bar q}\mc S^*_{ X}$ at $x$. This induces a map of global sections $H^*(\mc I^{\bar q}\mc S^*_{\hat X}(\hat X))=H^*(p_*\mc I^{\bar q}\mc S^*_{\hat X}(X))\to H^*(\mc I^{\bar q}\mc S^*_{ X}(X))$. Since 
$\mc I^{\bar q}\mc S^*_{\hat X}$ and $\mc I^{\bar q}\mc S^*_{ X}$ are $C$-ready, this represents the desired map on hypercohomology. Furthermore, by \cite[Section 3]{GBF10}, $I^{\bar q}S_{n-*}(X; F)\cong \Gamma(X;\mc I^{\bar q}\mc S_{X}^*)$ (and similarly for $\hat X$), and it is easy to check at the stalk level  that the homomorphism $p:I^{\bar q}S_{n-*}(\hat X; F)\to I^{\bar q}S_{n-*}(X; F)$ is compatible with the sheaf morphism so that the right side of the cube commutes

The left side is similar, though we instead use a map $\mc I_{\bar p}\mc C_{X}^*\to p_*\mc I_{\bar p}\mc C_{\hat X}^*$, corresponding to the fact that cochains pull back over maps. 
The sheafification map $I_{\bar p}S^{*}(X; F)\to \Gamma(X;\mc I_{\bar p}\mc C_{X}^*)$ is only surjective here, but this does not disturb the commutativity argument. 

Turning to the front of the cube, we have already seen that the left and right vertical maps can be interpreted as the cohomology maps obtained from the maps on global sections induced by the sheaf maps $\mc I_{\bar p}\mc C_{X}^*\to p_*\mc I_{\bar p}\mc C_{\hat X}^*$ and $p_*\mc I^{\bar q}\mc S^*_{\hat X}\to \mc I^{\bar q}\mc S^*_{ X}$  over $X$. Similarly, to keep all the sheaves over $X$, we see that the map induced by $\O_{\hat X}$ can be interpreted over $X$ as the hypercohomology map induced by  $p_*\O_{\hat X}:p_*\mc I_{\bar p}\mc C_{\hat X}^*\to p_*\mc I^{\bar q}\mc S_{\hat X}^*$. So to show that the front square commutes it suffices to show that the composition

$$\mc I_{\bar p}\mc C_{X}^*\to p_*\mc I_{\bar p}\mc C_{\hat X}^*\xr{p_*\O^{\hat X}} p_*\mc I^{\bar q}\mc S_{\hat X}^*\to \mc I^{\bar q}\mc S_{X}^*$$ agrees with $\mc O_{X}$. 

For this, we know that it is sufficient to consider what the composition does over points in the regular strata. But since $p$ is an orientation preserving homeomorphism over $X-X^{n-1}$,  if we restrict all sheaves and maps to  $X-X^{n-1}$ the first and third maps become identity maps and $p_*\O_{\hat X}$ becomes exactly $\O|_{X-X^{n-1}}$. So this face commutes.

It now follows from a diagram chase that the bottom of the cube commutes, which proves Theorem \ref{T: MAIN} for not-necessarily normal compact  $F$-oriented pseudomanifolds.\hfill\qedsymbol

\section{Compatibility of cup, intersection, and sheaf products}\label{S: cap/cup}

We now turn to the proof of Theorem \ref{T: cubes}, stated in the Introduction, which relates the cup product, the intersection product, and the sheaf products. 
We first discuss the commutativity of the top cube, for which $X$ can be a topological stratified pseudomanifold. We then turn to the bottom cube with the further assumption that $X$ be a PL stratified pseudomanifold.

\paragraph{Top cube. }
The top of the cube commutes by Theorem \ref{T: cup product}. The right side commutes by Theorem \ref{T: MAIN}. For the left side, we note that $\sigma$ and $\mathbb O$ are degree $0$ maps while the signed cap product and $\sigma'$ have (cohomological) degrees $-n$ and $n$, respectively. Therefore, using Theorem \ref{T: MAIN} again,  the lefthand square commutes up to $(-1)^{jn+(n-j)n}=(-1)^n$.

For the front and bottom faces of the cube, we recall that we have defined  $\td\psi$ to be the sheaf-theoretic intersection pairing of Goresky and MacPherson, which takes the tensor product of preferred generators at points in $X-X^{n-1}$ to preferred generators; see Section \ref{S: O}. We then \emph{define}\label{int def1} the map $\psi$ so that the bottom square commutes up to\footnote{This sign will be necessary for the rest of the diagram to commute as desired. Our computations in Appendix \ref{S: signs} reveal that we could eliminate the $(-1)^n$ by reversing the $\sigma'$ isomorphisms, but this seems unnatural  and would nonetheless cause signs to pop up in other places.} $(-1)^n$. We claim the front square then also commutes: By Section \ref{R: uniqueness}, it is sufficient to consider what happens in $H^0$ over points  $x\in X-X^{n-1}$. But we know $\td \cup$ takes $1_x\otimes 1_x$ to $1_x$, while $\mathbb{O}(1_x)$ is, by definition, represented by the local orientation class at $x$. Then also by definition, $\td\psi(\mathbb{O}(1_x)\otimes \mathbb{O}(1_x))$ is represented by the local orientation class, which is again $\mathbb{O}(1_x)=\mathbb{O}(1_x\td\cup 1_x)$.  

Using that the left and bottom faces commute up to $(-1)^n$, that the right, front, and top commute on the nose, and that the left and  right faces  are squares of isomorphisms, a diagram chase demonstrates that the back square commutes, as desired.

\paragraph{Bottom cube. } For the remainder of this section we assume that $X$ is a compact $F$-oriented PL stratified pseudomanifold; see \cite[Section 2.5]{GBF35} for background details. We let $I^{\bar p}\mf C_*(X;F)$ and $I^{\bar p}\mf H_*(X;F)$ denote respectively the PL intersection chain complex and PL intersection homology groups of $X$ \cite[Sections 3.3 and 6.2]{GBF35}.

By \cite[Theorem 6.3.31]{GBF35}, the singular and PL intersection homology groups are isomorphic, i.e.\ $I^{\bar p}\mf H_*(X;F)\cong I^{\bar p}H_*(X;F)$. The technical details are given there only for the ``GM'' case of intersection homology (in \cite[Section 5.4]{GBF35}), but they are analogous for the ``non-GM'' case that we are using here. The proof proceeds by constructing (degree $0$) quasi-isomorphisms

$$I^{\bar p}\mf C_*(X;F)\xleftarrow{} I^{\bar p}\mf C^T_*(X;F) \to I^{\bar p}\mf S_*(X;F)\xleftarrow{} I^{\bar p}S_*(X;F).$$
Here $I^{\bar p}\mf C^T_*(X;F)$ is the direct limit of simplicial intersection chains with respect to barycentric subdivisions of an arbitrary fixed triangulation $T$ of $X$, while $I^{\bar p}\mf S_*(X;F)$ is the direct limit of the complex of singular intersection chains under barycentric subdivision. The precise details of these groups and the maps between them will not be essential for us, so we denote this zig-zag of maps by  $I^{\bar p}\mf C_*(X;F)\leftrightarrow I^{\bar p}S_*(X;F)$. As each map is a quasi-isomorphism, this induces an isomorphism $I^{\bar p}\mf H_*(X;F)\cong I^{\bar p}H_*(X;F)$.

Furthermore, the maps involved in $I^{\bar p}\mf C_*(X;F)\leftrightarrow I^{\bar p}S_*(X;F)$ restrict to open subsets  and thus there are corresponding zig-zags of maps 
$I^{\bar p}\mf C_*(X, X-\bar U;F)\leftrightarrow I^{\bar p}S_*(X, X-\bar U;F)$, each of which is a quasi-isomorphism; see \cite[Corollaries 5.4.3 and Corollary 6.3.32]{GBF35}. Consequently, the maps all sheafify, and we obtain a zig-zag of sheaf quasi-isomorphisms that we denote $\mc I^{\bar p}\ms S^*_{PL}\leftrightarrow \mc I^{\bar p}\mc S^*$. For each such sheaf complex, we employ the same indexing shifts as above for $\mc I^{\bar p}\mc S^*$.

The sheaf complex $\mc I^{\bar p}\ms S^*_{PL}(X;F)$ is not quite the usual PL intersection chain sheaf, which rather is typically defined by the presheaf $U\to I^{\bar p}\mf C^{\infty}_{n-*}(U;F)$, which  is the complex of locally finite (not necessarily compactly supported) intersection chains on $U$; see \cite[Section I and II]{Bo} or \cite[Section 2.1]{GM2}. This presheaf is in fact a sheaf \cite[page 30]{BoHab} and furthermore a soft sheaf \cite[Proposition 5.1]{BoHab}.  We will denote it $\mc I^{\bar p}\mc S_{PL}^*$.

If $V\subset U$ are open subsets of $X$ with $\bar V\subset U$ then restriction provides a well-defined chain map $I^{\bar p}\mf C_*(X, X-\bar U;F)\to I^{\bar p}\mf C^\infty_*(V;F)$. Furthermore, suppose $U\cong \R^k\times cL$ is a distinguished neighborhood of a point $x\in X$ with $x=(0,v)$ and $v$ denoting the cone vertex. If we can identify $V$ under this homeomorphism with $B^k\times c_rL$, where $B^k$ is a ball containing the origin in $\R^k$ and $c_rL =\frac{[0,r)\times L}{0\times L} \subset \frac{[0,1)\times L}{0\times L}=cL$ is a subcone, $0<r<1$, then this restriction map is a quasi-isomorphism via the standard local computations for intersection homology\footnote{Sketch of proof: In the degrees where the local intersection homology groups are not automatically $0$ they are isomorphic to the intersection homology groups of the links $I^{\bar p}\mf H_{*-k-1}(L;F)$. In particular, if $\xi$ is a PL cycle representing an element of $I^{\bar p}\mf H_{*-k-1}(L;F)$ in such an appropriate degree and we take $U$ sufficiently small to be embedded in a larger $\R^k\times cL$, then the isomorphism takes the class of $\xi$ to an element of $I^{\bar p}\mf H_{*}(X, X-\bar U;F)$ represented by a chain of the form $\eta \times \bar c\xi$, where $\eta$ is a generator of $\mf H_k(\R^k, \R^k-B^k;F)$ and $\bar c\xi$ is the closed cone on $\xi$ (see \cite[Theorem 6.3.20, Corollary 6.2.15, and the proof of Theorem 6.2.13]{GBF35}). The restriction of such a chain to $I^{\bar p}\mf C^\infty_{*}(V;F)$ then represents the corresponding homology class in $I^{\bar p}\mf H^\infty_{*}(V;F)$; cf.\ \cite[Section II.2-II.3]{BoHab}.}.  But such data is sufficient to construct a quasi-isomorphism $\mc I^{\bar p}\ms S^*_{PL}\to \mc I^{\bar p}\mc S_{PL}^*$ just as in the discussion of box VII in the proof of Theorem \ref{T: MAIN}. 
So we can extend our zig-zag of quasi-isomorphisms to $\mc I^{\bar p}\mc S^*\leftrightarrow \mc I^{\bar p}\mc S_{PL}^*$. 

This zig-zag of quasi-isomorphisms determines an isomorphism, and in particular a morphism,  $\mc I^{\bar p}\mc S^*\to \mc I^{\bar p}\mc S_{PL}^*$ in the derived category $D(X)$. Since each of these complexes is $C$-ready by Lemma \ref{L: C-ready}, we obtain a unique induced isomorphism $I^{\bar p}H_{n-*}(X;F)\cong \H^*(X;\mc I^{\bar p}\mc S^*)\to \H^*(X;\mc I^{\bar p}\mc S_{PL}^*)\cong I^{\bar p}\mf H_{n-*}(X;F)$ by Lemma \ref{L: rep}.

Now we consider the commutativity of the bottom cube in the statement of Theorem \ref{T: cubes}:

We already know that  top square  commutes up to $(-1)^n$ by definition. 

On the right side, $\sigma'$ and $\sigma''$ are the sheafification maps. The vertical maps are the compositions of the quasi-isomorphisms discussed above, and the square itself can be decomposed into the sheafification diagrams for each of the individual maps. To see that the result is in fact a commutative square of isomorphisms, let us consider  $f:A_{n-*}\to B_{n-*}$, representing the map between any two neighboring complexes in the zig-zag of chain complexes. Let $\td f:\mc A^*\to \mc B^*$ denote the corresponding map of sheaves, and let $\hat f:\mc I^*\to \mc J^*$ be the induced map of injective resolutions. Then we have a diagram

\begin{diagram}[LaTeXeqno]\label{D: comp}
H_{n-*}(A_*)&\rTo &H^*(\Gamma(X;\mc A^*))&\rTo & H^*(\Gamma(X;\mc I^*))=\H^*(X;\mc A^*)\\
\dTo^f&&\dTo^{\td f}&&\dTo^{\hat f}\\
H_{n-*}(B_*)&\rTo &H^*(\Gamma(X;\mc B^*))&\rTo & H^*(\Gamma(X;\mc J^*))=\H^*(X;\mc B^*).
\end{diagram}
Since the left square represents sheafification and the right square is induced from injective resolutions, these diagrams all commute. Furthermore, since the leftmost and rightmost vertical maps are isomorphisms, if either the top or bottom composition is an isomorphism, so is the other. Since we know that the composition $I^{\bar p}H_{n-*}(X;F)\to H^*(\Gamma(X;\mc I^{\bar p}\mc S^*))\to H^*(\Gamma(X;\mc I^*))$ consists of isomorphisms, it follows that all of the horizontal compositions are isomorphisms. We also know that the maps $I^{\bar p}\mf H_{n-*}(X;F)\to H^*(\Gamma(X;\mc I^{\bar p}\mc S_{PL}^*))\to H^*(\Gamma(X;\mc K^*))$ are isomorphisms, letting $\mc K^*$ be an injective resolution of $\mc I^{\bar p}\mc S_{PL}^*$. So stacking together all of these diagrams gives us the commutativity of the right face of the cube with the front vertical map being the isomorphism $H^*(\Gamma(X;\mc I^{\bar p}\mc S^*))\to H^*(\Gamma(X;\mc I^*))$ followed by the sequence of isomorphisms corresponding to the right sides of the various versions of  diagram \eqref{D: comp} and then finally the inverse of the isomorphism $H^*(\Gamma(X;\mc I^{\bar p}\mc S_{PL}^*))\to H^*(\Gamma(X;\mc K^*))$. Since the injective complexes form a zig-zag of quasi-isomorphisms from $\mc I^{\bar p}\mc S^*$ to $\mc I^{\bar p}\mc S_{PL}^*$ that commutes with our original zig-zag of quasi-isomorphisms, this indeed represents the hypercohomology isomorphism shown in the cube.

The left face of the cube consists of the tensor product of two versions of the right face. Note that all vertical maps are degree $0$ while all sheafification maps have degree $n$. Since both horizontal tensor products of maps create the same sign $(-1)^{n(n-i)}$ (and similarly for all the intermediate maps as in the preceding paragraph), the lefthand square commutes exactly.

We consider now the bottom of the cube. The map labeled $\psi$ is the PL intersection product.  In \cite[Theorem 1]{GBF39}, it is show that this product can be defined via the following chain maps: 
\begin{diagram}[LaTeXeqno]\label{D: pf}
I^{D\bar p}\mf C_*^{\infty}(X;F)\otimes I^{D\bar q}\mf C^{\infty}_*(X;F)&\lTo & G^{\infty,P}_*(X;F)&\rTo^{\mu} & I^{D\bar p+D\bar q}\mf C^{\infty}_{*-n}(X;F).
\end{diagram}
Here $P$ simply represents the pair of perversities $D\bar p,D\bar q$ and the leftward arrow is an inclusion and a quasi-isomorphism. The map $\mu$ is a composition of the PL chain cross product and a certain umkehr map of (homological) degree $-n$. 
The PL intersection product is then the composition of the canonical map $I^{D\bar p}\mf H_*^{\infty}(X;F)\otimes I^{D\bar q}\mf H^{\infty}_*(X;F)\to H_*(I^{D\bar p}\mf C_*^{\infty}(X;F)\otimes I^{D\bar q}\mf C^{\infty}_*(X;F))$, the map on homology induced by \eqref{D: pf}, and 
 the map 
$I^{D\bar p+D\bar q}\mf H^{\infty}_*(X;F)\to I^{D\bar r}\mf H^{\infty}_*(X;F)$ that exists because $D\bar r\geq D\bar p+D\bar q$. By \cite[Proposition 6.9]{GBF39}, this product agrees with the Goresky-MacPherson intersection product of \cite{GM1}, possibly up to sign conventions.

It is furthermore shown in \cite[Proposition 6.7]{GBF39} that the maps in Diagram \eqref{D: pf} commute with restriction. So, tacking on the map $\mc I^{D\bar p+D\bar q}\mc S^{n-*}_{PL}\to \mc I^{D\bar r}\mc S^{n-*}_{PL}$, the diagram \eqref{D: pf} sheafifies to
\begin{diagram}
\mc I^{D\bar p}\mc S^{n-*}_{PL}\otimes \mc I^{D\bar q}\mc S^{n-*}_{PL}&\lTo & \mc G^P_*&\rTo^{\mu} & \mc I^{D\bar r}\mc S^{n-*}_{PL}.
\end{diagram}
 Let us denote\footnote{The sign here corresponds to the similar one needed for middle horizontal square in the diagram of Theorem \ref{T: cubes}.} $(-1)^n$ times the resulting morphism in $D(X)$ by $\td \pf$. 

We claim that the bottom of the cube corresponds to the following diagram induced by sheafification and taking cohomology (eliminating the $\infty$ decorations because $X$ is compact and leaving $F$ tacit):

\begin{equation}\label{D: pf sheaf}
\resizebox{.95\hsize}{!}{
 \begin{diagram}
I^{D\bar p}\mf H_{n-i}(X)\otimes I^{D\bar q}\mf H_{n-j}(X)&\rTo
&H_{2n-i-j}(I^{D\bar p}\mf C_*(X)\otimes I^{D\bar q}\mf C_*(X))&\lTo^\cong &H_{2n-i-j}(G^{P}_*(X))&\rTo^{\mu} & I^{D\bar r}\mf H_{n-i-j}(X)\\
\dTo&&\dTo^{\sigma''\otimes \sigma''}&&\dTo&&\dTo&&\\
H^i(\Gamma(\mc I^{D\bar p}\mc S^*_{PL}))\otimes H^j(\Gamma(\mc I^{D\bar q}\mc S^*_{PL}))&\rTo
&H^{i+j}(\Gamma(\mc I^{D\bar p}\mc S^*_{PL}\otimes \mc I^{D\bar q}\mc S^*_{PL}))&\lTo^\cong &H^{i+j}(\Gamma(\mc G^{P,*}))&\rTo^{\mu} & H^{i+j}(\Gamma(\mc I^{D\bar r}\mc S^*_{PL})).
\end{diagram}}
\end{equation}
This diagram commutes up to $(-1)^n$ by definition. The top composition is $\pf$ by definition. We will show that the bottom corresponds to the front bottom map in the cube.

Since $\mc I^{D\bar p}\mc S^*_{PL}$ is soft and flat, the tensor product $\mc I^{D\bar p}\mc S^*_{PL}\otimes \mc I^{D\bar q}\mc S^*_{PL}$ is soft \cite[Proposition V.6.5]{Bo}. Thus as $H^i(\mc I^{D\bar p}\mc S^*_{PL,x})=0$ for large enough $|i|$,  the tensor product $\mc I^{D\bar p}\mc S^*_{PL}\otimes \mc I^{D\bar q}\mc S^*_{PL}$ is $C$-ready by Lemma \ref{L: C-ready}. So by the proof of Lemma \ref{L: rep} the composition along the bottom from $H^{i+j}(\Gamma(\mc I^{D\bar p}\mc S^*_{PL}\otimes \mc I^{D\bar q}\mc S^*_{PL}))$ to $H^{i+j}(\Gamma(\mc I^{D\bar r}\mc S^*_{PL}))$ represents the hypercohomology map $\H^{i+j}(X;\mc I^{D\bar p}\mc S^*_{PL}\otimes \mc I^{D\bar q}\mc S^*_{PL})\to \H^{i+j}(X;\mc I^{D\bar r}\mc S^*_{PL})$ 
induced by $\td \pf$. So composing with the canonical map $\H^i(X;\mc I^{D\bar p}\mc S^*_{PL})\otimes \H^j(X;\mc I^{D\bar q}\mc S^*_{PL})\to 
\H^{i+j}(X;\mc I^{D\bar p}\mc S^*_{PL}\otimes \mc I^{D\bar q}\mc S^*_{PL})$, we get precisely the front bottom map of the cube. 

Finally, we show that the front face of the cube commutes. The front face is the composition of the canonically commuting square 

\begin{diagram}
 \H^{i} (X;\mc I^{D\bar p}\mc S^*)\otimes \H^j(X;\mc I^{D\bar q}\mc S^*)&\rTo & \H^{i+j}(\mc I^{D\bar p}\mc S^*\otimes \mc I^{D\bar q}\mc S^*)\\
 \dCorresponds&&\dCorresponds\\
 \H^{i} (X;\mc I^{D\bar p}\mc S_{PL}^*)\otimes \H^j(X;\mc I^{D\bar q}\mc S_{PL}^*)&\rTo & \H^{i+j}(\mc I^{D\bar p}\mc S_{PL}^*\otimes \mc I^{D\bar q}\mc S_{PL}^*)
\end{diagram}
with the maps on hypercohomology induced by the diagram

\begin{diagram}
 \mc I^{D\bar p}\mc S^*\otimes \mc I^{D\bar q}\mc S^*&\rTo^{\td \psi}&\mc I^{D\bar r}\mc S^*\\
 \dCorresponds&&\dCorresponds\\
 \mc I^{D\bar p}\mc S_{PL}^*\otimes \mc I^{D\bar q}\mc S_{PL}^*&\rTo^{\td \pf}&\mc I^{D\bar r}\mc S_{PL}^*.
\end{diagram}
Note that all the maps in these diagrams have degree $0$, so they do not introduce any signs. To verify commutativity of this last square, it suffices to show that both compositions from top left to bottom right represent the same morphism in $D(X)$, and for this it suffices to consider the degree $0$ cohomology stalks at points of $X-X^{n-1}$ by Section \ref{R: uniqueness}.

As previously noted, these cohomology stalks are $H^0(\mc I^{\bar s}\mc S^*_x)\cong F$, for any $\bar s$, and they are generated by cycles representing the local orientation class of $X$ at $x$. Tracing through the isomorphisms from singular to PL homology (which are defined precisely in \cite[Section 5.4]{GBF35}), a singular chain representative of the local orientation class will be taken to a corresponding PL  representative of the local orientation class\footnote{In fact, any $n$-simplex of $T$ oriented compatibly with $X$ and containing $x$ in its interior can be used to represent all of the local orientation classes.}. We also know that $\td \psi$ takes the tensor product of local orientation classes to a local orientation class. So it remains to see that $\td \pf$ does so as well. 

Let $U$ be a Euclidean neighborhood of some point $x\in X-X^{n-1}$, and suppose $U$ is triangulated\footnote{There are some additional technical requirements imposed on the triangulation $T$ in \cite{GBF39}. In particular, our triangulation must be the restriction to the diagonal of a triangulation of $U\times U$, but this can always be arranged (especially since we are free to rechoose $x$). See \cite[Section 5]{GBF39} for more details.} so that $x$ is in the interior of some $n$-simplex $\tau$. Let $\gamma$ be a PL cycle  in $\mf C_n^{\infty}(U;F)=I^{D\bar p}\mf C_n^{\infty}(U;F)=I^{D\bar q}\mf C_n^{\infty}(U;F)$ representing the fundamental class of $\mf H_n^{\infty}(U;F)=I^{D\bar p}\mf H_n^{\infty}(U;F)=I^{D\bar q}\mf H_n^{\infty}(U;F)$. Then the sheafification  $\sigma''(\gamma)$ represents our preferred generator of each $H^0(\mc I^{\bar s}\mc S^*_{PL,x})$.

 We can compute $\gamma\pf\gamma$ using the tools of \cite{GBF39}. In particular, it will suffice  to compute the 
coefficient $I_\tau$ of $\tau$ in $\gamma\pf\gamma$. The computation of such \emph{intersection coefficients} is described in \cite[Section 5]{GBF39}. By definition, we take the product $\gamma\times \gamma$ 
and then apply the formula of  \cite[Definition 5.3]{GBF39}. For this we let $T$ in this formula be our given triangulation, and we can take $Z=|\tau|$. If $\tau$ is oriented to agree with the orientation of $X$, then, roughly speaking, the computation proceeds as follows:
Noting that $\tau$ must appear in $\gamma$ with coefficient $1$, the product $\gamma\times \gamma$ is represented in a neighborhood of $(x,x)$ by $\tau\times \tau$ considered  as an element of $H_{2n}(|\tau|\times |\tau|, \bd (|\tau|\times |\tau|);F)$ (following an excision). This dualizes by (signed) Lefschetz duality to\footnote{Since the duality map in the definition is the inverse to the duality map from cohomology to homology, there is no sign because $1$ is a $0$-cochain.} $1\in H^0(\text{int}(|\tau|\times |\tau|);F)$, which pulls pack under the diagonal to $1\in H^0(\text{int}(|\tau|);F)$. Finally, this dualizes back to the generator of $H_n(|\tau|,\bd|\tau|;F)$ consistent with the orientation. Thus the coefficient is $I_\tau=1$. It follows that $\gamma\pf\gamma$ must be $\gamma$ itself, as this is the only element of  $\mf H_n^{\infty}(U;F)$ that can have $1$ as the coefficient of $\tau$ in a representing cycle. 

So from the definition of $\td\pf$ and the commutativity of Diagram \eqref{D: pf sheaf}, substituting $U$ for $X$,  that the sheafification  $\sigma''(\gamma\pf\gamma)=\sigma''(\gamma)$ is $(-1)^n$ times $\td\pf(\sigma''\otimes \sigma'')(\gamma\otimes \gamma)=(-1)^n\td \pf(\sigma''(\gamma)\otimes \sigma''(\gamma))=(-1)^n\sigma''(\gamma)\td\pf \sigma''(\gamma)$. In other words, $\sigma''(\gamma)\td\pf \sigma''(\gamma)=\sigma''(\gamma)$, as desired.

We have now shown that the top and bottom faces of the cube commute up to $(-1)^n$ while the left, right, and front commute exactly. Using that the left and right faces are squares of isomorphisms, a diagram chase now shows that the back face commutes. \hfill\qedsymbol

\section{Classical duality and Verdier duality}\label{S: Verdier}

In this section we prove Theorem \ref{T: Verdier} and Corollary \ref{C: Verdier} as stated in the Introduction.
We continue to assume $F$ is a field and $X$ is a compact $F$-oriented $n$-dimensional topological stratified pseudomanifold. Let $\bar p$ be a general perversity on $X$. 

In this setting, the Poincar\'e duality statement for intersection homology is usually obtained by using the Deligne sheaf axioms to establish that the shifted Verdier dual of the perversity $\bar p$ Deligne sheaf, $\mc D\mc P^*_{\bar p}[-n]$, is quasi-isomorphic to the perversity $D\bar p$ Deligne sheaf, $\mc P_{D\bar p}^*$; see \cite[Section 5.3]{GM2} or \cite[Section V.9.B]{Bo}. Then one of the basic properties of the Verdier dual is that $\H^*(X;\mc D\ms S^*)\cong \Hom(\H^{-*}(X;\ms S^*),F)$; this can be observed most easily by using the definition of $\mc D\ms S^*$ corresponding to what we have here called\footnote{In fact, this is taken to be the initial definition of the Verdier dual in \cite[Section V.7.7]{Bo}.}  $\ms L\ms S^*$ (assuming $\ms S^*$ is soft) and then simply applying the Universal Coefficient Theorem\footnote{If $X$ is not compact, then the correct statement is $\H^*(X;\mc D\ms S^*)\cong \Hom(\H^{-*}_c(X;\ms S^*),F)$.}. Therefore it follows that 
\begin{align*}
\H^*(X;\mc P_{D\bar p}^*)&\cong \H^*(X;\mc D\mc P^*_{\bar p}[-n])\\
&=\H^{*-n}(X;\mc D\mc P^*_{\bar p})\\
&\cong \Hom(\H^{n-*}(X;\mc P^*_{\bar p}),F).
\end{align*}
This isomorphism $\H^*(X;\mc P_{D\bar p}^*)\cong \Hom(\H^{n-*}(X;\mc P^*_{\bar p}),F)$ is what is often called intersection homology Poincar\'e duality. In the special case where $M$ is a compact oriented manifold, then both $\mc P_{\bar p}^*$ and $\mc P^*_{D\bar p}$ are resolutions of the constant sheaf $\mc F$, and $\mc D\mc F\cong \D^*$, the Verdier dualizing sheaf. Then $\H^*(X;\mc F)\cong H^*(X;F)$, and the statement of Poincar\'e duality  becomes $H^*(X;F)\cong \Hom(H^{n-*}(X;F),F)$. 

What is unclear in this approach to Poincar\'e duality is that the duality isomorphism relates to that of more classical approaches to duality, such as  the cap product with the fundamental class, via geometrically meaningful maps. However, the proof of Theorem \ref{T: MAIN} demonstrates such a compatibility,
provided by putting together the boxes I through III of Diagram \eqref{D: MAIN}. Assembling these, we obtain the first commutative diagram of isomorphisms of Theorem \ref{T: Verdier}: 

\begin{diagram}[LaTeXeqno]\label{D: IC Verdier}
I_{\bar p}H^i(X; F)&&\rTo^{(-1)^{in}\cdot\cap \Gamma} &&I^{D\bar p}H_{n-i}(X; F)\\
\dTo^{\sigma}&&&&\dTo_{\kappa}\\
\H^i(X;\mc I_{\bar p}\mc C^*)&\rTo^{\hat \Phi} &\H^{i}(X;\mc D\mc I_{D\bar p}\mc C^*[-n])&\rTo^{\sigma^*e}&\Hom(I_{D\bar p}H^{n-i}(X;F),F).
\end{diagram}
The map on the left is sheafification, the map on top is the cap product Poincar\'e duality isomorphism, and the map on the right is the signed Kronecker evaluation such that $\kappa(x)(\alpha)=(-1)^{n-i}\alpha(x)$. The composition along the bottom is induced by a quasi-isomorphism $\mc I_{\bar p}\mc C^*\to \mc D\mc I_{D\bar p}\mc C^*[-n]$, the universal coefficient isomorphism for Verdier duals, and  the dual of sheafification.

When our space is a compact $F$-oriented manifold $M$,  the sheaf of cochains is  a resolution of the constant sheaf $\mc F$, and the diagram \eqref{D: IC Verdier} becomes the diagram of 
Corollary \ref{C: Verdier}:

\begin{diagram}
H^i(X; F)&&\rTo^{(-1)^{in}\cdot\cap \Gamma} &&H_{n-i}(X; F)\\
\dTo^{\sigma}&&&&\dTo_{\kappa}\\
\H^i(X;\mc F)&\rTo^{\hat \Phi} &\H^{i}(X;\D^*[-n])&\rTo^{\sigma^*e}&\Hom(H^{n-i}(X;F),F).
\end{diagram}
This provides a compatibility between classical Poincar\'e duality (via the cap product) and Verdier duality. We note  that the point is not simply that the two duality isomorphisms can be put together in a diagram (which can always be done) but that they are related by the ``natural''  isomorphisms given by sheafification, Kronecker evaluation, and the Universal Coefficient Theorem.

Alternatively, if we wish instead to have an analogous diagram built from a chain (as opposed to cochain) representation of the Deligne sheaf, we claim that the following diagram commutes, providing the remainder of 
Theorem \ref{T: Verdier}:

\begin{diagram}[LaTeXeqno]\label{D: IS Verdier}
 I^{D\bar p}H_{n-i}(X;F) & &\lTo^{(-1)^{in}\cdot\cap\Gamma} & &I_{\bar p}H^i(X;F)    \\
 \dTo^{\sigma'}  & &&  &  \dTo_{\kappa'}  \\
\H^{i} (X;\mc I^{D\bar p}\mc S^*)  &  \rTo^{\hat{\Psi}}  &       \H^{i} (X;\mc D\mc I^{\bar p}\mc S^*[-n])    &      \rTo^{{\sigma'}^*e} &     \Hom(I^{\bar p}H_{i}(X;F),F) ,\\
\end{diagram}
Here $\sigma'$ is sheafification, ${\sigma'}^*$ is the $\Hom$ dual of sheafification, $e$ is the Verdier duality universal coefficient map, $\kappa'$ is the Kronecker evaluation, and $\hat{\Psi}$ is the adjoint to the composition of the Goresky-MacPherson sheaf product $\td\psi$ with the  morphism $\mathbb{K}:\mc I^{\bar t}\mc S^*\to \D^*[-n]$ of Section \ref{S: O}. 

To prove the commutativity of \eqref{D: IS Verdier}, we can construct the following diagram of \emph{isomorphisms}. The top face is the commutative diagram \eqref{D: IC Verdier}, which we already know commutes, and the bottom is diagram \eqref{D: IS Verdier}. Thus we need only show the vertical faces commute. 
\bigskip

\begin{equation*}
\resizebox{.8\hsize}{!}{
\begin{diagram}
  &I_{\bar p}H^i(X)  &    &\rTo^{(-1)^{in}\cdot\cap\Gamma} &      &    I^{D\bar p}H_{n-i}(X)  \\
 \ldTo(1,2)_{\sigma} & \vLine   &  &      &      & \uTo^{(-1)^{in}\cdot\cap\Gamma}& \rdTo^{\kappa}  \\
\H^{i} (X;\mc I_{\bar p}\mc C^*) & \HonV  & \rTo^{\hat \Phi}   &   \H^{i} (X;\mc D\mc I_{D\bar p}\mc C^*[-n])   & \rTo^{\sigma^*e}  & \HonV   &    &  \Hom(I_{D\bar p}H^{n-i}(X),F)  \\
  &    \dTo_{(-1)^{in}\cdot\cap\Gamma} &     & \uTo_{\mc D\mathbb{O}[-n]} &      & \vLine  &&\uTo_{D((-1)^{(n-i)n}\cdot\cap\Gamma)} \\
 &  I^{D\bar p}H_{n-i}(X) & \lTo & \VonH   & \hLine^{(-1)^{in}\cdot\cap\Gamma} & I_{\bar p}H^i(X)  &    &  \\
 \dTo^{\mathbb{O}}\ldTo(1,2)_{\sigma'}  &   &  &      &      &      & \rdTo^{\kappa'}  \\
\H^{i} (X;\mc I^{D\bar p}\mc S^*)  &  \rTo^{\hat{\Psi}}  &    &   \H^{i} (X;\mc D\mc I^{\bar p}\mc S^*[-n])    &      & \rTo^{{\sigma'}^*e} &    &   \Hom(I^{\bar p}H_{i}(X),F) .\\
\end{diagram}}
\end{equation*}

\medskip

The back face is ``commutative'' in a canonical sense. The commutativity of the left side is Theorem \ref{T: MAIN}. On the right side, $\kappa$ and $\kappa'$ are the Kronecker evaluations (though recall $\kappa$ has a sign - see the footnote on page \pageref{pageref kappa} and Appendix \ref{Appendix}) and $D((-1)^{in}\cdot\cap\Gamma)$ is the $\Hom$ dual of the Poincar\'e duality map. To see that the right side commutes, let $\alpha\in I_{\bar p}H^{i}(X;F)$, $\beta\in I_{D\bar p}H^{n-i}(X;F)$,
$x=\alpha\cap\Gamma$, and $y=\beta\cap\Gamma$. Then
\begin{align*}
(-1)^{in}(\kappa\circ(\cdot\cap\Gamma)(\alpha))(\beta)&= (-1)^{in}(\kappa (x))(\beta)\\
&=(-1)^{in+n-i}\beta(x)\\
&=(-1)^{in+n-i}\aug(\beta\cap x)&\text{by \cite[Proposition 7.3.25]{GBF35}}\\
&=(-1)^{in+n-i}\aug( \beta\cap(\alpha\cap\Gamma))\\
&=(-1)^{in+n-i}\aug( (\beta\cup \alpha)\cap\Gamma)&\text{by \cite[ Proposition 7.3.35]{GBF35}},
\end{align*}
while
\begin{align*}
(D((-1)^{(n-i)n}\cdot\cap\Gamma)\circ \kappa' (\alpha))(\beta)&=(-1)^{(n-i)n}(D(\cdot\cap\Gamma)\circ \kappa' (\alpha))(\beta)\\
&=(-1)^{(n-i)n+in} (\kappa' (\alpha))(\beta\cap\Gamma)\\
&=(-1)^{n}(\kappa' (\alpha))(y)\\
&=(-1)^n \alpha(y)\\
&=(-1)^n \aug(\alpha\cap y)&\text{by \cite[Proposition 7.3.25]{GBF35}}\\
&=(-1)^n \aug(\alpha\cap(\beta\cap\Gamma))\\
&=(-1)^n \aug((\alpha\cup \beta)\cap\Gamma)&\text{by \cite[ Proposition 7.3.35]{GBF35}}\\
&=(-1)^{n+i(n-i)} \aug((\beta\cup \alpha)\cap\Gamma)&\text{by \cite[Proposition 7.3.15]{GBF35}}.
\end{align*}
Thus the right side commutes.

In the right front square, $e$ denotes the Verdier duality universal coefficient map on both the top and bottom, while $\sigma^*$ and ${\sigma'}^*$ are the $\Hom$ duals of the two sheafifications maps. This face is then the composition of squares

\begin{diagram}
 \H^{i} (X;\mc D\mc I_{D\bar p}\mc C^*[-n])   & \rTo^{e}  & \Hom( \H^{n-i} (X;\mc I_{D\bar p}\mc C^*),F)   & \rTo^{\sigma^*}   &  \Hom(I_{D\bar p}H^{n-i}(X),F)  \\
 \uTo_{\mc D\mathbb{O}[-n]} &      & \uTo^{\mathbb{O}^*} &&\uTo_{D((-1)^{(n-i)n}\cdot\cap\Gamma)} \\
  \H^{i} (X;\mc D\mc I^{\bar p}\mc S^*[-n])    &    \rTo^{e}  & \Hom( \H^{n-i} (X;\mc I^{\bar p}\mc S^*),F)  &  \rTo^{\sigma'^*}  &   \Hom(I^{\bar p}H_{i}(X),F) .\\
\end{diagram}
The right square is just the $\Hom$ dual of Theorem \ref{T: MAIN}, while the left commutes by the naturality of the universal coefficient evaluation.

Finally, to see the commutativity of the front left face, it is sufficient by Section \ref{R: uniqueness} to check that the diagram of sheaf maps commutes on $H^0$ at each point $x\in X-X^{n-1}$. 
Let $1_x\in \mc I_{\bar p}\mc C^0_x$ be the germ of the unit $0$-cochain at $x$. Then $(\mc D\mathbb{O}[-n]\circ\hat \Psi\circ \mathbb{O})(1_x)=(\mc D\mathbb{O}[-n]\circ\hat \Psi)(\gamma_x)$, where $\gamma_x\in \mc I^{\bar q}\mc S^0$ is a germ representing  the local orientation class at $x$. Then, by definition, $\hat \Psi(\gamma_x)$ takes the germ of an intersection cycle $\xi_x$ at $x$ to the image of the sheaf theoretic intersection $\td\psi(\gamma_x\otimes \xi_x)$ under the map $\mathbb{K}: \mc I^{\bar t}\mc S^*_x \to \D^*_x[-n]$. Noting that the codomain of  $\hat \Psi$ is  $\mc D \mc I^{\bar p}\mc  S^* [-n]\cong \SHom(\mc I^{\bar p}\mc  S^* ,\mathbb D^*[-n])$, we interpret $\mc D\mathbb{O}[-n]$ as the $\SHom( \cdot ,\mathbb D^*[-n])$ dual of the degree $0$ map $\mathbb{O}$. Then we have 
\begin{align*}
((\mc D\mathbb{O}[-n]\circ\hat \Psi\circ \mathbb{O})(1_x))(1_x)
&=((\mc D\mathbb{O}[-n]\circ\hat \Psi)(\gamma_x))(1_x)\\
&=(\hat \Psi(\gamma_x)) \mathbb{O}(1_x)\\
&=(\hat \Psi(\gamma_x))(\gamma_x)\\
&=\mathbb{K}(\td\psi(\gamma_x\otimes\gamma_x))\\
&=\mathbb{K}(\gamma_x)\\
&=\mathbb{K}(\gamma_x),
\end{align*}
as $\td\psi(\gamma_x\otimes\gamma_x)=\gamma_x$ in $H^0$ by the definition of $\td \psi$ in Section \ref{S: O}.

 On the other hand, $\hat \Phi(1_x)$ acts by the cup product, so $(\hat \Phi(1_x))(1_x)$ is the image of $1_x\td \cup 1_x=1_x$ under $\mathbb{L}$ in $H^0(\D^*_x[-n])$, i.e. $$(\hat \Phi(1_x))(1_x)=\mathbb{L}(1_x).$$ But this is sufficient to prove the result by  the commutativity of diagram \eqref{D: O triangle}, since $\mathbb{O}$ takes the $H^0$ class represented by $1_x$ to the $H^0$ class represented by $\gamma_x$.

This completes our proofs of Theorem \ref{T: Verdier} and Corollary \ref{C: Verdier}.\hfill\qedsymbol

\appendix

\section{Co-cohomology}\label{Appendix} We  need some results on the cohomology of double duals. These results are no doubt well known. We place them in this appendix for the convenience of the reader.

Let $A^*$ be a complex of vector spaces over the field $F$.

\begin{lemma}\label{L: coco UCT} 
Let $f$ be the homomorphism $A^*\to \Hom^*(\Hom^*(A^*,F),F)$ defined such that if $x\in A^*$ and $\alpha\in \Hom^*(A^*,F)$, then  $f(x)(\alpha)=(-1)^{|\alpha|}\alpha(x)$. Then $f$ is a degree $0$ chain map. 
If $H^*(A^*)$ is finitely generated in each dimension and vanishes for $*$ sufficiently large, then $f$  is an isomorphism.
\end{lemma}
\begin{proof}
Note that $\Hom(A^i, F)$ consists of degree $-i$ maps and so $\Hom(\Hom(A^i, F),F)$ consists of degree $i$ maps. Since $f$ takes elements of $A^i$ to elements of $\Hom(\Hom(A^i, F),F)$, it is a degree $0$ homomorphism. 
To see that $f$ is a degree $0$  chain map, we compute that $f(dx)(\alpha)=(-1)^{|\alpha|}\alpha(dx)$ and (letting $\bar A^*=\Hom^*(A^*,F)$) 
\begin{align*}
d(f(x))(\alpha)&= 
(d_F\circ f(x)-(-1)^{|f(x)|}f(x)\circ d_{\bar A})(\alpha)\\
&=(-1)^{|f(x)|+1}f(x)\circ d_{\bar A}(\alpha)\\
&=(-1)^{|f(x)|+1+|\alpha|+1}(d_{\bar A}\alpha)(x)\\
&=(-1)^{|f(x)|+1+|\alpha|+1+|\alpha|+1}\alpha(dx)\\
&=(-1)^{|f(x)|+1}\alpha(dx)\\
&= (-1)^{|\alpha|}\alpha(dx).
\end{align*}
The second equality is because the boundary is trivial in the complex $F$. The last equality is because the entire expression will all be $0$ unless $|\alpha|=|dx|$, and $|dx|=|x|+1=|f(x)|+1$. Comparing expressions, we see that $f$ commutes with $d$ and so is a degree $0$ chain map. 

Suppose now that $H^*(A^*)$ is finitely generated in each dimension and that $A^*$ is bounded above. Then by \cite[Lemma 56.3 and Theorem 46.2]{MK}, there is a chain homotopy equivalence $\phi: B^*\to A^*$ with $B^*$ a complex finitely generated in each dimension. The map $\phi$ induces a diagram
\begin{diagram}[LaTeXeqno]\label{D: coco}
H^*(B^*)&\rTo^{f'} & H^*(\Hom^*(\Hom^*(B^*,F),F))\\
\dTo^\phi&&\dTo^{\phi^{**}}\\
H^*(A^*)&\rTo^f & H^*(\Hom^*(\Hom^*(A^*,F),F)).
\end{diagram}
This diagram commutes, since for $x\in B^*$ and $\alpha\in \Hom^*(A^*,F)$, we have $(f\phi)(x)(\alpha)=(-1)^{|\alpha|}\alpha(\phi(x))$, while $(\phi^{**})f'(x)(\alpha)=f'(x)(\phi^*(\alpha))=(-1)^{|\alpha|}(\phi^*(\alpha))(x)=(-1)^{|\alpha|}\alpha(\phi(x))$. Since $\phi$ is a chain homotopy equivalence, so are $\phi^*$ and $\phi^{**}$, so the vertical maps are both isomorphisms. Furthermore, the top map is an isomorphism since $B^*$ finite implies that $f': B^* \to \Hom^*(\Hom^*(B^*,F),F)$ is actually an isomorphism. Thus $f$ is a homology isomorphism.
\end{proof}

\section{A meditation on signs}\label{S: signs}

In Section \ref{S: cap/cup} we are primarily concerned with the relationship between products (the cup product and intersection product) on the (co)homology of two different chain complexes (the intersection cochain and intersection chain complexes) related by a degree $-n$ chain map (the duality homomorphism). Due to the degree shift, even more signs come into play than usual\footnote{In general, signs are unavoidable. See \cite[Section A.1]{GBF35} for a discussion.}, and it turns out that they cause some headaches that we here illustrate. To keep other distractions to a minimum, we strip down our scenario and consider the following as our given data:

\begin{enumerate}
\item Cohomologically indexed chain complexes $A^*$ and $B^*$ with a degree $n$ chain \emph{isomorphism} $f:A^*\to B^*$. 

\item A degree $0$ chain map $P: A^*\otimes A^*\to A^*$ that is associative,  (graded) commutative, and unital. For simplicity, we write $P(x\otimes y)=x\boxplus_A y\in A^*$.  

\end{enumerate}
The question then is how to define a compatible product  $Q:B^*\otimes B^*\to B^*$, for which we will write $Q(a\otimes b)=a\boxtimes_B b$. The upshot is that if we want to transfer the product using the degree $n$ chain maps then the nice properties assumed for $P$ will only carry over to $Q$ with some unpleasant sign corrections.

\paragraph{Defining the transferred product.}
Starting with the  diagram

\begin{diagram}
A^*\otimes A^*&\rTo^P &A^*\\
\dTo^{f\otimes f}&&\dTo_f\\
B^*\otimes B^*&\rDashto^Q &B^*,
\end{diagram}
a natural first definition for $Q$ would be the $Q$ that makes the diagram commute: $Q=fP(f\otimes f)^{-1}$. Assuming that $f$ is a degree $n$ chain map and $P$ is a degree $0$ chain map, then $f\otimes f$ is a degree $2n$ chain map \cite[Section A.1.5]{GBF35} and so  $(f\otimes f)^{-1}$ is a degree $-2n$ chain map. So 
$Q$ is a chain map of degree $-n$.

To apply this to elements, though, we first need to note that $(f\otimes f)^{-1}=(-1)^nf^{-1}\otimes f^{-1}$, which we can verify using the basic sign properties of degree $n$ maps \cite[Section A.1]{GBF35}:
\begin{align*}
(-1)^n(f^{-1}\otimes f^{-1})(f\otimes f)(x\otimes y)&=(-1)^n(f^{-1}\otimes f^{-1})((-1)^{n|x|}f(x)\otimes f(y))\\
&=(-1)^{n+n|x|}(-1)^{n|f(x)|}f^{-1}f(x)\otimes f^{-1}f(y)\\
&=(-1)^{n+n|x|+n(|x|+n)}x\otimes y\\
&=x\otimes y,
\end{align*}
using that $|f(x)|=|x|+n$. So we see that the inverse for $f\otimes f$ is indeed $(-1)^n f^{-1}\otimes f^{-1}$. 
To simplify notation, let $g=f^{-1}$ so that $(f\otimes f)^{-1}=(-1)^ng\otimes g$.

So now if $a,b\in B^*$, we obtain 

\begin{align*}
a\boxtimes_B b&=fP(f\otimes f)^{-1}(a\otimes b)\\
&=(-1)^nfP(g\otimes g)(a\otimes b)\\
&=(-1)^{n+n|a|}fP(g(a)\otimes g(b))\\
&=(-1)^{n+n|a|}f( g(a)\boxtimes_A g(b)).
\end{align*}

One immediate curiosity arises if we instead use the diagram 

\begin{diagram}
A^*\otimes A^*&\rTo^P &A^*\\
\uTo^{g\otimes g}&&\uTo_g\\
B^*\otimes B^*&\rDashto^{Q'} &B^*,
\end{diagram}
which we obtain by privileging the inverse isomorphism $g:B^*\to A^*$ to define $Q'$. In this case $Q'=g^{-1}P(g\otimes g)=fP(g\otimes g)$, which is again a degree $-n$ chain map.
But in this case,

\begin{align}
a\boxtimes'_B b&=fP(g\otimes g)(a\otimes b)\notag\\
&=(-1)^{n|a|}fP(g(a)\otimes g(b))\label{E: transfer}\\
&=(-1)^{n|a|}f( g(a)\boxtimes_A g(b))\notag.
\end{align}
So $Q(a\otimes b)$ and $Q'(a\otimes b)$ defer by a sign of $(-1)^n$, which is not surprising given our previous observation about the relationship between $(f\otimes f)^{-1}$ and $f^{-1}\otimes f^{-1}$. 

 This leads us to our first moral:

\begin{itemize}
\item The direction of the chain isomorphism between $A^*$ and $B^*$ matters for transferring the pairing, but only up to a sign $(-1)^n$, where $\pm n$ is the degree of the isomorphism.
\end{itemize}

\paragraph{Sign problems.}
At first the second product, $Q'$, seems somewhat more appealing. It eliminates a $(-1)^n$ of questionable necessity that would also seem to be problematic in iterating the product, since each time we iterate we pick up an extra $(-1)^n$ compared to the product $P$. Unfortunately, both $Q$ and $Q'$ have a bigger problem with iterations, as we see by considering associativity. Recall that we assume that $P$ is associative. Using \eqref{E: transfer},

\begin{align*}
(a\boxtimes'_B b)\boxtimes'_B c&= (-1)^{n|a\boxtimes'_B b|}f( g(a\boxtimes'_B b)\boxtimes_A g(c))\\
&=(-1)^{n(|a|+|b|-n)}f( g((-1)^{n|a|}f( g(a)\boxtimes_A g(b)))\boxtimes_A g(c))\\
&=(-1)^{n(|a|+|b|-n)+n|a|}f( ( g(a)\boxtimes_A g(b))\boxtimes_A g(c))\\
&=(-1)^{n+n|b|}f( ( g(a)\boxtimes_A g(b))\boxtimes_A g(c))\\
a\boxtimes'_B (b\boxtimes'_B c)&= (-1)^{n|a|}f( g(a)\boxtimes_A g(b\boxtimes'_B c))\\
&=(-1)^{n|a|}f( g(a)\boxtimes_A g((-1)^{n|b|}f( g(b)\boxtimes_A g(c))))\\
&=(-1)^{n|a|+n|b|}f( g(a)\boxtimes_A ( g(b)\boxtimes_A g(c)))\\
\end{align*}

Using the associativity of $\boxtimes_A$, we see that $(a\boxtimes'_B b)\boxtimes'_B c$ and $a\boxtimes'_B (b\boxtimes'_B c)$ differ by the sign $n+n|a|$. If we had used $\boxtimes_B$ instead of $\boxtimes'_B$, we would pick up two canceling $(-1)^n$ factors, and so have the same associativity defect.

\begin{itemize}
\item Transferring the pairing via degree $n$ chain maps results in a pairing that is only associative up to a sign that can depend on the degrees of the elements involved.
\end{itemize}

We might hope to fix this problem by defining a $Q''$ via $a\boxtimes''_B b=(-1)^{n|a|}a\boxtimes'_B b$, but now this is not a chain map. Using that we know $\boxtimes'_B$ is a degree $n$ chain map, we can compute 
\begin{align*}
Q''(d(a\otimes b))&=Q''((d a)\otimes b+(-1)^{|a|}a\otimes d b)\\
&=(-1)^{n(|a|+1)}(da)\boxtimes'_B b+(-1)^{|a|+n|a|}a\boxtimes'_B db\\
dQ''(a\otimes b)&=(-1)^{n|a|}dQ'(a\otimes_B b)\\
&=(-1)^{n|a|}Q'd(a\otimes_B b)\\
&=(-1)^{n|a|}(da)\boxtimes'_Bb+(-1)^{n|a|+|a|}a\boxtimes'_B db.
\end{align*}
These do not defer by the sign $(-1)^n$ that would be required of a degree $-n$ chain map. 

Turning to commutativity we have 

\begin{align*}
a\boxtimes'_B b&=(-1)^{n|a|}f( g(a)\boxtimes_A g(b))\\
&=(-1)^{n|a|+|g(a)||g(b)|}f( g(b)\boxtimes_A g(a))\\
&=(-1)^{n|a|+(|a|-n)(|b|-n)}f( g(b)\boxtimes_A g(a))\\
&=(-1)^{n|a|+|a||b|+|a|n+|b|n+n}f( g(b)\boxtimes_A g(a))\\
&=(-1)^{|a||b|+n+|b|n}f( g(b)\boxtimes_A g(a))\\
&=(-1)^{|a||b|+n}b\boxtimes'_B a.
\end{align*}
Working with $\boxtimes_B$ instead would multiply both sides by $(-1)^n$, not affecting the sign by which commutativity fails. 

\begin{itemize}
\item Transferring the pairing via degree $n$ chain maps results in a pairing that is only graded commutative  up to the  sign $(-1)^n$.
\end{itemize}

Lastly, we consider the unital property. Let $1\in A^0$ be the unit; note that the unit must have degree $0$. Let $f(1)=u$, which has degree $n$. Then  $g(u)=1$, and we have 

\begin{align*}
u\boxtimes'_B b=(-1)^{n}f( g(u)\boxtimes_A g(b))=(-1)^{n}f( 1\boxtimes_A g(b))=(-1)^{n}f(g(b))=(-1)^n b\\
a\boxtimes'_B u=(-1)^{n|a|}f( g(a)\boxtimes_A g(u))=(-1)^{n|a|}f( g(a)\boxtimes_A 1)=(-1)^{n|a|}f( g(a))=(-1)^{n|a|}a
\end{align*}

If we instead use $\boxtimes_B$, then $u$ becomes a left unit on the nose but only a right unit up to $(-1)^{n+n|a|}$. 

\begin{itemize}
\item If we transfer the pairing via degree $n$ chain maps, the image of the unit of $A^*$ is only a unit of $B^*$ up to signs and not necessarily the same sign from each side. 
\end{itemize}

So, in summary, if we transfer an associative, commutative, unital pairing via a degree $n$ chain isomorphism, the resulting pairing (of degree $-n$) will only be associative, commutative, and unital up to sign discrepancies.

There are two standard ways to correct this defect in practice:

\paragraph{The classical approach.} The first way to fix the sign problems is to discard the requirement that our diagrams commute via chain maps. This is essentially the solution employed tacitly by Dold in his definition of the intersection product \cite[Section VIII.13]{Dold}. More specifically, Dold's intersection product utilizes a transfer map, and he acknowledges in \cite[Exercise VIII.10.14.4]{Dold} that he has not chosen the signs in his definition of transfer maps to be consistent with the Koszul conventions for a chain map. Looking at Equation (13.5) in \cite[Section VIII.13]{Dold} and borrowing the notation $\bullet$ for our modified product on $B^*$, we find a formula that, with our conventions here, could be translated as simply\footnote{This is still not the perfect analogy for what we're doing here, as Dold's duality maps, given by capping with the fundamental class, are also not chain maps.} 
$$a\bullet b=f(g(a)\boxtimes_A g(b)).$$ 

With this modified definition, we have the following good properties:

\begin{align*}
(a\bullet b)\bullet c&=f(g(a\bullet b)\boxtimes_A g(c))\\
&=f(g(f(g(a)\boxtimes_A g(b)))\boxtimes_A g(c))\\
&=f((g(a)\boxtimes_A g(b))\boxtimes_A g(c))\\
&=f(g(a)\boxtimes_A (g(b)\boxtimes_A g(c)))\\
&=f(g(a)\boxtimes_A g(f(g(b)\boxtimes_A g(c))))\\
&=a\bullet (b\bullet c)\\
a\bullet b&=f(g(a)\boxtimes_A g(b))=(-1)^{(|a|-n)(|b|-n)} f(g(b)\boxtimes_A g(a))=(-1)^{(|a|-n)(|b|-n)} b\bullet a.\\ 
u \bullet b&=f(g(u)\boxtimes_A g(b))=f(1\boxtimes_A g(b))=f(g(b))=b\\
a \bullet u&=f(g(a)\boxtimes_A g(u))=f(g(a)\boxtimes_A 1)=f(g(a))=a
\end{align*}

These formulas are consistent with the behavior of the Dold intersection product in \cite[Section VIII.13]{Dold}. 

\begin{itemize}
\item If we transfer the pairing via the formula $a\bullet b=f(g(a)\boxtimes_A g(b))$ then  $\bullet$ is not a chain map, but it is associative and unital. It is commutative up to a sign of $(-1)^{(|a|-n)(|b|-n)}$, as opposed to the sign $(-1)^{|a||b|}$ we would expect for graded commutativity.
\end{itemize}

So to get a well-behaved intersection product in this sense, we can either dispense with chain maps or use chain maps to define the product initially, say at the chain level, but then add some signs to get better behavior after passing to homology. 

\paragraph{The shifty approach.}
There is another way toward better signs that is taken in \cite{McC, GBF18}, which is to shift the complex $B^*$. Using the conventions of \cite[Appendix A.3]{GBF35}, let $\mf s^{n}: B^*[n]\to B^*$ be the shift map, which is a degree $n$ chain map. Given an element $x\in B^{i+n}$, we let $\bar x$ denote the corresponding (identical) element in $(B[n])^i$ so that $\mf s^n(\bar x)=x$. For simplicity, we write $(\mf s^{n})^{-1}=\mf t^n$. For simplicity, we write $(\mf s^{n})^{-1}=\mf t^n$. 

Now consider the diagram

\begin{diagram}
A^*\otimes A^*&\rTo^P &A^*\\
\uTo^{g\otimes g}&&\dTo_f\\
B^*\otimes B^*&\rDashto &B^*\\
\uTo^{\mf s^n\otimes \mf s^n}&&\dTo_{\mf t^n}\\
B^*[n]\otimes B^*[n]&\rDashto^{R'}&B^*[n].
\end{diagram}

The compositions on the left and right of the diagram are now degree $0$ chain maps. Suppose we define $R'$ so that the diagram commutes:
\begin{align*}
R'(\bar a\otimes \bar b)&=\mf t^n f P(g\otimes g)(\mf s^n\otimes \mf s^n)(\bar a\otimes \bar b)\\
&=(-1)^{n|\bar a|}\mf t^n f P(g\otimes g)(a\otimes b)\\
&=(-1)^{n|\bar a|+n|a|}\mf t^n f (g(a)\boxplus_A g(b))\\
&=(-1)^{n(|a|+n)+n|a|}\mf t^n f (g(a)\boxplus_A g(b))\\
&=(-1)^{n}\mf t^n f (g(a)\boxplus_A g(b)).
\end{align*}
If we don't like that $(-1)^n$ we know that we can get rid of it by using $(f\otimes f)^{-1}$ instead of $g\otimes g$:

\begin{diagram}
A^*\otimes A^*&\rTo^P &A^*\\
\dTo^{f\otimes f}&&\dTo_f\\
B^*\otimes B^*&\rDashto &B^*\\
\uTo^{\mf s^n\otimes \mf s^n}&&\dTo_{\mf t^n}\\
B^*[n]\otimes B^*[n]&\rDashto^R&B^*[n].
\end{diagram}

Using this diagram, we can define $R$ instead by 
\begin{align*}
R(\bar a\otimes \bar b)&=\mf t^n f P(f\otimes f)^{-1}(\mf s^n\otimes \mf s^n)(\bar a\otimes \bar b)\\
&=(-1)^{n+n|\bar a|}\mf t^n f P(g\otimes g)(\mf s^n(\bar a)\otimes \mf s^n(\bar b))\\
&=(-1)^{n+n|\bar a|+n(|\bar a|+n)}\mf t^n f P(g\mf s^n(\bar a)\otimes g\mf s^n(\bar b))\\
&=\mf t^n f (g\mf s^n(\bar a)\boxplus_A g\mf s^n(\bar b))\\
&=\mf t^n f (g(a) \boxplus_A g(b)).
\end{align*}

Let's check the properties now, letting $\boxtimes_B^n$ denote the product $R$. We also let $\upsilon=\mf t^nf (1)$.

\begin{align*}
(\bar a\boxtimes_B^n \bar b) \boxtimes_B^n \bar c&=\mf t^n f ( g\mf s^n( \bar a\boxtimes_B^n \bar b  )  \boxplus_A g(\mf s^n\bar c))\\
&=\mf t^n f ( g\mf s^n( \mf t^n f (g\mf s^n(\bar a)\boxplus_A g\mf s^n(\bar b))  )  \boxplus_A g(\mf s^n\bar c))\\
&=\mf t^n f ( (g\mf s^n(\bar a)\boxplus_A g\mf s^n(\bar b))    \boxplus_A g(\mf s^n\bar c))\\
&=\mf t^n f ( g\mf s^n(\bar a)\boxplus_A (g\mf s^n(\bar b)   \boxplus_A g(\mf s^n\bar c)))\\
&=\mf t^n f (g\mf s^n(\bar a)\boxplus_A g\mf s^n(\mf t^n f (g\mf s^n(\bar b)\boxplus_A g\mf s^n(\bar c))))\\
&=\mf t^n f (g\mf s^n(\bar a)\boxplus_A g\mf s^n(\bar b\boxtimes_B^n \bar c))\\
&=\bar a\boxtimes_B^n (\bar b\boxtimes_B^n \bar c)\\
\bar a\boxtimes_B^n \bar b&=\mf t^n f (g\mf s^n(\bar a)\boxplus_A g\mf s^n(\bar b))\\
&=(-1)^{|g\mf s^n(\bar a)||g\mf s^n(\bar b)|}\mf t^n f (g\mf s^n(\bar b)\boxplus_A g\mf s^n(\bar a))\\
&=(-1)^{|\bar a||\bar b|}\bar b\boxtimes_B^n \bar a\\
\upsilon \boxtimes_B^n \bar b&=\mf t^n f (g\mf s^n(\upsilon)\boxplus_A g\mf s^n(\bar b))\\
&=\mf t^n f (g\mf s^n(\mf t^nf (1))\boxplus_A g\mf s^n(\bar b))\\
&=\mf t^n f (1\boxplus_A g\mf s^n(\bar b))\\
&=\mf t^n f (g\mf s^n(\bar b))\\
&=\bar b\\
\bar a  \boxtimes_B^n \upsilon&=\mf t^n f (g\mf s^n(\bar a)\boxplus_A g\mf s^n(\upsilon))\\
&=\mf t^n f (g\mf s^n(\bar a)\boxplus_A g\mf s^n(\mf t^nf (1)))\\
&=\mf t^n f (g\mf s^n(\bar a)\boxplus_A 1)\\
&=\mf t^n f (g\mf s^n(\bar a))\\
&=\bar a.
\end{align*}

So we here recover a unital, graded commutative, associative product, but on $B^*[n]$ instead of $B^*$. 
In this case the product on $B^*[n]$ is induced by a degree $0$ chain map. So in some sense this is the ideal, except that in practice it involves the application of some confusing shifts, which become even more unexpected if we're dealing with the classical singular chain complex (who wants to think about $S_*(X)[n]$?).

\paragraph{Conclusion.}
While not being able to avoid signs (without having to put up with shifts) is somewhat discouraging, the bright side is that we have seen that there are several options for transferring pairings. If we want products induced by chain maps but don't plan to focus on algebraic properties, then $\boxtimes_B$ and $\boxtimes_B'$ are perfectly acceptable and make the nicest diagrams. This is the setting of \cite{GBF39}. If we want better algebraic properties \emph{and} chain maps, we need to be willing to employ shifts; this is a useful perspective for \cite{McC, GBF18}. If we ultimately just want to use a product with nice algebraic properties but are not so concerned with chain maps, we can take the classical approach as in Dold. Hopefully what we've provided in this appendix is something of a dictionary explaining the interactions among these approaches, along with what to expect and not to expect of the sign behavior. In other words, we've seen the properties to which we must re-sign ourselves.

\bibliographystyle{amsplain}
\providecommand{\bysame}{\leavevmode\hbox to3em{\hrulefill}\thinspace}
\providecommand{\MR}{\relax\ifhmode\unskip\space\fi MR }
\providecommand{\MRhref}[2]{%
  \href{http://www.ams.org/mathscinet-getitem?mr=#1}{#2}
}
\providecommand{\href}[2]{#2}

\end{document}